\documentclass[review]{elsarticle}
\usepackage{srcltx}
\usepackage{float}
\usepackage{pgf,tikz}
\usepackage{tikz-cd}
\usepackage{subcaption}
\usetikzlibrary{positioning}
\usetikzlibrary{arrows,patterns}
\usetikzlibrary{decorations.pathreplacing}
\usepackage{eurosym}
\usepackage{mathtools}
\usepackage{amsmath}
\usepackage{amsfonts}
\usepackage{amssymb}
\usepackage{amsthm}
\usepackage{graphicx}
\usepackage{mathrsfs}
\usepackage{xcolor}
\usepackage{exscale}
\usepackage{latexsym}
\usepackage{esvect}
\usepackage[font=small,skip=-10pt]{caption}
\usepackage[overload]{empheq}
\definecolor{airforceblue}{rgb}{0.36, 0.54, 0.66}
\definecolor{auburn}{rgb}{0.43, 0.21, 0.1}
\definecolor{alizarin}{rgb}{0.82, 0.1, 0.26}
\definecolor{Plum}{HTML}{89b02e}

\definecolor{Violet}{HTML}{58429B}

\definecolor{OliveGreen}{HTML}{0d8795}

\usepackage[colorlinks,plainpages=true,pdfpagelabels,hypertexnames=true,colorlinks=true,pdfstartview=FitV,linkcolor=blue,citecolor=red,urlcolor=black,backref=page]{hyperref}
\PassOptionsToPackage{unicode}{hyperref}
\PassOptionsToPackage{naturalnames}{hyperref}
\usepackage{cancel}
\usepackage{enumerate}
\usepackage[shortlabels]{enumitem}
\usepackage{bookmark}
\usepackage{wasysym}
\usepackage{esint}
\usepackage[titletoc,toc,page]{appendix}
\usepackage[ddmmyyyy]{datetime}
\usepackage{setspace}
\usepackage[margin=2.2cm]{geometry}

\usepackage[capitalize,nameinlink]{cleveref}
\crefname{section}{Section}{Sections}
\crefname{subsection}{Subsection}{Subsections}
\crefname{condition}{Condition}{Conditions}
\crefname{hypothesis}{Hypothesis}{Hypothesis}
\crefname{assumption}{Assumption}{Assumptions}
\crefname{lemma}{Lemma}{Lemmas}
\crefname{claim}{Claim}{Claims}
\crefname{remark}{Remark}{Remarks}
\crefname{figure}{Figure}{Figures}
\newtheorem{theorem}{Theorem}[section]
\newtheorem{lemma}[theorem]{Lemma}

\newtheorem{claim}[theorem]{Claim}
\newtheorem{proposition}[theorem]{Proposition}

\newtheorem{definition}[theorem]{Definition}% Use {\rm ...}
\newtheorem{remark}[theorem]{Remark}        % Use {\rm ...}
\newtheorem{assumption}[theorem]{Assumption} 
\newcommand{\tlcref}[1]{\textup{\labelcref{#1}}}
\crefformat{equation}{\textup{#2(#1)#3}}
\crefrangeformat{equation}{\textup{#3(#1)#4--#5(#2)#6}}
\crefmultiformat{equation}{\textup{#2(#1)#3}}{ and \textup{#2(#1)#3}}
{, \textup{#2(#1)#3}}{, and \textup{#2(#1)#3}}
\crefrangemultiformat{equation}{\textup{#3(#1)#4--#5(#2)#6}}%
{ and \textup{#3(#1)#4--#5(#2)#6}}{, \textup{#3(#1)#4--#5(#2)#6}}%
{, and \textup{#3(#1)#4--#5(#2)#6}}

\Crefformat{equation}{#2Equation~\textup{(#1)}#3}
\Crefrangeformat{equation}{Equations~\textup{#3(#1)#4--#5(#2)#6}}
\Crefmultiformat{equation}{Equations~\textup{#2(#1)#3}}{ and \textup{#2(#1)#3}}
{, \textup{#2(#1)#3}}{, and \textup{#2(#1)#3}}
\Crefrangemultiformat{equation}{Equations~\textup{#3(#1)#4--#5(#2)#6}}%
{ and \textup{#3(#1)#4--#5(#2)#6}}{, \textup{#3(#1)#4--#5(#2)#6}}%
{, and \textup{#3(#1)#4--#5(#2)#6}}

\crefdefaultlabelformat{#2\textup{#1}#3}

\def\YYint#1#2#3{{\setbox0=\hbox{$#1{#2#3}{\iint}$}
		\vcenter{\hbox{$#2#3$}}\kern-.50\wd0}}
\def\XXint#1#2#3{{\setbox0=\hbox{$#1{#2#3}{\int}$}
		\vcenter{\hbox{$#2#3$}}\kern-.50\wd0}}

\makeatletter
\def\namedlabel#1#2{\begingroup
	\def\@currentlabel{#2}%
	\label{#1}\endgroup
}
\makeatother
\makeatletter
\newcommand{\rmh}[1]{\mathpalette{\raisem@th{#1}}}
\newcommand{\raisem@th}[3]{\hspace*{-1pt}\raisebox{#1}{$#2#3$}}
\makeatother

\newcommand{\lsb}[2]{#1_{\rmh{-3pt}{#2}}}

\newcommand{\redref}[2]{\texorpdfstring{\protect\hyperlink{#1}{\textcolor{black}{(}\textcolor{red}{#2}\textcolor{black}{)}}}{}}
\newcommand{\redlabel}[2]{\hypertarget{#1}{\textcolor{black}{(}\textcolor{red}{#2}\textcolor{black}{)}}}

\newcommand{\descref}[2]{\hyperref[#1]{\textup{\textcolor{black}{(}\textcolor{blue}{\bf #2}\textcolor{black}{)}}}}

\newcommand{\descitemnormal}[2]{\item[#1]\label{#2}}
\newcommand{\descrefnormal}[2]{\hyperref[#1]{\textup{\textcolor{blue}{\bf #2}}}}

\newcommand{\overlabel}[2]{\overset{\text{\cref{#1}}}{#2}}
\newcommand{\overred}[3]{\overset{\redlabel{#1}{#2}}{#3}}

\makeatletter
\g@addto@macro\normalsize{%
	\setlength\abovedisplayskip{3pt}
	\setlength\belowdisplayskip{3pt}
	\setlength\abovedisplayshortskip{1pt}
	\setlength\belowdisplayshortskip{3pt}
}
\makeatother
\makeatletter
\def\ps@pprintTitle{%
	\let\@oddhead\@empty
	\let\@evenhead\@empty
	\def\@oddfoot{}%
	\let\@evenfoot\@oddfoot}
\makeatother
\numberwithin{equation}{section}
\everymath{\displaystyle}
\allowdisplaybreaks
\newcommand{\mfd}{\mathfrak{d}}

\newcommand{\mfg}{\mathfrak{g}}

\newcommand{\mfq}{\mathfrak{q}}

\newcommand{\mfs}{\mathfrak{s}}
\newcommand{\mft}{\mathfrak{t}}

\newcommand{\mcq}{\mathcal{q}}

\newcommand{\mcB}{\mathcal{B}}

\newcommand{\mcI}{\mathcal{I}}

\newcommand{\mcQ}{\mathcal{Q}}

\newcommand{\mbC}{\mathbb{C}}

\newcommand{\mbfa}{\mathbf{a}}

\DeclareMathOperator{\mopt}{{t}}

\DeclareMathOperator{\mopC}{{C}}

\newcommand{\al}{\alpha}
\newcommand{\be}{\beta}

\newcommand{\de}{\delta}
\newcommand{\ve}{\varepsilon}

\newcommand{\tht}{\theta}

\newcommand{\ep}{\epsilon}
\newcommand{\pa}{\partial}

\newcommand{\om}{\omega}

\newcommand{\Om}{\Omega}

\newcommand{\La}{\Lambda}
\DeclareMathOperator*{\I}{\mathbf{I}}
\DeclareMathOperator*{\II}{\mathbf{II}}

\newcommand\RR{\mathbb{R}}

\DeclareMathOperator{\dv}{div}
\DeclareMathOperator{\spt}{spt}

\DeclareMathOperator{\loc}{loc}

\DeclareMathOperator*{\esssup}{ess\,sup}
\DeclareMathOperator*{\essinf}{ess\,inf}
\DeclareMathOperator*{\essosc}{ess\,osc}

\newcommand{\norm}[1]{\left|\hspace{-0.2mm}\left| #1\right|\hspace{-0.2mm}\right|}
\newcommand{\abs}[1]{\left| #1\right|}
\newcommand{\lbr}[1][(]{\left#1}
\newcommand{\rbr}[1][)]{\right#1}

\newcommand{\txt}[1]{\qquad \text{#1} \quad}
\newenvironment{notationlist}
{\begin{enumerate}[labelindent=*,
		leftmargin=*,
		label=(\bf{N}{\arabic*})
		]}
	{\end{enumerate}}
\usepackage[titletoc]{appendix}
\DeclareMathOperator*{\tail}{{\textup{Tail}_\infty}}
\newcommand{\tailp}[1][p-1]{\textup{Tail}_\infty^{#1}}

\newcommand{\rint}{\mathring{\mopt}}

\newcommand{\mreta}{\mathring\eta}

\DeclareMathOperator{\bsmu}{\boldsymbol\mu}
\DeclareMathOperator{\bsom}{\boldsymbol\omega}
\DeclareMathOperator{\bsy}{\boldsymbol Y}
\DeclareMathOperator{\bsc}{\boldsymbol C}

\newcommand{\data}[1]{\{n,p,q,s,\Lambda#1\}}
\newcommand{\datanb}[1]{n,p,q,s,\Lambda#1}

\begin{document}
	
	\begin{frontmatter}
		\title{H\"older regularity of doubly nonlinear nonlocal quasilinear parabolic equations in some mixed singular - degenerate regime}
		\author{Karthik Adimurthi\tnoteref{thanksfirstauthor}}
		\ead{karthikaditi@gmail.com and kadimurthi@tifrbng.res.in}
		\tnotetext[thanksfirstauthor]{Supported by the Department of Atomic Energy,  Government of India, under
			project no.  12-R\&D-TFR-5.01-0520}
		%		 and  SERB grant SRG/2020/000081}
		\author{Mitesh Modasiya\tnoteref{thanksfirstauthor}}
		\ead{mit.modasiya@gmail.com}
		
		\address{Tata Institute of Fundamental Research, Centre for Applicable Mathematics,Bangalore, Karnataka, 560065, India}
		
		\begin{abstract}
			We study local H\"older regularity of bounded, weak solutions for the nonlocal quasilinear equations of the form
			\[
		(|u|^{q-2}u)_t + \text{P.V.} \int_{\mathbb{R}^n} \frac{|u(x,t) - u(y,t)|^{p-2}(u(x,t)-u(y,t))}{|x-y|^{n+sp}} dy= 0,
			\]
			with $p\in (1,\infty)$, $q\in (1,\infty)$ and $s \in (0,1)$.   
			Analogous H\"older continuity result in the local case is known  in the purely singular case $\{1<p<2, p<q\}$, purely degenerate case $\{2<p, q<p\}$, scale invariant case $\{p=q\}$ and translation invariant case $\{q=2,1<p<\infty\}$. In the nonlocal setting, H\"older regularity is known when the equation is either translation invariant $\{q=2, 1<p<\infty\}$ or  scale invariant $\{q=p, 1<p<\infty\}$ or purely degenerate case $\{2<p, q<p\}$. Similar strategy can be used to obtain H\"older regularity in the purely singular case $\{1<p<2, p<q\}$.

			In this paper, we adapt several ideas developed over the past few years and combine it with a new intrinsic scaling to prove H\"older regularity in the mixed singular - degenerate range $\max\{p,q,2\} <  \min\left\{q + \tfrac{p-1}{1+\frac{n}{sp}}, 2 + \tfrac{p-1}{1+\frac{n}{sp}}\right\}$. The proof explicitly makes use of the nonlocal nature of the problem and as a consequence, our estimates are not stable at $s \rightarrow 0$.  We note that the analogous regularity in the local problem remains open.
%			  and the H\"older regularity theory is not very well understood. 
%			In this paper, we  adapt the recently developed ideas to obtain a unified proof that works in the full range $s \in (0,1)$, $p \in (1,\infty)$ and $q\in (1,\infty)$. 
%			

%			In the local case with $\{q=2,1<p<\infty\}$, a unified proof which does not distinguish the singular or degenerate regimes  was conjectured to hold  by  G.M.Lieberman in the 80's and still remains a challenging open problem.  In this paper, we partially answer this questions  in the nonlocal setting by obtaining a unified proof for $\{q=2,  2<p<\infty\}$ and $\{q=2,1<p_o<p<2+\ep_{p_o}\}$ for any fixed $p_o \in (1,2)$ and explicitly quantified $\ep_{p_o}(n,s,p_o)$. Subsequently, we extend these ideas to prove H\"older regularity of bounded, weak solutions for $\{2 \leq q<p<\infty\}$ and $\{2 \leq q\leq 2 + \ep_{p_o}, 1<p_o<p<2+\ep_{p_o}\}$ for any fixed $p_o \in (1,2)$.
%			% all $p\in(1,\infty)$ and $q \in (1,\infty)$. 
%			We note that this is in contrast to the local case where only partial results are available.  
		\end{abstract}
		
	\end{frontmatter}
	%\begin{singlespace}
	\tableofcontents
	%\end{singlespace}
	
\section{Introduction}

In this article, we prove local H\"older regularity for bounded, sign-changing solutions of nonlocal parabolic equations whose prototype structure is  of the form
\begin{equation}\label{maineq}
	\partial_t (|u|^{q-2}u) + \text{P.V.}\int_{\RR^n} |u(x,t)-u(y,t)|^{p-2}(u(x,t)-u(y,t))K(x,y,t)\,dy=0,
\end{equation} with $p\in (1,\infty)$, $q \in (1,\infty)$ and $s \in (0,1)$. Moreover, for some universally fixed constant  $\La \geq 1$ and for almost every $x,y \in \RR^n$, we take $K:\RR^n\times\RR^n\times \RR \to [0,\infty)$ to be a symmetric measurable function satisfying
\begin{equation*}%\label{boundsonKernel}
	\frac{(1-s)}{\Lambda|x-y|^{n+sp}}\leq K(x,y,t)\leq \frac{(1-s)\Lambda}{|x-y|^{n+sp}}. 
\end{equation*}

First let us collect all the main theorems that will be proved in this paper:
%\hrule \hrule \hrule \hrule \hrule 
\subsection{Main theorems}

Following the proof of the degenerate case from \cite{liaoHolderRegularityParabolic2024} combined with our unified scaling, we obtain the following theorem:
%The first main theorem we prove is a unified proof of H\"older regularity when $q=2$ and $\tfrac32<p<\infty$:
\begin{theorem}\label{holderparabolicqtwo}
	Let $p\in[2,\infty)$, $q =2$, $s\in(0,1)$ and let $u\in L^p(I;W^{s,p}_{\text{loc}}(\Om))\cap L^\infty(I;L^q_{\text{loc}}(\Om))\cap L^\infty(I;L^{p-1}_{sp}(\RR^n))$ be any locally bounded, sign-changing weak solution to \cref{maineq}. Then $u$ is locally H\"older continuous in $\Om_T$, i.e., there exist constants $j_o > 1$, $\mathbf{C}_o>1$ and $\alpha\in (0,1)$ depending only on the data and $\mfd$, such that the following holds:
	\[
	|u(0,0) - u(x,t)| \leq C_{\data{,\mfd}} \frac{L}{R^{\alpha}} \lbr \max\{L^{\frac{2-\mfd}{sp}}|x - 0|, L^{\frac{p-\mfd}{sp}}|t - 0|^{\frac{1}{sp}}\}\rbr^{\alpha},
	\]
	for any $(x,t) \in B_{\frac12 R}(0)\times (-(\tfrac12 R)^{sp},0)$, the exponent $\mfd$ is any exponent satisfying $0<\mfd \leq \min\{2,p\}$ with $p-3+\mfd >0$ and 
	\[
	L:= 2 \mathbf{C}_o^{\frac{sp}{p-1}j_o} \|u\|_{L^{\infty}(B_{8R}\times (-(8R)^{sp},0))} + \tail(|u|,8R,0,(-(8R)^{sp},0)).
	\]
	Here $R$ is any fixed radius and we assume $B_{8R}\times (-(8R)^{sp},0)\subset \Omega_T$ and $(0,0) \in \Omega_T$ is any fixed point.
\end{theorem}

Following the proof of the singular case from \cite{liaoHolderRegularityParabolic2024} combined with our unified scaling, we obtain the following theorem:
\begin{theorem}\label{holderparabolicqtwosing}
	Let $p_o\in(1,2)$  be given,  $q =2$, $s\in(0,1)$ and for $p \in (p_o,2+\epsilon_{p_o})$ with $\epsilon_{p_o}$ as given in \cref{defeppo}, let $u\in L^p(I;W^{s,p}_{\text{loc}}(\Om))\cap L^\infty(I;L^q_{\text{loc}}(\Om))\cap L^\infty(I;L^{p-1}_{sp}(\RR^n))$ be any locally bounded, sign-changing weak solution to \cref{maineq}. Furthermore, define \begin{equation}\label{defeppo}\epsilon_{p_o} := \min\left\{p_o-1, \tfrac{sp_o(p_o-1)}{n}, \tfrac{sp_o(p_o-1)}{n+3s}\right\},\end{equation} then $u$ is locally H\"older continuous in the range $p \in (p_o,2+\epsilon_{p_o})$, i.e., there exist constants $j_o > 1$, $\mathbf{C}_o>1$ and $\alpha\in (0,1)$ depending only on the data and $p_o$, such that the following holds:
	\[
	|u(0,0) - u(x,t)| \leq C_{\data{,\mfd}} \frac{L}{R^{\alpha}} \lbr \max\{L^{\frac{2-\mfd}{sp}}|x - 0|, L^{\frac{p-\mfd}{sp}}|t - 0|^{\frac{1}{sp}}\}\rbr^{\alpha},
	\]
	for any $(x,t) \in B_{\frac12 R}(0)\times (-(\tfrac12 R)^{sp},0)$, the exponent $\mfd:=2+\epsilon_{p_o}$ and 
	\[
	L:= 2 \mathbf{C}_o^{\frac{sp}{p-1}j_o} \|u\|_{L^{\infty}(B_{8R}\times (-(8R)^{sp},0))} + \tail(|u|,8R,0,(-(8R)^{sp},0)).
	\]
	Here $R$ is any fixed radius and we assume $B_{8R}\times (-(8R)^{sp},0)\subset \Omega_T$ and $(0,0) \in \Omega_T$ is any fixed point.
\end{theorem}

In the general setting of \cref{maineq}, the main theorem we prove in this paper is the following:

\begin{theorem}\label{holderparabolic}
	Let $p,q\in(1,\infty)$, $s\in(0,1)$ be given satisfying $\max\{p,q\} <  q + \tfrac{p-1}{1+\frac{n}{sp}}$ and $\max\{p,q,2\} <  2 + \tfrac{p-1}{1+\frac{n}{sp}}$. Furthermore let $\mfd$ be chosen to satisfy 
	\[
	\max\{p,q,2\} < \mfd < \min\left\{q + \tfrac{p-1}{1+\frac{n}{sp}}, 2 + \tfrac{p-1}{1+\frac{n}{sp}}\right\}.
	\]
	Let $u\in L^p(I;W^{s,p}_{\text{loc}}(\Om))\cap L^\infty(I;L^q_{\text{loc}}(\Om))\cap L^\infty(I;L^{p-1}_{sp}(\RR^n))$ be any locally bounded, sign-changing weak solution to \cref{maineq},  then $u$ is locally H\"older continuous in $\Om_T$, i.e., there exist constants $j_o > 1$, $\mathbf{C}_o>1$ and $\alpha\in (0,1)$ depending only on $\data{}$, such that the following holds:
	\[
	|u(0,0) - u(x,t)| \leq C_{\data{}} \frac{L}{R^{\alpha}} \lbr \max\{L^{\frac{q-\mfd}{sp}}|x - 0|, L^{\frac{p-\mfd}{sp}}|t - 0|^{\frac{1}{sp}}\}\rbr^{\alpha},
	\]
	for any $(x,t) \in B_{\frac12 R}(0)\times (-(\tfrac12 R)^{sp},0)$ and
	\[
	L:= 2 \mathbf{C}_o^{\frac{sp}{p-1}j_o} \|u\|_{L^{\infty}(B_{8R}\times (-(8R)^{sp},0))} + \tail(|u|,8R,0,(-(8R)^{sp},0)).
	\]
	Here $R$ is any fixed radius and we assume $B_{8R}(0)\times (-(8R)^{sp},0)\subset \Omega_T$ and $(0,0) \in \Omega_T$ is any fixed point.
\end{theorem}

%\begin{remark}
%	The H\"older regularity in \cref{holderparabolic} actually holds in the range $p \in (1,2+\ve_o)$ and $q\in (1,2+\ve_o)$ where $\ve_o = \min \left\{1, \tfrac{2s}{n}, \tfrac{2s}{n+3s}\right\}$. This is because we can first choose $p_o=2$ in \cref{holderparabolic} and this would give H\"older regularity in the range $q \in (1,2+\ve_o)$ and $p \in (2, 2+\ve_o)$. Now we can keep changing the choice of $p_o \in (1,2)$ to get H\"older regularity in the range $p \in (p_o,2+\ve_{p_o})$  and $q \in (1,2+\ve_{p_o})$. This range is pictorially represented in \cref{figrangenonlocalmixed}. 
%\end{remark}
%\hrule \hrule \hrule \hrule \hrule
%Our proof follows a combination of the ideas developed in \cite{adimurthiLocalHolderRegularity2022,bogeleinHolderRegularitySigned2021,liaoUnifiedApproachHolder2020,liaoHolderRegularityParabolic2022}.

\subsection{On historical development of intrinsic scaling}

The method of intrinsic scaling was developed by E.DiBenedetto when $p \geq 2$ in \cite{dibenedettoLocalBehaviourSolutions1986} to prove H\"older regularity for degenerate quasilinear parabolic equations of the form
\[
u_t - \dv |\nabla u|^{p-2} \nabla u = 0.
\]  A technical requirement of the proof was a novel logarithmic estimate which aids in the expansion of positivity.  Subsequently, the proof of H\"older regularity for the singular case  $p < 2$ was given in \cite{ya-zheLocalBehaviorSolutions1988} by switching the scaling from time variable to the space variable. These results are collected in E.DiBenedetto's treatise \cite{dibenedettoDegenerateParabolicEquations1993}. After a gap of several years, E.DiBenedetto, U.Gianazza and V.Vespri \cite{dibenedettoHarnackEstimatesQuasilinear2008, dibenedettoHarnackInequalityDegenerate2012} were able to develop fundamentally new ideas involving an exponential change of variable in time to prove Harnack's inequality for the parabolic $p$-Laplace equations. Crucially, this proof relies on expansion of positivity estimates and does not involve logarithmic test functions. Then, a new proof of H\"older regularity was given with a more geometric flavour in \cite{gianazzaNewProofHolder2010}. This theory was extended to generalized parabolic $p$-Laplace equations with Orlicz growth in \cite{hwangHolderContinuityBounded2015,hwangHolderContinuityBounded2015a}, where they were able to adapt the techniques used for degenerate equations and prove H\"older regularity for the singular equation. Subsequently, using the exponential change of variables, a new proof of H\"older regularity was given in \cite{liaoUnifiedApproachHolder2020}.

\subsection{A brief history of the problem} 

Much of the early work on regularity of fractional elliptic equations in the case $p=2$ was carried out by Silvestre \cite{silvestreHolderEstimatesSolutions2006}, Caffarelli and Vasseur \cite{caffarelliDriftDiffusionEquations2010}, Caffarelli, Chan, Vasseur \cite{caffarelliRegularityTheoryParabolic2011} and Bass-Kassmann  \cite{bassHarnackInequalitiesNonlocal2005,bassHolderContinuityHarmonic2005, kassmannPrioriEstimatesIntegrodifferential2009}. An early formulation of the fractional $p$-Laplace operator was done by Ishii and Nakamura \cite{ishiiClassIntegralEquations2010} and existence of viscosity solutions was established. DiCastro, Kuusi and Palatucci extended the De Giorgi-Nash-Moser framework to study the regularity of the fractional $p$-Laplace equation in \cite{dicastroLocalBehaviorFractional2016}. The subsequent work of Cozzi \cite{cozziRegularityResultsHarnack2017} covered a stable (in the limit $s\to 1$) proof of H\"older regularity by defining a novel fractional De Giorgi class. An alternate proof of H\"older regularity based on the ideas of \cite{tilliRemarksHolderContinuity2006} was given in \cite{adimurthiOlderRegularityFractional2022}. Explicit exponents for H\"older regularity with no coefficients were found in \cite{brascoHigherHolderRegularity2018,brascoContinuitySolutionsNonlinear2021} and the range of exponents was improved to the sharp range in \cite{BOGELEIN2025111078}. The full Lipschitz regularity was  recently obtained in the fundamental work of \cite{biswas2025lipschitzregularityfractionalplaplacian}.

 H\"older regularity of bounded, weak solutions for doubly nonlinear equations of the form
\begin{equation}\label{localtrudinger}
	\partial_{t}(|u|^{q-2}u) - \Delta_p u = 0,
\end{equation}
was proved in \cite{bogeleinHolderRegularitySigned2021} when $q =p$,  in \cite{bogeleinHolderRegularitySigned2023} when $2<p$, $q<p$ and in  \cite{liaoHolderRegularitySigned2022} when $1<p<2$ and $p<q$. In the special case of non-negative solutions, analogous H\"older regularity results were proved in the important paper \cite{MR2957656,MR3029403} which forms the framework for all future developments in understanding the regularity for \cref{localtrudinger}. In \cite{adimurthi2025localholderregularitybounded}, bounded, sign changing weak solutions of \cref{maineq} in the case $q=p$ was shown to be H\"older continuous, see \cref{figrange}.

%In the case of more general $p \in (1,\infty)$ and $q \in (1,\infty)$, first progress was made in the fundamental papers \cite{MR2957656,MR3029403}. Subsequently, in a series of three important  papers \cite{bogeleinHolderRegularitySigned2021,bogeleinHolderRegularitySigned2023,liaoHolderRegularitySigned2022}, the theory was extended to obtain H\"older regularity for bounded, sign changing solutions of \cref{maineq} in the ranges $\{q=p\}$, $\{p<2, p<q\}$ and $\{2<p,q<p\}$. This left open the remaining cases, see \cref{figrange} for the details. 

\begin{figure}[ht]
	%			\begin{figure}{.5\textwidth}
	\begin{center}
		\begin{tikzpicture}[line cap=round,line join=round,>=latex,scale=0.5]
			\coordinate  (O) at (0,0);
			\draw[help lines, color=gray!40, dashed] (-5,-5) grid (5,5);
			\draw[->,ultra thick] (-5,-5)--(5,-5) node[right]{$p$};
			\draw[->,ultra thick] (-5,-5)--(-5,5) node[above]{$q$};
			
			\node [below] at (-3,-5) {$1$};
			\node [below] at (-1,-5) {$2$};
			\node [below] at (1,-5) {$3$};
			\node [below] at (3,-5) {$4$};
			
			\node [left] at (-5,-3) {$1$};
			\node [left] at (-5,-1) {$2$};
			\node [left] at (-5,1) {$3$};
			\node [left] at (-5,3) {$4$};

			\node [left] at (1,3) {{\bf ??}};
			\node [left] at (-1,-2.5) {{\bf ??}};

			\draw[draw=black, ultra thick, dotted,->] (-3,-3) -- (5,-3);
			\draw[draw=black, ultra thick, dotted,->] (-3,-3) -- (-3,5);
			
			\draw[draw=teal,ultra thick,->] (-3,-3) -- (5,5);
			
			\draw[draw=orange, ultra thick] (-3,-1) -- (-1,-1);
			\draw[draw=blue, ultra thick,->] (-1,-1) -- (5,-1);
			
			\fill [opacity=0.3,alizarin](-1,-1) -- (5,5) -- (5,-3) -- (-1,-3) -- cycle;
			
			\fill [opacity=0.3,auburn] (-1,5) -- (-1,-1) -- (-3,-3) -- (-3,5) -- cycle;
			
			\node[right,align=left] at (8,3) { 
				{\color{orange}$q=2, 1<p<2$} in \cite{dibenedettoLocalBehaviourSolutions1986}\\
				{\color{blue}$q=2, 2<p<\infty$} in \cite{ya-zheLocalBehaviorSolutions1988}\\
				{\color{teal}$p=q$} in \cite{bogeleinHolderRegularitySigned2021}\\
				{\color{alizarin}$q<p, 2<p<\infty$} in \cite{bogeleinHolderRegularitySigned2023}\\
				{\color{auburn}$p<q, 1<p<2$} in \cite{liaoHolderRegularitySigned2022} };
		\end{tikzpicture}
	\end{center}
	\caption{Ranges of $p$ and $q$ for local doubly nonlinear equations}
	\label{figrange}
\end{figure}
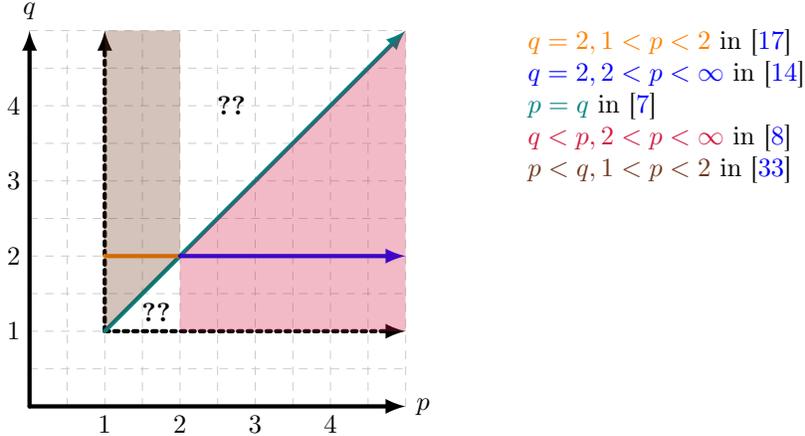

\emph{We note that  \cref{maineq} is translation invariant only when $q=2$, scale invariant when $q=p$ and neither scale invariant nor translation invariant in other cases.  It is well known that failure of either translation invariance or scale invariance leads to fundamental difficulties when trying to obtain H\"older regularity. }

%Our approach to proving H\"older regularity for \cref{maineq} combines the  techniques from \cite{bogeleinHolderRegularitySigned2021} to handle the time terms and \cite{adimurthiLocalHolderRegularity2022,liaoHolderRegularityParabolic2022} to handle the nonlocal term. 

%The main difficulty in studying \cref{maineq} was that the regularity theory for nonlocal $p$-parabolic equations was unknown until recently which played a crucial role when studying \cref{maineq}. As a consequence, the only known regularity results for \cref{maineq} in in existing literature were local semi-continuity, local boundedness  and reverse H\"older inequality for globally bounded, strictly positive weak solutions of \cref{maineq} in the case $p=q$ proved in \cite{banerjeeLocalPropertiesSubsolution2021,banerjeeLowerSemicontinuityPointwise2021}. 

Regarding parabolic nonlocal equations of the form 
\begin{equation*}%\label{nonlocalp}
	\partial_t u + \text{P.V.}\int_{\RR^n} \frac{|u(x,t)-u(y,t)|^{p-2}(u(x,t)-u(y,t))}{|x-y|^{n+sp}}\,dy=0,
\end{equation*}
bounded, weak solutions were shown to be H\"older continuous in \cite{adimurthi2024localholderregularitynonlocal,liaoHolderRegularityParabolic2024}. In the scale invariant case of $p=q$, bounded weak solutions were shown to be H\"older continuous in \cite{adimurthi2025localholderregularitybounded}. The ideas from \cite{adimurthi2025localholderregularitybounded} was extended to the doubly degenerate setting in \cite{li2025holderregularityweaksolutions} (see \cref{figrangenonlocal}) and we note that the same strategy can be applied to obtain H\"older regularity in the doubly singular regime, see \cref{figrangenonlocalsing}. We leave this extension to the doubly singular regime to the interested reader.  
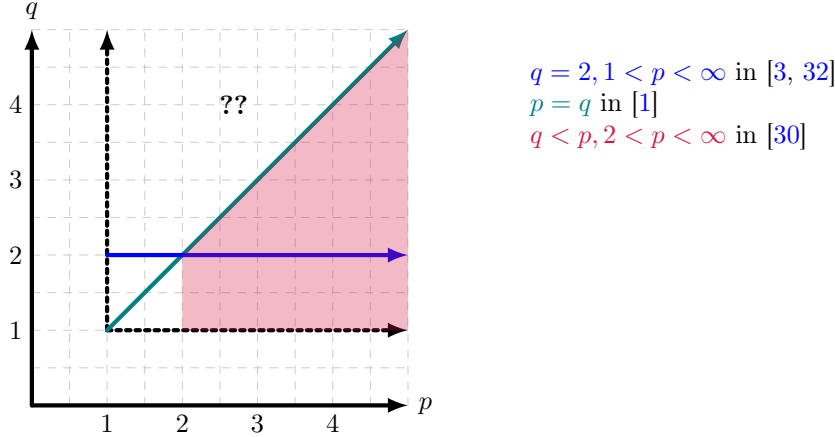
\begin{figure}[ht]
	%			\begin{figure}{.5\textwidth}
	\begin{center}
		\begin{tikzpicture}[line cap=round,line join=round,>=latex,scale=0.5]
			\coordinate  (O) at (0,0);
			\draw[help lines, color=gray!40, dashed] (-5,-5) grid (5,5);
			\draw[->,ultra thick] (-5,-5)--(5,-5) node[right]{$p$};
			\draw[->,ultra thick] (-5,-5)--(-5,5) node[above]{$q$};
			
			\node [below] at (-3,-5) {$1$};
			\node [below] at (-1,-5) {$2$};
			\node [below] at (1,-5) {$3$};
			\node [below] at (3,-5) {$4$};
			
			\node [left] at (-5,-3) {$1$};
			\node [left] at (-5,-1) {$2$};
			\node [left] at (-5,1) {$3$};
			\node [left] at (-5,3) {$4$};

			\node [left] at (1,3) {{\bf ??}};
%			\node [left] at (-1,-2.5) {{\bf ??}};

			\draw[draw=black, ultra thick, dotted,->] (-3,-3) -- (5,-3);
			\draw[draw=black, ultra thick, dotted,->] (-3,-3) -- (-3,5);
			
			\draw[draw=teal,ultra thick,->] (-3,-3) -- (5,5);
			
			\draw[draw=blue, ultra thick] (-3,-1) -- (-1,-1);
			\draw[draw=blue, ultra thick,->] (-1,-1) -- (5,-1);
			
			\fill [opacity=0.3,alizarin](-1,-1) -- (5,5) -- (5,-3) -- (-1,-3) -- cycle;
			
%			\fill [opacity=0.3,auburn] (-1,5) -- (-1,-1) -- (-3,-3) -- (-3,5) -- cycle;
			
%			\fill [opacity=0.3,alizarin] (-3,-3) -- (-1,-1) -- (-1,-3)  -- cycle;
%			\fill [opacity=0.3,alizarin] (-1,-1) -- (0,0) -- (-1,0)  -- cycle;
			
			\node[right,align=left] at (8,3) { 
%				{\color{orange}$q=2, 1<p<2$} in \cite{dibenedettoLocalBehaviourSolutions1986}\\
				{\color{blue}$q=2, 1<p<\infty$} in \cite{adimurthi2024localholderregularitynonlocal,liaoHolderRegularityParabolic2024}\\
				{\color{teal}$p=q$} in \cite{adimurthi2025localholderregularitybounded}\\
				{\color{alizarin}$q<p, 2<p<\infty$} in \cite{li2025holderregularityweaksolutions}};
%				{\color{auburn}$p<q, 1<p<2$} in \cite{liaoHolderRegularitySigned2022} };
		\end{tikzpicture}
	\end{center}
	\caption{Ranges of $p$ and $q$ for nonlocal doubly nonlinear equations}
	\label{figrangenonlocal}
\end{figure}

 In this paper, we show that bounded, weak solutions of \cref{maineq} are H\"older continuous in the range  \cref{figrangenonlocalmixed} where $p_o$ is any arbitrary fixed exponent in $(1,2)$. We note that H\"older regularity for  the range of exponents considered in \cref{figrangenonlocalmixed} for the local problem \cref{localtrudinger} remains a challenging open problem. 

\begin{figure}[ht]
	\begin{minipage}{0.5\textwidth}
	%			\begin{figure}{.5\textwidth}
	%	\caption{Ranges of $p$ and $q$ for nonlocal doubly nonlinear equations}
	\begin{center}
		\begin{tikzpicture}[line cap=round,line join=round,>=latex,scale=0.5]
			\coordinate  (O) at (0,0);
			\draw[help lines, color=gray!40, dashed] (-5,-5) grid (5,5);
			\draw[->,ultra thick] (-5,-5)--(5,-5) node[right]{$p$};
			\draw[->,ultra thick] (-5,-5)--(-5,5) node[above]{$q$};
			
			\node [below] at (-3,-5) {$1$};
			\node [below] at (-1,-5) {$2$};
			\node [below] at (1,-5) {$3$};
			\node [below] at (3,-5) {$4$};
			
			\node [left] at (-5,-3) {$1$};
			\node [left] at (-5,-1) {$2$};
			\node [left] at (-5,1) {$3$};
			\node [left] at (-5,3) {$4$};

			%			\node [left] at (1,3) {{\bf ??}};
			%			\node [left] at (-1,-2.5) {{\bf ??}};

			\draw[draw=black, ultra thick, dotted,->] (-3,-3) -- (5,-3);
			\draw[draw=black, ultra thick, dotted,->] (-3,-3) -- (-3,5);
			
			\draw[draw=teal,ultra thick,->] (-3,-3) -- (5,5);
			
			\draw[draw=blue, ultra thick] (-3,-1) -- (-1,-1);
			\draw[draw=blue, ultra thick,->] (-1,-1) -- (5,-1);
			\fill [opacity=0.3,auburn] (-1,5) -- (-1,-1) -- (-3,-3) -- (-3,5) -- cycle;
			
			%			\fill [opacity=0.3,auburn](-1,-1) -- (5,5) -- (5,-1) -- cycle;
			
			%\fill [opacity=0.3,auburn] (-1,5) -- (-1,-1) -- (-3,-3) -- (-3,5) -- cycle;
			
%			\fill [opacity=0.3,alizarin] (-3,-3) -- (-3,0) -- (0,0) -- (0,-3)  -- cycle;
			%			\fill [opacity=0.3,alizarin] (-1,-1) -- (0,0) -- (-1,0)  -- cycle;
			
			%			\node[right,align=left] at (8,3) { 
			%				{\color{orange}$q=2, 1<p<2$} in \cite{dibenedettoLocalBehaviourSolutions1986}\\
			%				{\color{blue}$q=2, 1<p<\infty$} in \cite{KHV2024, N24}\\
			%				{\color{orange}$p=q$} in \cite{KA2025}};
			%				{\color{alizarin}$q<p, 2<p<\infty$} in \cite{bogeleinHolderRegularitySigned2023}\\
			%				{\color{auburn}$p<q, 1<p<2$} in \cite{liaoHolderRegularitySigned2022} };
		\end{tikzpicture}
	\end{center}
	\caption{Doubly singular ranges of $p$ and $q$}
	\label{figrangenonlocalsing}
	\end{minipage}
\begin{minipage}{0.5\textwidth}
	\begin{center}
	\begin{tikzpicture}[line cap=round,line join=round,>=latex,scale=0.5]
		\coordinate  (O) at (0,0);
		\draw[help lines, color=gray!40, dashed] (-5,-5) grid (5,5);
		\draw[->,ultra thick] (-5,-5)--(5,-5) node[right]{$p$};
		\draw[->,ultra thick] (-5,-5)--(-5,5) node[above]{$q$};
		
		\node [below] at (-3,-5) {$1$};
		\node [below] at (-1,-5) {$2$};
		\node [below] at (1,-5) {$3$};
		\node [below] at (3,-5) {$4$};
		
		\node [left] at (-5,-3) {$1$};
		\node [left] at (-5,-1) {$2$};
		\node [left] at (-5,1) {$3$};
		\node [left] at (-5,3) {$4$};

%		\node [left] at (1,3) {{\bf ??}};
%		\node [left] at (-1,-2.5) {{\bf ??}};
		\fill[opacity=0.3,teal, dashed] (-3,-1) .. controls (-2,-0.8).. (0.5,0.5) -- (0.5,-1)  ..controls (0,-1.1).. (-1,-1.5) ..controls(-2,-1.2).. (-3,-1);
		\draw[opacity=1,draw=teal, dashed] (-3,-1) .. controls (-2,-0.8).. (0.5,0.5) -- (0.5,-1)  ..controls (0,-1.1).. (-1,-1.5) ..controls(-2,-1.2).. (-3,-1);

		\draw[draw=black, ultra thick, dotted,->] (-3,-3) -- (5,-3);
		\draw[draw=black, ultra thick, dotted,->] (-3,-3) -- (-3,5);
		
		\draw[draw=teal,ultra thick,->] (-3,-3) -- (5,5);
		
		\draw[draw=blue, ultra thick] (-3,-1) -- (-1,-1);
		\draw[draw=blue, ultra thick,->] (-1,-1) -- (5,-1);
		
		\fill [opacity=0.3,alizarin](-1,-1) -- (5,5) -- (5,-3) -- (-1,-3) -- cycle;
		
		\fill [opacity=0.3,auburn] (-1,5) -- (-1,-1) -- (-3,-3) -- (-3,5) -- cycle;
	\end{tikzpicture}
	\end{center}
	\caption{Mixed ranges of $p$ and $q$ from \cref{holderparabolic}\\ For illustration purpose only for a fixed choice of $s$ and $n$,  not to scale.}
	\label{figrangenonlocalmixed}
	\end{minipage}
\end{figure}

\subsection{Our approach to proving H\"older regularity to \texorpdfstring{\cref{maineq}}.}

The general strategy to proving H\"older regularity is based on the framework developed in \cite{bogeleinHolderRegularitySigned2021,liaoHolderRegularityParabolic2024,adimurthi2025localholderregularitybounded}, though our proof requires additional difficulties to be overcome. 
\begin{enumerate}[(i)]
	\item First we obtain a new proof of \cref{maineq} in the case $q=2$ using a new scaling, which enables us to partially handle the singular and degenerate cases at one go. In particular, we provide a partially unified proof of H\"older regularity, see the proof of \cref{holderparabolicqtwosing}. This unified scaling is the main new idea and it's adaptation for general $q \neq 2$ forms the basis for the rest of the paper. This represents a purely nonlocal phenomenon and as a consequence, our estimates are unstable as $s \rightarrow 1$. 
	
	\item The second important observation that is used in this paper, is that the isoperimetric inequality in the nonlocal setting is much stronger due to it being derived from the energy estimates. As a consequence, this gives a polynomial decay of the level sets, which is in stark contrast to the local problem where one obtains a much weaker exponential decay. This improved polynomial decay is the crucial ingredient that enables us to obtain H\"older regularity in the mixed regime (see \cref{figrangenonlocalmixed}) which remains a challenging problem for the local equations. 
%	\item The proof is split into two regimes, the first in \cref{fig1} which implements the ideas from \cite{bogeleinHolderRegularitySigned2021,bogeleinHolderRegularitySigned2023} and the second in \cref{fig2} which makes use of the unified scaling developed in this paper.
	
	\item As in \cite{bogeleinHolderRegularitySigned2021}, we first obtain reduction of oscillation in the `close to zero' case. To achieve this, we extend the new unified scaling idea to the general case $q \in (1,\infty)$.  We note that our approach deviates significantly from the local theory developed in \cite{bogeleinHolderRegularitySigned2023,liaoHolderRegularitySigned2022} and instead more closely follows the strategy of \cite{bogeleinHolderRegularitySigned2021} and its adaptation in \cite{adimurthi2025localholderregularitybounded} which in turn is based on \cite{liaoHolderRegularitySigned2022}. 
	\item In the `away from zero case', the  change of variable is only available locally and cannot be performed to handle the tail term. This difficulty was overcome in the case $p=q$ in \cite{adimurthi2025localholderregularitybounded} and the same strategy can be implemented verbatim here. Since this part follows exactly as in \cite{adimurthi2025localholderregularitybounded}, we only provide the final statement of the conclusions. 
	\item By choosing $\mfd=q$ and following the proof of \cref{holderparabolic} in the standard way with obvious modifications, we can directly obtain H\"older regularity of bounded weak solutions of \cref{maineq} in the doubly singular regime $\{1<p<2, p<q\}$ (see \cref{figrangenonlocalsing}) and we leave the details to the interested reader.  
%	\item In the `away from zero case', we can obtain reduction of oscillation considering the degenerate and singular cases separately as done in \cite{adimurthi2025localholderregularitybounded}, or we can make use of the unified approach from \cref{section3}. Both of these approaches gives the same result and we only provide a sketch of the proof and leave the details to the interested reader.  
	
%	\item We provide a new unified proof for weak solutions of 
%	\begin{equation*}
%		\partial_t u + \text{P.V.}\int_{\RR^n} \frac{|u(x,t)-u(y,t)|^{p-2}(u(x,t)-u(y,t))}{|x-y|^{n+sp}}\,dy=0,
%	\end{equation*}
%	 obtained in \cite{adimurthiLocalHolderRegularity2022,liaoHolderRegularityParabolic2022} in order to obtain H\"older regularity in the `away from zero' case. It is not clear how to adapt the ideas from \cite{adimurthiLocalHolderRegularity2022} since the change of variables in \cref{section6} is not available for the tail, although we expect this can be made to work with delicate modifications. In order to overcome this, we instead follow the ideas from \cite{liaoHolderRegularityParabolic2022} and prove H\"older regularity for the mixed type energy estimate obtained in \cref{lemma6.3}. 
%	\item The H\"older regularity and covering argument for the mixed type equation in \cref{lemma6.3} is quite delicate and requires several modifications of existing ideas.
\end{enumerate}

%
%\begin{remark}
%	In this paper, we assume that weak solutions of \cref{maineq} are locally bounded. 
%\end{remark}

\subsection{Notations}
We begin by collecting the standard notation that will be used throughout the paper:
\begin{notationlist}
	\item\label{not1} We shall denote $n$ to be the space dimension and by $z=(x,t)$ to be  a point in $ \RR^n\times (0,T)$.  
	\item\label{not2} We shall alternately use $\dfrac{\partial f}{\partial t}$, $\partial_t f$, $f'$ to denote the time derivative of $f$.
	\item\label{not3} Let $\Omega$ be an open bounded domain in $\mathbb{R}^n$ with boundary $\partial \Omega$ and for $0<T\leq\infty$,  let $\Omega_T\coloneqq \Omega\times (0,T)$. 
	\item\label{not4} We shall use the notation
	\begin{equation*}
		\begin{array}{ll}
			B_{\varrho}(x_0)=\{x\in\RR^n:|x-x_0|<\varrho\}, &
			\overline{B}_{\varrho}(x_0)=\{x\in\RR^n:|x-x_0|\leq\varrho\},\\
			I_{\tht}(t_0)=\{t\in\RR:t_0-\tht<t<t_0\},
			&Q_{\varrho,\tht}(z_0)=B_{\varrho}(x_0)\times I_\tht(t_0).
		\end{array}
	\end{equation*} 
	\item\label{not5} We shall also use the notation $Q_R^\theta(x_0,t_0) := B_R(x_0) \times (t_0 - \theta R^{sp}, t_0)$.  In the case of different scaling in space and time, we denote $Q_R^{\theta_x,\theta_t}(x_0,t_0) := B_{\theta_xR}(x_0) \times (t_0 - \theta_t R^{sp}, t_0)$. 
	%\item\label{not6} The maximum of two numbers $a$ and $b$ will be denoted by $a\wedge b\coloneqq \max(a,b)$. 
	\item\label{not7} Integration with respect to either space or time only will be denoted by a single integral $\int$ whereas integration on $\Om\times\Om$ or $\RR^n\times\RR^n$ will be denoted by a double integral $\iint$. 
	\item\label{not8} We will  use $\iiint dx\,dy\,dt$ to denote integral over $\RR^n \times \RR^n \times (0,T)$. More specifically, we will use the notation $\iiint_{\mcQ}$ and $\iiint_{\mcI \times \mcB}$ to denote the integral over $\iiint_{\mcI\times \mcB \times \mcB }dx\,dy\,dt$ where $\mcQ = \mcB\times \mcI$ with the order of integration made precise by the order of the differentials.
	\item\label{not9} The notation $a \lesssim b$ is shorthand for $a\leq \mopC b$ where $\mopC$ is a constant which only depends on $\datanb{}$ and could change from line to line. We shall denote a constant to depend on data if it depends on $\data{}$. Many a times, we will ignore writing this dependence explicitly for the sake or brevity.
	\item\label{not10} Given a domain $\mcQ=B\times(0,T)$, we denote it's parabolic boundary by $\pa_p \mcQ = (\pa B \times (0,T)) \cup (B \times \{t=0\})$.
	\item\label{not11} For any fixed $t,k\in\RR$ and set $\Om\subset\RR^n$, we denote $A_{\pm}(k,t) := \{x\in \Om: (u-k)_{\pm}(\cdot,t) > 0\}$ and for any ball $B_r$ we write $A_{\pm}(k,t) \cap B_{r} =: A_{\pm}(k,t,r)$. 
	\item\label{not12} For any ball $B \subset \RR^n$, we denote $\mathcal{C}_B:=(B^c\times B^c)^c=\left(B\times B\right) \cup \left( B\times(\RR^n\setminus B)\right) \cup \left((\RR^n\setminus B)\times B \right)$.
%	\item\label{not13} We shall denote a constant to depend on data if it depends on $\data{}$. Many a times, we will ignore writing this dependence explicitly for the sake or brevity.
\end{notationlist}

\subsection{Function spaces}
Let $1<p<\infty$, we denote by $p'=p/(p-1)$ the conjugate exponent of $p$. Let $\Om$ be an open subset of $\RR^n$, we  define the {\it Sobolev-Slobodeki\u i} space, which is the fractional analogue of Sobolev spaces as follows:
\begin{equation*}
	W^{s,p}(\Om):=\left\{ \psi\in L^p(\Omega): [\psi]_{W^{s,p}(\Om)}<\infty \right\},\qquad \text{for} \  s\in (0,1),
\end{equation*} where the seminorm $[\cdot]_{W^{s,p}(\Om)}$ is defined by 
\begin{equation*}
	[\psi]_{W^{s,p}(\Om)}:=\left( \iint_{\Om\times\Om} \frac{|\psi(x)-\psi(y)|^p}{|x-y|^{n+sp}}\,dx\,dy \right)^{\frac 1p}.
\end{equation*}
The space when endowed with the norm $\norm{\psi}_{W^{s,p}(\Om)}=\norm{\psi}_{L^p(\Om)}+[\psi]_{W^{s,p}(\Om)}$ becomes a Banach space. The space $W^{s,p}_0(\Om)$ is the subspace of $W^{s,p}(\RR^n)$ consisting of functions that vanish outside $\Om$. We will use the notation $W^{s,p}_{(u_0)}(\Om)$ to denote the space of functions in $W^{s,p}(\RR^n)$ such that $u-u_0\in W^{s,p}_0(\Om)$.

Let $I$ be an interval and let $V$ be a separable, reflexive Banach space, endowed with a norm $\norm{\cdot}_V$. We denote by $V^*$ to be its topological dual space. Let $v$ be a mapping such that for a.e. $t\in I$, $v(t)\in V$. If the function $t\mapsto \norm{v(t)}_V$ is measurable on $I$, then $v$ is said to belong to the Banach space $L^p(I;V)$ provided $\int_I\norm{v(t)}_V^p\,dt<\infty$. It is well known that the dual space $L^p(I;V)^*$ can be characterized as $L^{p'}(I;V^*)$.

Since the boundedness result requires some finiteness condition on the nonlocal tails, we define the tail space for some $m >0$ and $s >0$ as follows:
\begin{equation*}%\label{tailspace}
	L^m_{s}(\RR^n):=\left\{ v\in L^m_{\text{loc}}(\RR^n):\int_{\RR^n}\frac{|v(x)|^m}{1+|x|^{n+s}}\,dx<+\infty \right\}.
\end{equation*}
Then a nonlocal tail is defined by 
\begin{equation*}
	%    \label{NonlocalTail}
	\text{Tail}_{m,s,\infty}(v;R,x_0,I):=\text{Tail}_\infty(v;R,x_0,t_0-\tht,t_0):=\sup_{t\in (t_0-\tht, t_0)}\left( R^{sm}\int_{\RR^n\setminus B_R(x_0)} \frac{|v(x,t)|^{m-1}}{|x-x_0|^{n+sm}}\,dx \right)^{\frac{1}{m-1}},
\end{equation*} where $(x_0,t_0)\in \RR^n\times (-T,T)$ and the interval $I=(t_0-\tht,t_0)\subseteq (-T,T)$. From this definition, it follows that for any $v\in L^\infty(-T,T;L^{m-1}_{sm}(\RR^n))$, there holds $\text{Tail}_{m,s,\infty}(v;R,x_0,I)<\infty$. 

\subsection{Definitions}

For $k, l\in\RR$ we define 
\begin{equation}\label{defgpm}
	\mfg_\pm^q (l,k):=\pm (q-1)\int_{k}^{l}|s|^{q-2}(s-k)_\pm\,ds.
\end{equation}
Note that $\mfg_\pm^q (l,k)\ge 0$, then we have the following lemma from \cite[Lemma 2.2]{bogeleinHolderRegularitySigned2021}:
\begin{lemma}\label{lem:g}
	For any $q \in (1,\infty)$, there exists a constant $\mopC =\mopC (q)$ such that,
	for all $l,k\in\RR$, the following inequality holds true:
	\begin{equation*}
		\tfrac1{\mopC} \lbr|l| + |k|\rbr^{q-2}(l-k)_\pm^2
		\le
		\mfg_\pm^q (l,k)
		\le
		\mopC  \lbr|l| + |k|\rbr^{q-2}(l-k)_\pm^2
	\end{equation*}
\end{lemma}
Now, we are ready to state the definition of a weak sub(super)-solution.

\begin{definition}
	A function $u\in L^p_{\loc}(I;W^{s,p}_{\text{loc}}(\Om))\cap L^\infty_{\loc}(I;L^q_{\text{loc}}(\Om))\cap L^\infty_{\loc}(I;L^{p-1}_{sp}(\RR^n))$ is said to be a local weak sub(super)-solution to  if for any closed interval $[t_1,t_2]\subset I$ and a compact set $B\subseteq \Om$, the following holds:
	\begin{multline*}
		\int_{B} (|u|^{q-2}u)(x,t_2)\phi(x,t_2)\,dx - \int_{B} (|u|^{q-2}u)(x,t_1)\phi(x,t_1)\,dx - \int_{t_1}^{t_2}\int_{B} (|u|^{q-2}u)(x,t)\partial_t\phi(x,t)\,dx\,dt\\
		+\iiint_{[t_1,t_2] \times \mathcal{C}_B  }\,K(x,y,t)|u(x,t)-u(y,t)|^{p-2}(u(x,t)-u(y,t))(\phi(x,t)-\phi(y,t))\,dy\,dx\,dt
		\leq (\geq)0 ,
	\end{multline*} for all $\phi\in L^p_{\loc}(I,W^{s,p}_{\loc}(\Omega))\cap W^{1,q}_{\loc}(I,L^q(\Om))$ and the spatial support of $\phi$ is contained in $\sigma B$ for some $\sigma \in (0,1)$.
\end{definition}

\subsection{Auxiliary Results}
We collect the following standard results which will be used in the course of the paper. We begin with the Sobolev-type inequality~\cite[Lemma 2.3]{dingLocalBoundednessHolder2021}.

\begin{theorem}\label{fracpoin}
	Let $t_2>t_1>0$ and suppose $\mfs\in(0,1)$ and $1\leq {\mfq}<\infty$. Then for any $f\in L^{\mfq}(t_1,t_2;W^{\mfs,{\mfq}}(B_r))\cap L^\infty(t_1,t_2;L^2(B_r))$, we have
	\begin{equation*}%\label{sobolev}
		\begin{array}{rc@{}l}
			\int_{t_1}^{t_2}\fint_{B_r}|f(x,t)|^{{\mfq}\left(1+\frac{2\mfs}{N}\right)}\,dx\,dt
			& \apprle_{n,\mfs,{\mfq}} &  \left(r^{\mfs{\mfq}}\int_{t_1}^{t_2}\int_{B_r}\fint_{B_r}\frac{|f(x,t)-f(y,t)|^{\mfq}}{|x-y|^{n+\mfs{\mfq}}}\,dx\,dy\,dt+\int_{t_1}^{t_2}\fint_{B_r}|f(x,t)|^{\mfq}\,dx\,dt\right) \\
			&&\quad \times\left(\sup_{t_1<t<t_2}\fint_{B_r}|f(x,t)|^2\,dx\right)^{\frac{\mfs{\mfq}}{N}}.
		\end{array}
	\end{equation*}  
\end{theorem}

%We also list a number of algebraic inequalities that are customary in obtaining energy estimates for nonlinear nonlocal equations.

%\begin{lemma}(\cite[Lemma 4.1]{cozziRegularityResultsHarnack2017})\label{pineq1}
%	Let $p\geq 1$ and $a,b\geq 0$, then for any $\tht\in (0,1)$,  the following holds:
%	\begin{align*}
%		(a+b)^p-a^p\geq \tht p a^{p-1}b + (1-\tht)b^p.
%	\end{align*}
%\end{lemma}
%
%\begin{lemma}(\cite[Lemma 4.3]{cozziRegularityResultsHarnack2017})\label{pineq3}
%	Let $p\geq 1$ and $a\geq b\geq 0$, then for any $\ve>0$, the following holds:
%	\begin{align*}
%		a^p-b^p\leq \ve a^p+\left(\frac{p-1}{\ve}\right)^{p-1}(a-b)^p.
%	\end{align*}
%\end{lemma}
Let us recall the following simple algebraic lemma:
\begin{lemma}\label{alg_lem}
	Let $c_1, c_2 \in (0,\infty)$ and supposed $c_1 \leq \al,\be \leq c_2$ be two numbers with $\alpha \geq \beta$. Then for any $1\leq  \mfq< \infty$, we have 
	\[
	\mfq\frac{ c_1^{\mfq}}{c_2} (\al-\be) \leq \al^{\mfq} - \be^{\mfq} \leq \mfq \frac{c_2^{\mfq}}{c_1} (\al-\be).
	\]
	In the case $0 < \mfq < 1$, we instead have
	\[
	\mfq\frac{ c_1}{c_2^{2-\mfq}} (\al-\be) \leq \al^{\mfq} - \be^{\mfq} \leq \mfq \frac{c_2}{c_1^{2-\mfq}} (\al-\be).
	\]
\end{lemma}
\begin{proof}
	Let $\be \in [c_1,c_2]$ be fixed and denote $\al = \be+x$ for $x \in [0,c_2-\be]$. Define the functions 
	\[
	\text{Case}\,\, \mfq \geq 1:\begin{cases}
		f_1(x) = (\be+x)^{\mfq} - \be^{\mfq} - {\mfq}\tfrac{ c_1^{\mfq}}{c_2}x, \\
		f_2(x) = (\be+x)^{\mfq} - \be^{\mfq} - {\mfq} \tfrac{c_2^{{\mfq}}}{c_1}x,
	\end{cases}
	\text{Case}\,\, 0<{\mfq}<1:\begin{cases}
		h_1(x) = (\be+x)^{\mfq} - \be^{\mfq} - {\mfq}\tfrac{ c_1}{c_2^{2-{\mfq}}}x, \\
		h_2(x) = (\be+x)^{\mfq} - \be^{\mfq} -{\mfq} \tfrac{c_2}{c_1^{2-{\mfq}}}x.
	\end{cases}
	\]
	Then, we see that for $x \in [0,c_2-\be]$, there holds $f_1(0) = f_2(0) = h_1(0)= h_2(0)= 0$ along with  $f_1'(x) \geq 0$, $f_2'(x) \leq 0$, $h_1'(x) \geq 0$ and  $h_2'(x) \leq 0$ which proves the lemma.
\end{proof}

Finally, we recall the following well known lemma concerning the geometric convergence of sequence of numbers (see \cite[Lemma 4.1 from Section I]{dibenedettoDegenerateParabolicEquations1993} for the details): 
\begin{lemma}\label{geo_con}
	Let $\{Y_n\}$, $n=0,1,2,\ldots$, be a sequence of positive number, satisfying the recursive inequalities 
	\[ Y_{n+1} \leq C b^n Y_{n}^{1+\alpha},\]
	where $C > 1$, $b>1$, and $\alpha > 0$ are given numbers. If 
	\[ Y_0 \leq  C^{-\frac{1}{\alpha}}b^{-\frac{1}{\alpha^2}},\]
	then $\{Y_n\}$ converges to zero as $n\to \infty$. 
\end{lemma}
Let us recall the following algebraic lemma from \cite[Lemma 4.1]{cozziRegularityResultsHarnack2017}: 
\begin{lemma}\label{pineq1}
	Let ${\mfq}\geq 1$ and $a,b\geq 0$, then for any $\tht\in [0,1]$,  the following holds:
	\begin{equation*}
		(a+b)^{\mfq}-a^{\mfq}\geq \tht {\mfq} a^{{\mfq}-1}b + (1-\tht)b^{\mfq}.
	\end{equation*}
\end{lemma}

Let us recall the following algebraic lemma from \cite[Lemma 4.3]{cozziRegularityResultsHarnack2017}:
\begin{lemma}\label{pineq3}
	Let ${\mfq}\geq 1$ and $a\geq b\geq 0$, then for any $\ve>0$, the following holds:
	\begin{equation*}
		a^{\mfq}-b^{\mfq}\leq \ve a^{\mfq}+\left(\tfrac{{\mfq}-1}{\ve}\right)^{{\mfq}-1}(a-b)^{\mfq}.
	\end{equation*}
\end{lemma}
	%\hrule
	
	\section{Preliminary Estimates}
	In this section, we recall some important estimates. The first one is a standard energy estimate,   the proof of which follows by combining \cite[Proposition 3.1]{bogeleinHolderRegularitySigned2021} along with \cite[Theorem 3.1]{adimurthi2024localholderregularitynonlocal}.
	\begin{proposition}\label{Prop:energy}
		Let $u$ be a  local weak sub(super)-solution of \cref{maineq} in $\Om_T$ and $z_o = (x_o,t_o)$ be a fixed point.
		There exists a constant $\mopC  >0$ depending only on the data such that
		for all cylinders $\mcQ_{R,S}=B_R(x_o)\times (t_o-S,t_o)\Subset E_T$,
		every $k\in\RR$, and every non-negative, piecewise smooth cut-off function
		$\zeta$ vanishing on $\partial B_{R}(x_o)\times (t_o-S,t_o)$,  there holds 
		\begin{multline*}
			\esssup_{t_o-S<t<t_o}\int_{B_R(x_o)\times\{t\}}	
			\zeta^p\mfg_\pm^q (u,k)\,dx +
			\iint_{\mcQ_{R,S}(z_o)} (u-k)_{\pm}(x,t)\zeta^p(x,t)\int_{B_R(x_o)}\frac{(u-k)_{\mp}^{p-1}(y,t)}{|x-y|^{n+sp}}\,dy\,dx\,dt 
			\\
			+\int_{t_o-S}^{t_o}\iint_{B_R(x_o)\times B_R(x_o)}|(u-k)_{\pm}(x,t)\zeta(x,t)-(u-k)_{\pm}(y,t)\zeta(y,t)|^p\,d\mu\,dt\\ 
			\def\arraystretch{2.2}
			\begin{array}{rcl}
				&\leq&
				\mopC \int_{t_o-S}^{t_o}\iint_{B_R(x_o)\times B_R(x_o)}\hspace*{-1.6cm} \frac{\max\{(u-k)_{\pm}(x,t),(u-k)_{\pm}(y,t)\}^{p}|\zeta(x,t)-\zeta(y,t)|^p}{|x-y|^{n+sp}}\,dx\,dy\,dt
				\\
				&&+
				\iint_{\mcQ_{R,S}(z_o)}\mfg_\pm^q (u,k)|\partial_t\zeta^p| \,dx\,dt
				+\int_{B_R(x_o)\times \{t_o-S\}} \zeta^p \mfg_\pm^q (u,k)\,dx 
				\\ 
				&&+\mopC\lbr  \underset{\stackrel{t \in (t_o-S,t_o)}{x\in \spt \zeta}}{\esssup}\,\int_{\RR^n \setminus B_R(x_o)}\frac{(u-k)_{\pm}^{p-1}(y,t)}{|x-y|^{n+sp}}\,dy\rbr\iint_{(t_o-S,t_o)\times B_R(x_o)}\hspace*{-1.6cm} (u-k)_{\pm}(x,t)\zeta^p(x,t)\,dx\,dt,
			\end{array}
		\end{multline*}
		where we recall the definition of $\mfg_\pm^q (u,k)$ from  \cref{defgpm}.
	\end{proposition}
	
	\subsection{Shrinking Lemma}
	
	One of the main difficulties we face when dealing with regularity issues for nonlocal equations is the lack of a corresponding isoperimetric inequality for $W^{s,p}$ functions. Indeed, since such functions can have jumps, a generic isoperimetric inequality seems out of reach at this time (see \cite{cozziFractionalGiorgiClasses2019, adimurthiOlderRegularityFractional2022} ) One way around this issue is to note that since we are working with solutions of an equation, which we expect to be continuous and hence have no jumps, we could try and cook up an isoperimetric inequality for solutions. Such a strategy turns out to be feasible due to the presence of the ``good term'' or the isoperimetric term in the Caccioppoli inequality. The following lemma can be found in \cite[Lemma 3.3]{adimurthi2024localholderregularitynonlocal}.
	
	\begin{lemma}\label{lem:isop}
		Let $k<l<m$ be arbitrary levels and $A \geq 1$. Then,
		\[
		(l-k)(m-l)^{p-1}\abs{[u>m]\cap B_{\varrho}}\abs{[u<k]\cap B_{\varrho}} \leq 
		C\varrho^{n+sp}\int_{B_{\varrho}} (u-l)_{-}(x)\int_{B_{A\varrho}}\frac{(u-l)_{+}^{p-1}(y)}{|x-y|^{n+sp}}\,dy\,dx,
		\]
		where $C = C(n,s,p,A)>0$. 
	\end{lemma}
	
	We can apply \cref{lem:isop} to get the following shrinking lemma whose proof can be found in \cite[Lemma 3.4]{adimurthi2024localholderregularitynonlocal}.
	\begin{lemma}\label{lem:shrinking}
		Let $u$ be a given function and suppose that for some level $m$, some constant $\nu \in (0,1)$ and all time levels $\tau$ in some interval $I$, we have
		\[
		|[u(\cdot,\tau)>m]\cap B_{\varrho}| \geq \nu|B_{\varrho}|,
		\]
		and we can arrange that for some $A \geq 1$,  the following is also satisfied:
		\begin{equation*}%\label{smallness}
			\iint_{I\times B_{\varrho}} (u-l)_{-}(x,t)\int_{B_{A\varrho}}\frac{(u-l)_{+}^{p-1}(y,t)}{|x-y|^{n+sp}}\,dy\,dx\,dt \leq {\mbC}_1\frac{l^p}{\varrho^{sp}}|Q|,
		\end{equation*}
		where 
		\[
		l = \frac{m}{2^{j}}, \qquad j\geq 1 \qquad \text{ and } \qquad Q := B_{\varrho} \times I,
		\] 
		then the following conclusion holds:
		\[
		\left|\left[u<\frac{m}{2^{j+1}}\right]\cap Q\right| \leq \left(\frac{C}{2^{j}-1}\right)^{p-1}|Q|,
		\]
		where $C = ({\mbC}_1,A,n,\nu) >0$.
	\end{lemma}
	%\hrule
	\subsection{Tail Estimates}\label{sec:tail}
	In this section, we want to outline how we estimate the tail term and we shall refer to this section whenever we make a similar calculation.
	
	For any level $k = \bsmu^{-}+ M$ with $M > 0$, we want to estimate
	\[
	\underset{\stackrel{t \in I;}{x\in \spt \zeta}}{\esssup}\int_{\RR^n \setminus B_{\varrho}(y_o)}\frac{(u-k)_{-}^{p-1}(y,t)}{|x-y|^{n+sp}}\,dy,
	\]
	where $I$ is some time interval and $\zeta$ is a cut-off function supported in $B_{\vartheta\varrho}$ for some $\vartheta \in (0,1)$. In such case, we have
	\[
	|y-y_o| \leq |x-y|\left(1+\frac{|x-y_o|}{|x-y|}\right)\leq  |x-y|\left(1+\frac{\vartheta}{(1-\vartheta)}\right),
	\]
	using which, we get
	\[
	\underset{\stackrel{t \in I;}{ x\in \spt \zeta}}{\esssup}\int_{\RR^n \setminus B_{\varrho}(y_o)}\frac{(u-k)_{-}^{p-1}(y,t)}{|x-y|^{n+sp}}\,dy 
	\leq \frac{1}{(1-\vartheta)^{n+sp}}\underset{t \in I}{\esssup}\int_{\RR^n \setminus B_{\varrho}(y_o)}\frac{(u-k)_{-}^{p-1}(y,t)}{|y-y_o|^{n+sp}}\,dy. 
	\]
	In the local case, we could always estimate $(u-k)_{-} \leq k$ because we take $u \geq 0$ locally. However, in the $\tail$ term,  we are on the complement of a cube and so unless we make a global boundedness assumption (for e.g. $u \geq 0$ in full space), the best we can do is
	\[
	(u-k)_{-} \leq u_{-} + k,
	\]
	which leads us to the next estimate
	\[
	\underset{t \in I}{\esssup}\int_{\RR^n \setminus B_{\varrho}(y_o)}\frac{(u-k)_{-}^{p-1}(y,t)}{|y-y_0|^{n+sp}}\,dy  \leq C(p)\frac{M^{p-1}}{\varrho^{sp}} +  C(p)\,\underset{t \in I}{\esssup}\int_{\RR^n \setminus B_{\varrho}(y_o)}\frac{(u -\bsmu^-)_-^{p-1}(y,t)}{|y-y_o|^{n+sp}}\,dy.
	\]
	Putting together the above estimates yield
	\[
	\underset{\stackrel{t \in I;}{ x\in \spt \zeta}}{\esssup}\int_{\RR^n \setminus B_{\varrho}(y_o)}\frac{(u-k)_{-}^{p-1}(y,t)}{|x-y|^{n+sp}}\,dy  \leq \frac{C(p)}{\varrho^{sp}}\left[M^{p-1}+\tailp((u -\bsmu^-)_-;y_o,\varrho,I)\right].
	\]
	Finally we usually want an estimate of the form
	\[
	\underset{\stackrel{t \in I;}{ x\in \spt \zeta}}{\esssup}\int_{\RR^n \setminus B_{\varrho}(y_o)}\frac{(u-k)_{-}^{p-1}(y,t)}{|x-y|^{n+sp}}\,dy \leq C\frac{M^{p-1}}{\varrho^{sp}},
	\]
	and so we impose the condition that the following is satisfied:
	\[
	\tailp((u -\bsmu^-)_-;y_0,\varrho,I) \leq M^{p-1}.
	\]
	This is the origin of the various Tail alternatives. 
	
	\begin{remark}
		We have an analogous estimate regarding $k = \boldsymbol{\mu}^+-M$ which corresponds to subsolutions.
	\end{remark}

\begin{remark}\label{Rmk:5.1}
	Let $B_R(x_o) \times (t_o, t_o+S)$ be some cylinder and let $\nu^\pm$ be two numbers satisfying
	\[
	\nu^+\geq  \esssup_{B_R(x_o) \times (t_o, t_o+S)}u \qquad \text{ and }
	\qquad 
	\nu^-\leq \essinf_{B_R(x_o) \times (t_o, t_o+S)} u,
	\]
	then we see that $(u-\nu^{\pm})_{\pm} = 0$ in $B_R(x_o)\times (t_o,t_o+S)$. Thus, for any $\varrho \leq R$, we have
	\[
	\tail((u-\nu^{\pm})_{\pm};x_o;\varrho;(t_o, t_o + S)) \leq \left(\frac{\varrho}{R}\right)^{\frac{sp}{p-1}}\tail((u-\nu^{\pm})_{\pm};x_o;R;(t_0, t_o + S)).
	\]
	%	so the $\tail$ alternative from \cref{Prop:1:eq} in \cref{Prop:1:1} can be rewritten as
	%	\[
	%\left(\frac{\varrho}{R}\right)^{\frac{sp}{p-1}}\tail((u-\bsmu^{\pm})_{\pm};x_o;R;(t_o, t_o + \delta \varrho^{sp})) \leq \eta M.
	%\]
\end{remark}
%	{\color{blue}\hrule\hrule\hrule\hrule\hrule}
	
	\section{Partially unified proof of H\"older continuity when \texorpdfstring{$q=2$}.}\label{section3}
	
	Let us recall the energy estimate, see \cite[Lemma 3.3]{dingLocalBoundednessHolder2021} for the details. 
	\begin{lemma}\label{energyqequal2}
		 Let $u$ be a weak solution of \cref{maineq} in the case $q=2$ and $k \in \RR$ be given, then
		there exists a constant $\mopC  >0$ depending only on the data such that
		for all  every non-negative, piecewise smooth cut-off function
		$\zeta(x,t) = \zeta_1(x)\zeta_2(t)$ on $\mcQ = B_{R}(x_o) \times (t_o - S,t_o)$ and  vanishing on $\partial_p \mcQ$,  there holds
		\begin{multline*}
			\esssup_{t_o-S<t<t_o}\int_{B_R(x_o)\times\{t\}}	
			\zeta^p(u-k)_{\pm}^2\,dx +
			\iint_{\mcQ_{R,S}(z_o)} (u-k)_{\pm}(x,t)\zeta^p(x,t)\int_{B_R(x_o)}\frac{(u-k)_{\mp}^{p-1}(y,t)}{|x-y|^{n+sp}}\,dy\,dx\,dt 
			\\
			+\int_{t_o-S}^{t_o}\iint_{B_R(x_o)\times B_R(x_o)}|(u-k)_{\pm}(x,t)\zeta(x,t)-(u-k)_{\pm}(y,t)\zeta(y,t)|^p\,d\mu\,dt\\ 
			\def\arraystretch{2.2}
			\begin{array}{rcl}
				&\leq&
				\mopC \int_{t_o-S}^{t_o}\iint_{B_R(x_o)\times B_R(x_o)}\hspace*{-1.6cm} \frac{\max\{(u-k)_{\pm}(x,t),(u-k)_{\pm}(y,t)\}^{p}|\zeta(x,t)-\zeta(y,t)|^p}{|x-y|^{n+sp}}\,dx\,dy\,dt
				\\
				&&+
				\iint_{\mcQ_{R,S}(z_o)}(u-k)_{\pm}^2|\partial_t\zeta^p| \,dx\,dt
				+\int_{B_R(x_o)\times \{t_o-S\}} \zeta^p (u-k)_{\pm}^2\,dx 
				\\ 
				&&+\mopC\lbr  \underset{\stackrel{t \in (t_o-S,t_o)}{x\in \spt \zeta}}{\esssup}\,\int_{\RR^n \setminus B_R(x_o)}\frac{(u-k)_{\pm}^{p-1}(y,t)}{|x-y|^{n+sp}}\,dy\rbr\iint_{(t_o-S,t_o)\times B_R(x_o)}\hspace*{-1.6cm} (u-k)_{\pm}(x,t)\zeta^p(x,t)\,dx\,dt.
			\end{array}
		\end{multline*}
	\end{lemma}
\subsection{Defining the constants}
	For a compact ball $B_R\subset\RR^n$ of radius $R$ and
	a cylinder $\mcQ_o := B_R\times(T_1,T_2]\subset \Om_T$, 
	we introduce numbers $\bsmu^{\pm}$ and $\bsom$ satisfying
	\begin{equation*}
		\bsmu^+\geq  \esssup_{\mcQ_o}u, \qquad \bsmu^-\leq \essinf_{\mcQ_o} u \qquad \text{ and } 
		\qquad \bsom \geq \bsmu^+ - \bsmu^-.	
	\end{equation*}
	For some $\varrho$, we also assume $(x_o,t_o)\in \mcQ_o$, such that the	cylinder 
	\begin{equation*}%\label{Eq:3.1}
		B_{8\varrho}(x_o)\times\left(t_o-(8\varrho)^{sp},t_o+(8\varrho)^{sp}\right]\subset\mcQ_o.
	\end{equation*}
We further assume $\bsom$ is chosen such that 
\[
\tail((u-\bsmu^\pm)_{\pm},8\varrho,0,(-(8\varrho)^{sp},0)) \leq \bsom. 
\]
	In subsequent calculations, we will assume $(x_o,t_o) = (0,0)$ and obtain reduction of oscillation at $(0,0)$. 
	
	\subsection{de Giorgi iteration} 
	The first lemma we prove is the de Giorgi iteration:
	\begin{lemma}\label{lemma3.2}
		Let $u$ be a weak solution of \cref{maineq}. Given $\de_x, \de_t, \varepsilon \in (0,1)$, set $\theta_t = \de_t (\varepsilon \bsom)^{\mfd-p}$, $\theta_x = \de_x (\varepsilon \bsom)^{\frac{\mfd -2}{sp}}$ for some $1<\mfd<\min\{2,p\}$ and denote $\mcQ_{c_o\varrho}^{\theta_x,\theta_t} = B_{\theta_x c_o\varrho} \times (-\theta_t (c_o\varrho)^{sp},0)$. Then there exists a universal constant $\nu_1 = \nu(\datanb{,\delta_x,\delta_t}) \in (0,1)$ such that if 
		\begin{equation*}
			\left|\left\{
			\pm\left(\bsmu^{\pm}-u\right)\leq \varepsilon \bsom\right\}\cap \mcQ_{c_o\varrho}^{\theta_x,\theta_t}\right|
			\leq
			\nu_1|\mcQ_{c_o\varrho}^{\theta_x,\theta_t}|,
		\end{equation*}
		holds along with the assumption 
		\begin{equation*}%\label{vwLm:3:3:hypothesis}
			c_o^{\frac{sp}{p-1}}\tail((u-\bsmu_-)_-, \theta_x\varrho,0,(-\theta_t \varrho^{sp},0))\leq \varepsilon\bsom,
		\end{equation*}
		then the following conclusion follows:
		\begin{equation*}
			\pm\left(\bsmu^{\pm}-u\right)\geq\tfrac{1}2\varepsilon \bsom
			\quad
			\mbox{ on }\quad \mcQ_{\frac{1}{2}c_o\varrho}^{\theta_x,\theta_t} = B_{\theta_x\frac{c_o\varrho}2} \times \left(-\theta_t\lbr \tfrac12c_o\varrho\rbr^{sp}, 0\right].
		\end{equation*}
	If we denote $\Gamma := \lbr \tfrac{1}{\de_x^{sp}} + \tfrac{1}{\de_t}\rbr^{\frac{n+sp}{p(n+2s)}} (\de_x^n\de_t)^{\frac{s}{(n+2s)}}$, then the relation between $\nu_1$ and $\Gamma$ takes the form $ \nu_1 \approx \Gamma^{-\frac{n+2s}{s}}$ or in particular, $\nu_1 \approx \lbr \tfrac{1}{\de_x^{sp}} + \tfrac{1}{\de_t}\rbr^{-\frac{n+sp}{sp}} (\de_x^n\de_t)^{-1}$.
	\end{lemma}
	\begin{proof}
		For $j=0,1,\ldots$, define
		\begin{equation*}%\label{choices:B_n}
%			\left\{
			{\def\arraystretch{1.1}\begin{array}{cccc}
					k_j:=\bsmu^-+\tfrac{\varepsilon \bsom}2+\tfrac{\varepsilon \bsom}{2^{j+1}},& \tilde{k}_j:=\tfrac{k_j+k_{j+1}}2,&
					\varrho_j:=\tfrac{c_o\varrho}2+\tfrac{c_o\varrho}{2^{j+1}},
					&\tilde{\varrho}_j:=\tfrac{\varrho_j+\varrho_{j+1}}2,\\%[5pt]
					B_j:=B_{\varrho_j}^{\theta_x},& \tilde{B}_j:=B_{\tilde{\varrho}_j}^{\theta_x},&
					\mcQ_j:=\mcQ_{\varrho_j}^{\theta_x,\theta_t},&
					\tilde{\mcQ}_j:=\mcQ_{\tilde\varrho_j}^{\theta_x,\theta_t}.
			\end{array}}
%			\right.
		\end{equation*}
		Furthermore, we also define 
		\begin{equation*}%\label{choices:B_nn}
			%	\left\{
			%	{\def\arraystretch{1.1}\begin{array}{cccc}
			\hat{\varrho}_j:=\tfrac{3\varrho_j+\varrho_{j+1}}4, \quad 
			\bar{\varrho}_j:=\tfrac{\varrho_j+3\varrho_{j+1}}4,\quad 
			\hat{\mcQ}_j:=\mcQ_{\hat\varrho_j}^{\theta_x,\theta_t}, \quad
			\bar{\mcQ}_j:=\mcQ_{\bar\varrho_j}^{\theta_x,\theta_t}.
			%	\end{array}}
			%	\right.
		\end{equation*}
		We now consider a cutoff functions $\bar{\zeta_j}$ and $\zeta_j$ such that
		\begin{equation*}%\label{Lm:3.2:cutoff}
			\begin{array}{c}
				\bar{\zeta_j} \equiv 1 \text{ on } B_{j+1}, \quad \bar{\zeta_j} \in C_c^{\infty}(\bar{B}_{j}), \quad  |\nabla\bar{\zeta_j}| \apprle \frac{1}{\theta_x(\bar{\varrho_j} - \varrho_{j+1})} \approx \frac{2^j}{\theta_xc_o\varrho} \quad \text{and} \quad  |\pa_t\bar{\zeta_n}| \apprle \frac{1}{\theta_t(\bar\varrho_j^{sp} - \varrho_{j+1}^{sp})} \approx \frac{2^{jsp}}{\theta_t(c_o\varrho)^{sp}}, \\
				{\zeta_j} \equiv 1 \text{ on } \tilde{B}_{j}, \quad {\zeta_j} \in C_c^{\infty}(\hat{B}_{j}), \quad  |\nabla{\zeta_j}| \apprle \frac{1}{\theta_x(\hat{\varrho_n} - \tilde{\varrho_{j}})}\approx \frac{2^{j}}{\theta_xc_o\varrho} \quad \text{and} \quad  |\pa_t{\zeta_j}| \apprle \frac{1}{\theta_t(\hat\varrho_j^{sp} - \tilde{\varrho_{j}}^{sp})}\approx \frac{2^{jsp}}{\theta_t(c_o\varrho)^{sp}}.
			\end{array}
		\end{equation*}
		Let us estimate each of the terms appearing on the right hand side of \cref{energyqequal2} as follows:
		\begin{description}[leftmargin=*]
			\item[Estimate for the first term:] Since $(u-k_j)-_ \leq \ve \bsom$, we get 
			\begin{multline*}
				\int_{-\theta_t \varrho_j^{sp}}^0\iint_{B_j \times B_j}\hspace*{-0.7cm} \frac{\max\{(u-k_j)_{-}(x,t),(u-k_j)_{-}(y,t)\}^{p}|\zeta(x,t)-\zeta(y,t)|^p}{|x-y|^{n+sp}}\,dx\,dy\,dt\\
				\begin{array}{rcl}
					& \leq & (\ve \bsom)^p  \lbr \frac{2^{j}}{\theta_x(c_o\varrho)} \rbr^p \int_{-\theta_t \varrho_j^{sp}}^0 \int_{B_j}\int_{B_{j}} \frac{\lsb{\chi}{\{u(x,t) < k_j\}}}{|x-y|^{n+(s-1)p}}\,dx\,dy\,dt\\
					& \leq & \mopC (\ve \bsom)^p  \lbr \frac{2^{j}}{\theta_x(c_o\varrho)} \rbr^p \int_{-\theta_t \varrho_j^{sp}}^0 \int_{B_j}\int_{4B_{j}} \frac{\lsb{\chi}{\{u(x,t) < k_j\}}}{|y|^{n+(s-1)p}}\,dy\,dx\,dt\\
					& \leq & \mopC (\ve \bsom)^p   \frac{2^{pj}}{(\theta_x\varrho_j)^{sp}} |A_j| \approx \mopC (\ve \bsom)^p   \frac{2^{pj}}{\theta_x^{sp}(c_o\varrho)^{sp}} |A_j|,
				\end{array}
			\end{multline*}
		where we have denoted $A_j := \{u(x,t) \leq k_j\} \cap \mcQ_j$.
			
			\item[Estimate for the second term:] This is easily estimated as follows:
			\begin{equation*}
				\begin{array}{rcl}
					\iint_{\mcQ_j}(u-k_j)_{-}^2|\partial_t\zeta_j^p| \,dx\,dt & \leq & (\ve \bsom)^2 |A_j| \frac{2^{jsp}}{\theta_t (c_o \varrho)^{sp}}.
				\end{array}
			\end{equation*}
		\item[Estimate for the third term:] This term is zero as our cut-off functions are chosen to have zero values on the parabolic boundary of $\mcQ_j$. 
			\item[Estimate for the fourth term:] Since $x \in \spt \zeta_j \Longrightarrow |x| \leq \theta_x \hat{\varrho_j}$ and $|y| \geq \theta_x\varrho_j$, we get
			\[
			\frac{|y-x|}{|y|} \geq \frac{\varrho_j - \hat{\varrho_j}}{\hat{\varrho_j}} = \frac14 \lbr  \frac{\varrho_j - \varrho_{j+1}}{\varrho_j}\rbr \geq \frac{1}{2^{j+4}},
			\]
			which allows us to estimate as follows:
			\begin{multline*}
				\lbr  \underset{\stackrel{t \in (-\theta_t \varrho_j^{sp},0)}{x\in \spt \zeta_j}}{\esssup}\,\int_{\RR^n \setminus B_j}\frac{(u-k_j)_{-}^{p-1}(y,t)}{|x-y|^{n+sp}}\,dy\rbr\iint_{(-\theta_t \varrho_j^{sp},0)\times B_j}\hspace*{-0.6cm} (u-k_j)_{-}(x,t)\zeta_j^p(x,t)\,dx\,dt\\
				\begin{array}{rcl}
					& \leq & \mopC 2^{(n+sp)j}(\ve \bsom) |A_j| \lbr\underset{t \in (-\theta_t \varrho_j^{sp},0)}{\esssup}\,\int_{\RR^n \setminus B_j}\frac{(\ve\bsom + (\bsmu_--u)_+ )^{p-1}(y,t)}{|y|^{n+sp}}\,dy \rbr\\
					& \leq & \mopC 2^{(n+sp)j}(\ve \bsom) |A_j|  \lbr\frac{(\ve \bsom)^{p-1}}{(\theta_x\varrho_j)^{sp}}+\underset{t \in (-\theta_t \varrho_j^{sp},0)}{\esssup}\,\int_{\RR^n \setminus B_j}\frac{(u -\bsmu_- )_-^{p-1}(y,t)}{|y|^{n+sp}}\,dy \rbr\\
					& = & \mopC 2^{(n+sp)j}(\ve \bsom) |A_j|  \lbr\frac{(\ve \bsom)^{p-1}}{(\theta_x\varrho_j)^{sp}}+\frac{(\theta_x \varrho_j)^{sp-sp}}{(\theta_x \varrho)^{sp}}\underset{t \in (-\theta_t \varrho_j^{sp},0)}{\esssup} (\theta_x \varrho)^{sp}\int_{\RR^n \setminus B_j}\frac{(u -\bsmu_- )_-^{p-1}(y,t)}{|y|^{n+sp}}\,dy \rbr\\
					& \leq & \mopC 2^{(n+sp)j}(\ve \bsom) |A_j|  \lbr\frac{(\ve \bsom)^{p-1}}{(\theta_x\varrho_j)^{sp}}+\frac{(\theta_x \varrho_j)^{sp-sp}}{(\theta_x \varrho)^{sp}}\underset{t \in (-\theta_t \varrho^{sp},0)}{\esssup} (\theta_x \varrho)^{sp}\int_{\RR^n \setminus B_j}\frac{(u -\bsmu_- )_-^{p-1}(y,t)}{|y|^{n+sp}}\,dy \rbr\\
					& \leq & \mopC 2^{(n+sp)j}(\ve \bsom) |A_j|  \lbr\frac{(\ve \bsom)^{p-1}}{(\theta_x\varrho_j)^{sp}}+\frac{c_o^{sp}\tail((u-\bsmu_-)_-, \theta_x\varrho,0,(-\theta_t \varrho^{sp},0))^{p-1} }{(\theta_x\varrho_j)^{sp}}\rbr\\
					& \leq & \mopC 2^{(n+sp)j}(\ve \bsom) |A_j|  \frac{(\ve \bsom)^{p-1}}{(\theta_x\varrho_j)^{sp}} \approx  \mopC 2^{(n+sp)j}(\ve \bsom) |A_j|  \frac{(\ve \bsom)^{p-1}}{\theta_x^{sp}(c_o\varrho)^{sp}}.
				\end{array}
			\end{multline*}
		\end{description}
	Combining the previous four estimates, we get
	\begin{multline}\label{Eq:3.2}
%		\begin{aligned}
			\underset{(-\theta_t \tilde\varrho_j^{sp},0)}{\esssup}
			\int_{\tilde{B}_j} (u-{k}_j)_-^2\,dx
			+\iiint_{\tilde{\mcQ}_j}\frac{|(u-{k}_j)_{-}(x,t)-(u-{k}_j)_{-}|^p}{|x-y|^{n+sp}}\,dx \,dy\,dt
			\\\leq
			\mopC 2^{(n+sp)j} |A_j| \lbr  \frac{(\ve \bsom)^{p}}{\theta_x^{sp}(c_o\varrho)^{sp}} +  \frac{(\ve \bsom)^2}{\theta_t (c_o \varrho)^{sp}}\rbr.
%		\end{aligned}
	\end{multline}
%\hrule
Using Young's inequality, we have
\begin{multline*}%\label{Eq:3.3}
	|(u-{k}_{j})_- \bar{\zeta_j}(x,t) - (u-{k}_{j})_- \bar{\zeta_j}(y,t)|^p \leq c |(u-{k}_{j})_- (x,t) - (u-{k}_{j})_- (y,t)|^p\bar{\zeta_j}^p(x,t) \\
	+ c |(u-{k}_{j})_-(y,t)|^p |\bar{\zeta_j}(x,t) - \bar{\zeta_j}(y,t)|^p,
\end{multline*}
from which we obtain the following sequence of estimates:
\begin{equation*}
	\begin{array}{rcl}
		\frac{\ve \bsom}{2^{j+2}}
		|A_{j+1}|
		&\overred{3.4a}{a}{\leq} &
		\iint_{\tilde\mcQ_{j}}(u-{k}_j)_-\bar{\zeta_j}\,dx\,dt \\
		&\overred{3.4b}{b}{\leq} &
		\lbr\iint_{\tilde\mcQ_{j}}\left[(u-{k}_j)_-\bar{\zeta_j}\right]^{p\frac{n+2s}{n}}
		\,dx\,dt\rbr^{\frac{n}{p(n+2s)}}|A_j|^{1-\frac{n}{p(n+2s)}}\\
		&\overred{3.4c}{c}{\leq} &\mopC 
		\left(\iiint_{\tilde\mcQ_j}\frac{|(u-{k}_j)_{-}(x,t)\bar{\zeta_j}(x,t)-(u-{k}_j)_{-}\bar{\zeta_j}(y,t)|^p}{|x-y|^{n+sp}}\,dx \,dy\,dt\right)^{\frac{n}{p(n+2s)}} 
		\\&&\qquad\times \left(\underset{-\theta_t\tilde\varrho_j^{sp} < t < 0}{\esssup}\int_{\tilde{B}_j}[(u-\tilde{k}_j)_{-}\bar\zeta_j(x,t)]^2\,dx\right)^{\frac{s}{n+2s}}|A_j|^{1-\frac{n}{p(n+2s)}}\\
		&\overred{3.4d}{d}{\leq}&
		\mopC b_o^j   \lbr  \frac{(\ve \bsom)^{p}}{\theta_x^{sp}(c_o\varrho)^{sp}} +  \frac{(\ve \bsom)^2}{\theta_t (c_o \varrho)^{sp}}\rbr^{\frac{n+sp}{p(n+2s)}}
		|A_j|^{1+\frac{s}{n+2s}} \\
		&\overred{3.4e}{e}{\leq} &
		\mopC
		\frac{b_o^j}{(c_o\varrho)^\frac{s(n+sp)}{n+2s}}(\ve\bsom)^{(p+2-\mfd){\frac{n+sp}{p(n+2s)}}}
		\lbr \tfrac{1}{\de_x^{sp}} + \tfrac{1}{\de_t}\rbr^{\frac{n+sp}{p(n+2s)}}|A_j|^{1+\frac{s}{n+2s}}, 
	\end{array}
\end{equation*}
where to obtain \redref{3.4a}{a}, we made use of the observations and enlarged the domain of integration with $\bar{\zeta}_j$; to obtain \redref{3.4b}{b}, we applied H\"older's inequality; to obtain \redref{3.4c}{c}, we applied \cref{fracpoin}; to obtain \redref{3.4d}{d}, we made use of \cref{Eq:3.2}  with $\mopC = \mopC_{\data{}}$ and finally we collected all the terms to obtain \redref{3.4e}{e}, where $b_o = b_o(\datanb{})\geq 1$ is a constant. Setting
$\bsy_j=|A_j|/|\mcQ_j|$ and noting $|\mcQ_{i+1}| \approx |\mcQ_j| \approx \de_x^n \de_t (\ve \bsom)^{\frac{n(\mfd-2)}{sp}+\mfd -p} (c_o\varrho)^{n+sp}$,  we get
\[
\frac{\ve \bsom}{2^{j+2}} \bsy_{j+1} \leq \mopC  \bsy_j^{1+\frac{s}{n+2s}}\frac{b_o^j}{(c_o\varrho)^\frac{s(n+sp)}{n+2s}}(\ve\bsom)^{(p+2-\mfd){\frac{n+sp}{p(n+2s)}}}
\lbr \tfrac{1}{\de_x^{sp}} + \tfrac{1}{\de_t}\rbr^{\frac{n+sp}{p(n+2s)}} |\mcQ_{j}|^{\frac{s}{n+2s}}.
\]
Simplifying this expression noting that $(p+2-\mfd)\lbr \tfrac{n+sp}{p(n+2s)}\rbr + \lbr \tfrac{s}{n+2s}\rbr \lbr \tfrac{n(\mfd-2)}{sp} + \mfd - p \rbr = 1$,  we get
\begin{equation*}
	\bsy_{j+1}
	\le
	\bsc \Gamma  \boldsymbol b^j  \bsy_j^{1+\frac{s}{n+2s}},
\end{equation*}
where $\Gamma := \lbr \tfrac{1}{\de_x^{sp}} + \tfrac{1}{\de_t}\rbr^{\frac{n+sp}{p(n+2s)}} (\de_x^n\de_t)^{\frac{s}{(n+2s)}}$ and   $\bsc  = \bsc_{\data{}} \geq 1$,  $\boldsymbol b\geq  1$ are two constants depending only on the data. The conclusion now follows from \cref{geo_con}.
	\end{proof}
%\hrule
\subsection{de Giorgi iteration with quantitative initial data}
The second lemma we prove is a de Giorgi iteration involving quantitative initial data.
\begin{lemma}\label{lemma3.3}
	Let $u$ be a weak solution of \cref{maineq}. Given $\de_x, \varepsilon \in (0,1)$, there exists $\nu_1 = \nu_1(\datanb{})\in (0,1)$ such that if we take $\theta_x = \de_x (\varepsilon \bsom)^{\frac{\mfd -2}{sp}}$ and  $\de_t := \nu_1 \de_x^{sp}$ and suppose for some time level $t_o$, the following assumptions are satisfied:
	\begin{equation*} 
		\pm(\bsmu^{\pm}-u(\cdot,t_o))\geq \varepsilon \bsom \txt{on} B_{\theta_x c_o\varrho},
	\end{equation*}
	holds along with the assumption 
	\begin{equation*}%\label{vwLm:3:3:hypothesis}
		c_o^{\frac{sp}{p-1}}\tail((u-\bsmu_-)_-, \theta_x\varrho,0,(t_o,\rint))\leq \varepsilon\bsom,
	\end{equation*}
where $\rint := t_o + \de_t (\ve \bsom )^{-p}(c_o\varrho)^{sp}$,
	then the following conclusion follows:
	\begin{equation*}
		\pm(\bsmu^{\pm}-u)\geq\tfrac{1}2\varepsilon \bsom
		\quad
		\mbox{ on }\quad  B_{\theta_x\frac{c_o\varrho}2} \times (t_o, t_o+\nu_1 \de_x^{sp}(\ve \bsom)^{\mfd-p}(c_o\varrho)^{sp}].
	\end{equation*}
%where $\mopt = $.

%	If we denote $\Gamma := \lbr \tfrac{1}{\de_x^{sp}} + \tfrac{1}{\de_t}\rbr^{\frac{n+sp}{p(n+2s)}} (\de_x^n\de_t)^{\frac{s}{(n+2s)}}$, then the relation between $\nu$ and $\Gamma$ takes the form $ \nu \approx \Gamma^{-\frac{n+2s}{s}}$ or in particular, $\nu \approx \lbr \tfrac{1}{\de_x^{sp}} + \tfrac{1}{\de_t}\rbr^{-\frac{n+sp}{sp}} (\de_x^n\de_t)^{-1}$.
\end{lemma}
\begin{proof}
	 For $j=0,1,\ldots$, define
	\begin{equation*}%\label{choices:B_n}
		%			\left\{
		{\def\arraystretch{1.1}\begin{array}{cccc}
				k_j:=\bsmu^-+\tfrac{\varepsilon \bsom}2+\tfrac{\varepsilon \bsom}{2^{j+1}},& \tilde{k}_j:=\tfrac{k_j+k_{j+1}}2,&
				\varrho_j:=\tfrac{c_o\varrho}2+\tfrac{c_o\varrho}{2^{j+1}},
				&\tilde{\varrho}_j:=\tfrac{\varrho_j+\varrho_{j+1}}2,\\%[5pt]
				B_j:=B_{\varrho_j}^{\theta_x},& \tilde{B}_j:=B_{\tilde{\varrho}_j}^{\theta_x},&
				\hat{\varrho}_j:=\tfrac{3\varrho_j+\varrho_{j+1}}4, &
				\bar{\varrho}_j:=\tfrac{\varrho_j+3\varrho_{j+1}}4.
		\end{array}}
		%			\right.
	\end{equation*}
%	Furthermore, we also define 
%	\begin{align}\label{choices:B_nn}
%		%	\left\{
%		%	{\def\arraystretch{1.1}\begin{array}{cccc}
%		\hat{\varrho}_j:=\tfrac{3\varrho_j+\varrho_{j+1}}4, \quad 
%		\bar{\varrho}_j:=\tfrac{\varrho_j+3\varrho_{j+1}}4,\quad 
%		\hat{\mcQ}_j:=\mcQ_{\hat\varrho_j}^{\theta_x,\theta_t}, \quad
%		\bar{\mcQ}_j:=\mcQ_{\bar\varrho_j}^{\theta_x,\theta_t}.
%		%	\end{array}}
%		%	\right.
%	\end{align}
	We now consider a cutoff functions $\bar{\zeta_j}(x)$ and $\zeta_j(x)$ independent of time such that
	\begin{equation*}%\label{cutoff_size}
		\begin{array}{c}
			\bar{\zeta_j} \equiv 1 \text{ on } B_{j+1}, \quad \bar{\zeta_j} \in C_c^{\infty}(\bar{B}_{j}), \quad  |\nabla\bar{\zeta_j}| \apprle \frac{1}{\theta_x(\bar{\varrho_j} - \varrho_{j+1})} \approx \frac{2^j}{\theta_xc_o\varrho}, \\
			{\zeta_j} \equiv 1 \text{ on } \tilde{B}_{j}, \quad {\zeta_j} \in C_c^{\infty}(\hat{B}_{j}), \quad  |\nabla{\zeta_j}| \apprle \frac{1}{\theta_x(\hat{\varrho_n} - \tilde{\varrho_{j}})}\approx \frac{2^{j}}{\theta_xc_o\varrho}.
		\end{array}
	\end{equation*}
%\hrule
We will apply \cref{energyqequal2} over cylinders of the form $B_j \times (t_o,\rint)$ and denote $\rint := t_o + \tht_t (\ve \bsom )^{-p}(c_o\varrho)^{sp}$ (note that the time intervals do not change in the iteration). 
	Let us estimate each of the terms appearing on the right hand side of \cref{energyqequal2} as follows:
	\begin{description}[leftmargin=*]
		\item[Estimate for the first term:] Since $(u-k_j)-_ \leq \ve \bsom$, we get 
		\begin{multline*}
			\int_{t_o}^{\rint}\iint_{B_j \times B_j}\hspace*{-0.7cm} \frac{\max\{(u-k_j)_{-}(x,t),(u-k_j)_{-}(y,t)\}^{p}|\zeta(x)-\zeta(y)|^p}{|x-y|^{n+sp}}\,dx\,dy\,dt\\
			\begin{array}{rcl}
				& \leq & (\ve \bsom)^p  \lbr \frac{2^{j}}{\theta_x(c_o\varrho)} \rbr^p \int_{t_o}^{\rint} \int_{B_j}\int_{B_{j}} \frac{\lsb{\chi}{\{u(x,t) < k_j\}}}{|x-y|^{n+(s-1)p}}\,dx\,dy\,dt\\
				& \leq & \mopC (\ve \bsom)^p  \lbr \frac{2^{j}}{\theta_x(c_o\varrho)} \rbr^p \int_{t_o}^{\rint} \int_{B_j}\int_{4B_{j}} \frac{\lsb{\chi}{\{u(x,t) < k_j\}}}{|y|^{n+(s-1)p}}\,dy\,dx\,dt\\
				& \leq & \mopC (\ve \bsom)^p   \frac{2^{pj}}{(\theta_x\varrho_j)^{sp}} |A_j| \approx \mopC (\ve \bsom)^p   \frac{2^{pj}}{\theta_x^{sp}(c_o\varrho)^{sp}} |A_j|,
			\end{array}
		\end{multline*}
		where we have denoted $A_j := \{u(x,t) \leq k_j\} \cap (B_j \times (t_o,\rint))$.
		
		\item[Estimate for the second term:] This term is zero since $\pa_t\zeta_j = 0$.
%		$u \geq \bsmu_- + \ve \bsom$ This is easily estimated as follows:
%		\begin{equation}
%			\begin{array}{rcl}
%				\iint_{\mcQ_j}(u-k_j)_{-}^2|\partial_t\zeta_j^p| \,dx\,dt & \leq & (\ve \bsom)^2 |A_j| \frac{2^{jsp}}{\theta_t (c_o \varrho)^{sp}}
%			\end{array}
%		\end{equation}
		\item[Estimate for the third term:] This term is zero since the hypothesis says $u \geq \bsmu_-+\ve\bsom$ at the initial time level $t=t_o$ and hence $(u(\cdot,t_o)-k_j)_-=0$. 
		\item[Estimate for the fourth term:] Since $x \in \spt \zeta_j \Longrightarrow |x| \leq \theta_x \hat{\varrho_j}$ and $|y| \geq \theta_x\varrho_j$, we get
		\[
		\frac{|y-x|}{|y|} \geq \frac{\varrho_j - \hat{\varrho_j}}{\hat{\varrho_j}} = \frac14 \lbr  \frac{\varrho_j - \varrho_{j+1}}{\varrho_j}\rbr \geq \frac{1}{2^{j+4}},
		\]
		which allows us to estimate as follows:
		\begin{multline*}
			\lbr  \underset{\stackrel{t \in (t_o,\rint)}{x\in \spt \zeta_j}}{\esssup}\,\int_{\RR^n \setminus B_j}\frac{(u-k_j)_{-}^{p-1}(y,t)}{|x-y|^{n+sp}}\,dy\rbr\iint_{B_j\times (t_o,\rint)}\hspace*{-0.6cm} (u-k_j)_{-}(x,t)\zeta_j^p(x)\,dx\,dt\\
			\begin{array}{rcl}
				& \leq & \mopC 2^{(n+sp)j}(\ve \bsom) |A_j| \lbr\underset{t \in (t_o,\rint)}{\esssup}\,\int_{\RR^n \setminus B_j}\frac{(\ve\bsom + (\bsmu_--u)_+ )^{p-1}(y,t)}{|y|^{n+sp}}\,dy \rbr\\
				& \leq & \mopC 2^{(n+sp)j}(\ve \bsom) |A_j|  \lbr\frac{(\ve \bsom)^{p-1}}{(\theta_x\varrho_j)^{sp}}+\underset{t \in (t_o,\rint)}{\esssup}\,\int_{\RR^n \setminus B_j}\frac{(u -\bsmu_- )_-^{p-1}(y,t)}{|y|^{n+sp}}\,dy \rbr\\
				& = & \mopC 2^{(n+sp)j}(\ve \bsom) |A_j|  \lbr\frac{(\ve \bsom)^{p-1}}{(\theta_x\varrho_j)^{sp}}+\frac{(\theta_x \varrho_j)^{sp-sp}}{(\theta_x \varrho)^{sp}}\underset{t \in (t_o,\rint)}{\esssup} (\theta_x \varrho)^{sp}\int_{\RR^n \setminus B_j}\frac{(u -\bsmu_- )_-^{p-1}(y,t)}{|y|^{n+sp}}\,dy \rbr\\
				& \leq & \mopC 2^{(n+sp)j}(\ve \bsom) |A_j|  \lbr\frac{(\ve \bsom)^{p-1}}{(\theta_x\varrho_j)^{sp}}+\frac{(\theta_x \varrho_j)^{sp-sp}}{(\theta_x \varrho)^{sp}}\underset{t \in (t_o,\rint)}{\esssup} (\theta_x \varrho)^{sp}\int_{\RR^n \setminus B_j}\frac{(u -\bsmu_- )_-^{p-1}(y,t)}{|y|^{n+sp}}\,dy \rbr\\
				& = & \mopC 2^{(n+sp)j}(\ve \bsom) |A_j|  \lbr\frac{(\ve \bsom)^{p-1}}{(\theta_x\varrho_j)^{sp}}+\frac{c_o^{sp}\tail((u-\bsmu_-)_-, \theta_x\varrho,0,(t_o,\rint))^{p-1} }{(\theta_x\varrho_j)^{sp}}\rbr\\
				& \leq & \mopC 2^{(n+sp)j}(\ve \bsom) |A_j|  \frac{(\ve \bsom)^{p-1}}{(\theta_x\varrho_j)^{sp}} \approx  \mopC 2^{(n+sp)j}(\ve \bsom) |A_j|  \frac{(\ve \bsom)^{p-1}}{\theta_x^{sp}(c_o\varrho)^{sp}}.
			\end{array}
		\end{multline*}
	\end{description}
	Combining the previous two estimates along with the observations, we have
	\begin{equation}\label{Eq:3.6}
		%		\begin{aligned}
		\underset{(t_o,\rint)}{\esssup}
		\int_{\tilde{B}_j} (u-{k}_j)_-^2\,dx
		+\iiint_{\tilde{\mcQ}_j}\frac{|(u-{k}_j)_{-}(x,t)-(u-{k}_j)_{-}|^p}{|x-y|^{n+sp}}\,dx \,dy\,dt
				\leq
		\mopC 2^{(n+sp)j} |A_j| \lbr  \frac{(\ve \bsom)^{p}}{\theta_x^{sp}(c_o\varrho)^{sp}}\rbr,
		%		\end{aligned}
	\end{equation}
where we have denoted $\tilde\mcQ_j = \tilde{B}_j  \times (t_o,\rint)$. 
	%\hrule
	Using Young's inequality, we have
	\begin{multline}\label{Eq:3.7}
		|(u-{k}_{j})_- \bar{\zeta_j}(x,t) - (u-{k}_{j})_- \bar{\zeta_j}(y,t)|^p \leq c |(u-{k}_{j})_- (x,t) - (u-{k}_{j})_- (y,t)|^p\bar{\zeta_j}^p(x,t) \\
		+ c |(u-{k}_{j})_-(y,t)|^p |\bar{\zeta_j}(x,t) - \bar{\zeta_j}(y,t)|^p,
	\end{multline}
	from which we obtain the following sequence of estimates:
	\begin{equation*}
		\begin{array}{rcl}
			\frac{\ve \bsom}{2^{j+2}}
			|A_{j+1}|
			&\overred{3.8a}{a}{\leq} &
			\iint_{\tilde\mcQ_{j}}(u-{k}_j)_-\bar{\zeta_j}\,dx\,dt \\
			&\overred{3.8b}{b}{\leq} &
			\lbr\iint_{\tilde\mcQ_{j}}\left[(u-{k}_j)_-\bar{\zeta_j}\right]^{p\frac{n+2s}{n}}
			\,dx\,dt\rbr^{\frac{n}{p(n+2s)}}|A_j|^{1-\frac{n}{p(n+2s)}}\\
			&\overred{3.8c}{c}{\leq} &\mopC 
			\left(\iiint_{\tilde\mcQ_j}\frac{|(u-{k}_j)_{-}(x,t)\bar{\zeta_j}(x,t)-(u-{k}_j)_{-}\bar{\zeta_j}(y,t)|^p}{|x-y|^{n+sp}}\,dx \,dy\,dt\right)^{\frac{n}{p(n+2s)}} 
			\\&&\qquad\times \left(\underset{t\in (t_o,\rint)}{\esssup}\int_{\tilde{B}_j}[(u-\tilde{k}_j)_{-}\bar\zeta_j(x,t)]^2\,dx\right)^{\frac{s}{n+2s}}|A_j|^{1-\frac{n}{p(n+2s)}}\\
			&\overred{3.8d}{d}{\leq}&
			\mopC b_o^j   \lbr  \frac{(\ve \bsom)^{p}}{\theta_x^{sp}(c_o\varrho)^{sp}} \rbr^{\frac{n+sp}{p(n+2s)}}
			|A_j|^{1+\frac{s}{n+2s}} \\
			&\overred{3.8e}{e}{\leq} &
			\mopC
			\frac{b_o^j}{(c_o\varrho)^\frac{s(n+sp)}{n+2s}}(\ve\bsom)^{(p+2-\mfd){\frac{n+sp}{p(n+2s)}}}
			\lbr \tfrac{1}{\de_x^{sp}}\rbr^{\frac{n+sp}{p(n+2s)}}|A_j|^{1+\frac{s}{n+2s}}, 
		\end{array}
	\end{equation*}
	where to obtain \redref{3.8a}{a}, we made use of the observations and enlarged the domain of integration; to obtain \redref{3.8b}{b}, we applied H\"older's inequality; to obtain \redref{3.8c}{c}, we applied \cref{fracpoin}; to obtain \redref{3.8d}{d}, we made use of \cref{Eq:3.6} along with \cref{Eq:3.7} with $\mopC = \mopC_{\data{}}$ and finally we collected all the terms to obtain \redref{3.8e}{e}, where $b_o = b_o(\datanb{})\geq 1$ is a  constant. Setting
	$\bsy_j=|A_j|/|\mcQ_j|$  and noting $|\mcQ_{i+1}| \approx |\mcQ_j| \approx \de_x^n \de_t (\ve \bsom)^{\frac{n(\mfd-2)}{sp}+\mfd -p} (c_o\varrho)^{n+sp}$ where we have denoted $\mcQ_j := B_{j} \times (t_o,\rint)$,  we get
	\[
	\frac{\ve \bsom}{2^{j+2}} \bsy_{j+1} \leq \mopC  \bsy_j^{1+\frac{s}{n+2s}}\frac{b_o^j}{(c_o\varrho)^\frac{s(n+sp)}{n+2s}}(\ve\bsom)^{(p+2-\mfd){\frac{n+sp}{p(n+2s)}}}
	\lbr \tfrac{1}{\de_x^{sp}} + \tfrac{1}{\de_t}\rbr^{\frac{n+sp}{p(n+2s)}} |\mcQ_{j}|^{\frac{s}{n+2s}}.
	\]
	Simplifying this expression noting that $(p+2-\mfd)\lbr \tfrac{n+sp}{p(n+2s)}\rbr + \lbr \tfrac{s}{n+2s}\rbr \lbr \tfrac{n(\mfd-2)}{sp} + \mfd - p \rbr = 1$,  we get
	\begin{equation*}
		\bsy_{j+1}
		\le
		\bsc \Gamma  \boldsymbol b^j  \bsy_j^{1+\frac{s}{n+2s}},
	\end{equation*}
	where $\Gamma := \lbr \tfrac{1}{\de_x^{sp}}\rbr^{\frac{n+sp}{p(n+2s)}} (\de_x^n\de_t)^{\frac{s}{(n+2s)}}$ and   $\bsc  = \bsc_{\data{}} \geq 1$,  $\boldsymbol b\geq  1$ are two constants depending only on the data. We apply  \cref{geo_con} to see if $\bsy_0 \leq \bsc^{-1/\alpha}\Gamma^{-1/\alpha} \boldsymbol b^{1/\alpha^2}$ with $\alpha = \tfrac{s}{n+2s}$, then $\bsy_{\infty} = 0$ which is the desired conclusion.  We make the choice of $\de_t$ to satisfy
	\[
	1 = \bsc^{-1/\alpha}\Gamma^{-1/\alpha} \boldsymbol b^{1/\alpha^2} = \bsc^{-1/\alpha}\boldsymbol b^{1/\alpha^2} \de_x^{sp} \de_t^{-1} \quad \Longrightarrow\quad  \de_t = \bsc^{-1/\alpha}\boldsymbol b^{1/\alpha^2} \de_x^{sp} =: \nu_1 \de_x^{sp},
	\]
	which completes the proof of the lemma.
\end{proof}
\subsection{Propagation of measure information forward in time}

%\hrule

\begin{lemma}\label{lemma3.4}
	Let $\alpha \in(0,1)$ be given, then there exist constants $\de = \de(\alpha) \in (0,1)$ and $\bar\ve  = \bar\ve(n,p,s,\La,\alpha)\in (0,1)$,
	 such that whenever $u$ satisfies
	\begin{equation*}
		\left|\left\{
		\pm\lbr\bsmu^{\pm}-u(\cdot, t_o)\rbr\geq \ve\bsom
		\right\}\cap B_{\theta_xc_o\varrho}(x_o)\right|
		\geq\alpha \left|B_{\theta_xc_o\varrho}\right|,
	\end{equation*}
	and the bound  
	\begin{equation*}%\label{sec:exptime:hyp}
		c_o^{\frac{sp}{p-1}}\tail((u-\bsmu^{\pm})_{\pm};\theta_x\varrho,0,(t_o,t_o+\bar\ve \de_x^{sp}(\ve\bsom)^{\mfd-p}(c_o\varrho)^{sp}]) \leq  \ve\bsom,
	\end{equation*}
	then the following conclusion follows:
	\begin{equation*}%\label{Eq:3:9}
		|\{
		\pm\left(\bsmu^{\pm}-u(\cdot, t)\right)\geq \de\ve \bsom\} \cap B_{\theta_xc_o\varrho}|
		\geq\tfrac{\alpha}2 |B_{\theta_xc_o\varrho}|
		\quad \text{ for all } \,  t\in(t_o,t_o + \bar\ve \de_x^{sp}(\ve\bsom)^{\mfd-p}(c_o\varrho)^{sp}].
	\end{equation*}
\end{lemma}
\begin{proof} 
	We prove the case of supersolutions only because the case for subsolutions is analogous. We set
	\[
	A_{k}^{\varrho}(t) = \{u(\cdot,t)  - \bsmu^-< k\} \cap B_{\theta_xc_o\varrho}.
	\]
	By hypothesis, we have 
	\[
	|A_{\ve\bsom}^{\varrho}(t_o)|\leq (1-\alpha)|B_{\theta_xc_o\varrho}|.
	\]
	We apply the energy estimate  from \cref{energyqequal2} for $(u-k)_{-}$ over the cylinder $B_{\theta_xc_o\varrho}\times(t_o,\mft]$  with $k=\bsmu^-+\ve\bsom$ where $\delta>0$ will be chosen later in the proof. Note that $(u-k)_{-} \leq \ve\bsom$ in $B_{\theta_xc_o\varrho}$. For $\sigma \in (0,\tfrac18]$ to be chosen later, we take a cutoff function $\zeta = \zeta(x)\geq 0$,  such that it is supported in $B_{\varrho(1-\frac{\sigma}{2})}$ with $\zeta \equiv 1$ on $B_{(1-\sigma)\theta_xc_o\varrho}$ and $|\nabla \zeta| \apprle \tfrac{1}{\sigma\theta_xc_o\varrho}$. We now make use of  \cref{energyqequal2} over $(t_o,\mft)$ and estimate analogously as \cref{lemma3.3} to get
	\begin{equation}\label{Eq:3.10}
			\begin{array}{rcl}
		\int_{B_{(1-\sigma)\theta_xc_o\varrho}\times\{\mft\}}(u-k)_-^2 \,dx 
		&\leq&   \mopC\frac{ (\ve\bsom)^p}{\sigma^p \theta_x^{sp}(c_o\varrho)^{sp}}|\mft-t_o||B_{\theta_xc_o\varrho}|
		+ \int_{B_{\theta_xc_o\varrho}\times\{t_o\}}(u-k)_-^2\zeta^p\,dx\\
		& \leq & \mopC\frac{ (\ve\bsom)^p}{\sigma^p \theta_x^{sp}(c_o\varrho)^{sp}}|\mft-t_o||B_{\theta_xc_o\varrho}|
		+ (\ve\bsom)^2|A_{\ve\bsom}^{t_o}|\\
		& \leq & \mopC\frac{ (\ve\bsom)^p}{\sigma^p \theta_x^{sp}(c_o\varrho)^{sp}}|\mft-t_o||B_{\theta_xc_o\varrho}|
		+ (\ve\bsom)^2(1-\alpha)|B_{\theta_xc_o\varrho}|,
		    \end{array}
	\end{equation}
	where $\mft >t_o$ is any time level. 
	Setting $k_\de:=\bsmu^-+\de \ve\bsom$ for some $ \de \in (0,\tfrac12)$ to be chosen later, we estimate the term on the left hand side of \cref{Eq:3.10} from below to get
	\begin{align*}
		\int_{B_{(1-\sigma)\theta_xc_o\varrho}\times\{\mft\}}(u-k)_-^2 \,dx
		\geq \int_{B_{(1-\sigma)\theta_xc_o\varrho}\times\{\mft\} \cap\{u<k_{\de}\}}\hspace*{-1cm}(u-k)_-^2 \,dx \geq
		(1-\de)^2(\ve\bsom)^2|A_{k_\de}^{(1-\sigma)\varrho}(\mft)|.
	\end{align*}
%	Making use of the bound 
%	$\tfrac{1}{2}M\leq(1-\epsilon)M=k-k_\epsilon\leq |k_\epsilon|+|k|\leq 2(|\bsmu^-|+ M)\leq 18 M$, we further get
%	\begin{align}\label{prop_1}
%		\int_{k_\epsilon}^k|s|^{p-2}(s-k)_-\,ds
%		=
%		\tfrac{1}{p-1}\, \mathfrak g_-(k_\epsilon,k)
%		\geq 
%		\tfrac{1}{\bsc  (p)}
%		\lbr|k_\epsilon|+|k|\rbr^{p-2} (k-k_\epsilon)^2
%		\geq
%		\tfrac{1}{\bsc  (p)} M^p.
%	\end{align}
	Next, we note that
	\begin{align*}
		|A_{k_\de}^{\varrho}(\mft)| \leq 
		|A_{k_\de}^{(1-\sigma)\varrho}(\mft)|+n\sigma |B_{\theta_xc_o\varrho}|.
	\end{align*}
	Combining  all the above estimates gives
	\begin{align*}
		|\{-(\bsmu^- -u(\cdot,t))< \de\ve\bsom\} \cap B_{\theta_xc_o\varrho}|
		\leq 
		\lbr\frac{1-\al}{(1-\de)^2}+ \mopC \frac{(\ve\bsom)^{p}|\mft-t_o|}{\sigma^p \de_x^{sp} (\ve\bsom)^{\mfd}(1-\de)^2(c_o\varrho)^{sp}} +n\sigma\rbr |B_{\theta_xc_o\varrho}|,
	\end{align*}
%\hrule
	for a universal constant $\mopC  > 0$. We now choose $\de\in (0,1)$ and $\sigma$ small enough such that
	\begin{equation*}
		\frac{1-\alpha}{(1- \de)^2}\leq {1-\frac{3\alpha}{4}} \txt{and} \sigma = \frac{\alpha}{8n}.
	\end{equation*}
With these choices, let us take 
\[
|\mft- t_o| \leq \lbr\frac{\alpha}{8}\rbr \lbr\frac{\sigma^p \de_x^{sp}(1-\de)^2 (\ve\bsom)^{\mfd-p}(c_o\varrho)^{sp}}{\mopC}\rbr =: \bar\ve \de_x^{sp} (\ve\bsom)^{\mfd-p}(c_o\varrho)^{sp},
\]
	and make the choice $\sigma=\tfrac{\alpha}{8n}$.
	Finally, we choose $\delta\in (0,1)$ small enough so that $\tfrac{\bsc\delta}{\sigma^p}\leq\tfrac{\alpha}{8}$, which gives the required conclusion.
\end{proof}
\subsection{Measure shrinking lemma}
%\hrule2
\begin{lemma}\label{lemma3.5}
	Let $\de_x, \de_t$ be given and assume for some $\sigma \in (0,1)$, the following is satisfied:
	\begin{equation*}
		\left|\left\{
		\pm\lbr\bsmu^{\pm}-u(\cdot, t)\rbr\geq \ve\bsom
		\right\}\cap B_{\theta_xc_o\varrho}(x_o)\right|
		\geq\alpha \left|B_{\theta_xc_o\varrho}\right| \txt{for all} t \in (t_o, t_o + \theta_t (c_o\varrho)^{sp}),
	\end{equation*}
and 
	\begin{equation*}%\label{Lm:3:2:hyp}
		c_o^{\frac{sp}{p-1}}\tail((u-\bsmu^{\pm})_{\pm};2\theta_x\varrho,0,(t_o,t_o + \theta_t (c_o\varrho)^{sp}]) \leq \sigma \ve \bsom,
	\end{equation*}
where $\theta_x = \de_x (\sigma \ve\bsom)^{\frac{\mfd-2}{sp}}$ and $\theta_t = \de_t (\sigma \ve\bsom)^{\mfd-p}$.
	Then there exists $\mopC >0$ depending only on data,  we have 
	\begin{equation*}
		|\{\pm(\bsmu^{\pm}-u) \leq \tfrac12 \sigma \ve \bsom\} \cap \mcQ| \leq \mopC \lbr \tfrac{\de_x^{sp}}{\de_t}+1\rbr \tfrac{\sigma^{p-1}}{\al (1-\sigma)^{p-1}} |\mcQ|,
	\end{equation*}
where we have denoted $\mcQ = B_{\theta_x c_o \varrho} \times (t_o, t_o + \theta_t (c_o\varrho)^{sp})$.  
\end{lemma}
\begin{proof}
	We prove the case of super-solutions only because the case for sub-solutions is analogous.  We write the energy estimate \cref{energyqequal2} over the cylinder $B_{\theta_x c_o\varrho}\times (t_o,\mft]$ with $\mft = t_o + \theta_t (c_o\varrho)^{sp}$, for the functions $(u-k)_-$ where $k = \bsmu^-+\sigma \ve \bsom$.
	We choose a test function $\zeta = \zeta(x)$ such that $\zeta\equiv 1$ on $B_{\theta_xc_o\varrho}$, it is supported in $B_{\theta_x\frac32 c_o\varrho}$ and $|\nabla \zeta|\leq \frac{1}{\theta_xc_o\varrho}$.

	Using $(u-k)_{-} \leq \sigma \ve \bsom$ locally and  by our choice of the test function, applying \cref{energyqequal2}, we get
	%\hrule
	\begin{multline*}
		\int_{t_o}^{\mft}\int_{B_{\theta_xc_o\varrho}}|(u-k)_-(x,t)|\int_{B_{\theta_xc_o\varrho}}\frac{|(u-k)_+(y,t)|^{p-1}}{|x-y|^{n+sp}}\,dx\,dy\,dt
		\\
		\begin{array}{rcl}
		&\leq&  \mopC (\sigma\ve \bsom)^p   \frac{1}{\theta_x^{sp}(c_o\varrho)^{sp}} |B_{\theta_xc_o\varrho}||\mft-t_o|
		+\mopC (\sigma \ve\bsom)^2 |B_{\theta_xc_o\varrho}|\\
		& \leq &  \mopC (\sigma\ve \bsom)^{2}   \tfrac{\de_t}{\de_x^{sp}} |B_{\theta_xc_o\varrho}|
		+\mopC (\sigma \ve\bsom)^2 |B_{\theta_xc_o\varrho}|.
		\end{array}
%		\underset{1.1ackrel{t \in (-\delta\varrho^{sp},0]}{ x\in \spt \zeta}}{\esssup}\int_{B_{8\varrho}^c}\frac{(u-k_j)_{-}(y,t)}{|x-y|^{n+sp}}\,dy
%		+ C\int_{B_{8\varrho}\times\{-\de \varrho^{sp}\}} \mathfrak g_- (u,k_j) \,dx,
	\end{multline*}
	where $\mopC>0$ is a universal constant.
	We now invoke \cref{lem:shrinking} with $l = \bsmu^- + \sigma \ve\bsom$, $k = \bsmu^- + \tfrac12 \sigma \ve\bsom$ and $m = \bsmu^- + \ve\bsom$ to get
%	\hrule
	\[
	(\tfrac12 \sigma \ve \bsom)(\ve\bsom)^{p-1}(1-\sigma)^{p-1} \frac{\al |B_{\theta_x c_o\varrho}|}{(\theta_xc_o\varrho)^{n+sp}} |\{-(\bsmu^--u) \leq \tfrac12 \sigma \ve \bsom\} \cap \mcQ| \leq \mopC (\sigma\ve \bsom)^{2}   \lbr 1+\tfrac{\de_t}{\de_x^{sp}}\rbr |B_{\theta_xc_o\varrho}|.
	\]
	In particular, we get
	\[
	\al \tfrac{\sigma}{2} (\ve\bsom)^p (1-\sigma)^{p-1} \frac{|B_{\theta_x c_o\varrho}|}{|\mcQ|} \frac{\theta_t}{\theta_x^{sp}}
	|\{-(\bsmu^--u) \leq \tfrac12 \sigma \ve \bsom\} \cap \mcQ| \leq \mopC (\sigma\ve \bsom)^{2}   \lbr 1+\tfrac{\de_t}{\de_x^{sp}}\rbr |B_{\theta_xc_o\varrho}|,
	\]
	which becomes
	\[
	\al \tfrac{\sigma}{2} (\ve\bsom)^p (1-\sigma)^{p-1} \frac{|B_{\theta_x c_o\varrho}|}{|\mcQ|} \frac{\de_t}{\de_x^{sp}}(\sigma \ve \bsom)^{2-p}
	|\{-(\bsmu^--u) \leq \tfrac12 \sigma \ve \bsom\} \cap \mcQ| \leq \mopC (\sigma\ve \bsom)^{2}   \lbr 1+\tfrac{\de_t}{\de_x^{sp}}\rbr |B_{\theta_xc_o\varrho}|.
	\]
	This gives
	\[
	|\{-(\bsmu^--u) \leq \tfrac12 \sigma \ve \bsom\} \cap \mcQ| \leq \mopC \lbr \tfrac{\de_x^{sp}}{\de_t}+1\rbr \tfrac{\sigma^{p-1}}{\al (1-\sigma)^{p-1}} |\mcQ|,
	\]
	holds for a constant $\mopC>0$ depending only on data. 
\end{proof}
\begin{remark}
	In \cref{lemma3.2}, \cref{lemma3.3}, \cref{lemma3.4} and \cref{lemma3.5}, if we make the choice $\de_t = \de_x^{sp}$, then all the constants will become independent of $\de_x$ and $\de_t$. 
\end{remark}
\subsection{Reduction of oscillation when \texorpdfstring{$2<p<\infty$}.}
\begin{definition}\label{defd}
	We will choose $\mfd$ such that $1 < \mfd \leq \min\{2,p\}$.  %and $p+\mfd -3 > 0$.  
\end{definition}

Let us consider the reference cylinder $\mcQ_o = B_{\varrho^{1-\epsilon}} \times (-\varrho^{s},0)$ for some fixed $\epsilon \in (0,1)$ which is the reference cylinder, see \cref{figtwoalt}. Let us take constants
\begin{equation*}
	\bsmu^+\geq  \esssup_{\mcQ_o}u, \qquad \bsmu^-\leq \essinf_{\mcQ_o} u \qquad \text{ and } 
	\qquad \bsom \geq \bsmu^+ - \bsmu^-.	
\end{equation*}
For some $\mbfa \gg 1$, let us denote $\theta:= \lbr \tfrac{\bsom}{\mathbf{a}}\rbr$ and consider the cylinder 
\[
\mcQ_{\mbfa}^{\de_x,\de_t} = B_{\de_x \lbr \tfrac{\bsom}{\mathbf{a}}\rbr^{\frac{\mfd-2}{sp}} (c_o\varrho)} \times (- \de_t \lbr \tfrac{\bsom}{\mathbf{a}}\rbr^{\mfd-p} (c_o\varrho)^{sp},0].
\]
\begin{claim}\label{Claim1}
	The following inclusion holds $\mcQ_{\mbfa}^{\de_x,\de_t} \subset \mcQ_o$.
\end{claim}
\begin{proof}
	[Proof of \cref{Claim1}] We need to show $\de_x \lbr \tfrac{\bsom}{\mathbf{a}}\rbr^{\frac{\mfd-2}{sp}} (c_o\varrho) \leq \varrho^{1-\epsilon}$ and $\de_t \lbr \tfrac{\bsom}{\mathbf{a}}\rbr)^{\mfd-p} (c_o\varrho)^{sp} \leq \varrho^{s}$. Suppose the inclusion fails, then we would have either $c_o \de_x \varrho^{\ve} \geq \lbr \tfrac{\bsom}{\mathbf{a}}\rbr^{\frac{\mfd-2}{sp}}$ and/or $c_o^{sp} \de_t \varrho^{s(p-1)} \geq \lbr \tfrac{\bsom}{\mathbf{a}}\rbr^{p-\mfd}$ and in both cases, the oscillation is comparable to the radius. Hence the claim follows. 
\end{proof}
 \begin{assumption}\label{Tailassumption}
 	We further assume 
 	\[
 	\tail((u-\bsmu^\pm)_{\pm},\varrho^{1-\ep},0,(-\varrho^s,0)) \leq \bsom. 
 	\]
 \end{assumption}
%\hrule
\begin{definition}
%	Without loss of generality, we assume the following is satisfied:
%	\begin{equation}\label{degEq:mu-pm-}
%		\bsmu_v^+-\bsmu_v^->\frac12 \bsom_v.
%		%	, \mbox{ equivalently }\bsmu^+-\bsmu^->\frac12 \bsom.
%	\end{equation}
	For a constant $\nu$ to be eventually determined according to \descrefnormal{step1}{Step 1} in the proof of \cref{redoscfirstp>2scaled},  we have one of two alternatives:
	\begin{description}
		\descitemnormal{DegAlt-$\I$:}{alt1} Consider the cylinder $\mcQ_{\mbfa} = B_{\lbr \tfrac{\bsom}{\mathbf{a}}\rbr^{\frac{\mfd-2}{sp}} (c_o\varrho)} \times (-  \lbr \tfrac{\bsom}{\mathbf{a}}\rbr)^{\mfd-p} (c_o\varrho)^{sp},0]$. There exists some time level $\bar{t} \in (-  \lbr \tfrac{\bsom}{\mathbf{a}}\rbr)^{\mfd-p} (c_o\varrho)^{sp},0]$ such that the following holds:
		\begin{equation*}
			%		\label{firstaltdeg2}
			|\{u\leq \bsmu^-+\tfrac14 \bsom\}\cap
			\mcQ(\bar{t})|\leq \nu|\mcQ(\bar{t})|,
		\end{equation*}
		where we have denoted $\mcQ(\bar{t}) := B_{\lbr \tfrac{\bsom}{\mathbf{4}}\rbr^{\frac{\mfd-2}{sp}} (c_o\varrho)} \times ( \bar{t}- \lbr \tfrac{\bsom}{\mathbf{4}}\rbr^{\mfd-p} (c_o\varrho)^{sp},\bar{t}]$.
		\descitemnormal{DegAlt-$\II$:}{alt2} OR for every time level $\bar{t} \in (-  \lbr \tfrac{\bsom}{\mathbf{a}}\rbr)^{\mfd-p} (c_o\varrho)^{sp},0]$,  the following holds:
		\begin{equation*}
			%		\label{firstaltdeg2}
			|\{u\leq \bsmu^-+\tfrac14 \bsom\}\cap
			\mcQ(\bar{t})|\geq \nu|\mcQ(\bar{t})|,
		\end{equation*}
		where we have denoted $\mcQ(\bar{t}) := B_{\lbr \tfrac{\bsom}{\mathbf{4}}\rbr^{\frac{\mfd-2}{sp}} (c_o\varrho)} \times (\bar{t}- \lbr \tfrac{\bsom}{\mathbf{4}}\rbr^{\mfd-p} (c_o\varrho)^{sp},\bar{t}]$.
	\end{description}
\end{definition}
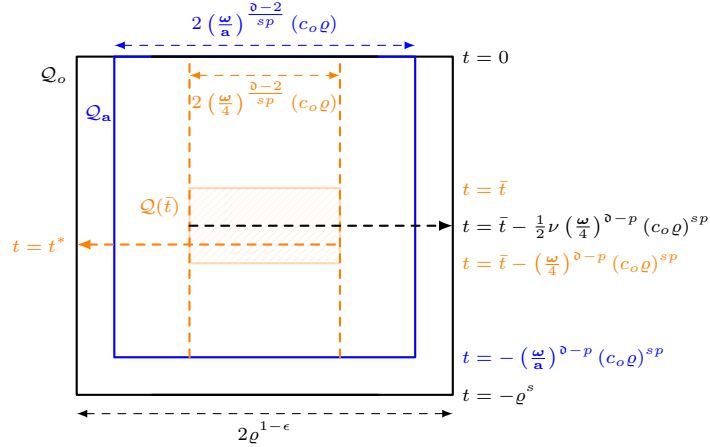
\begin{figure}[ht]
	%			\begin{figure}{.5\textwidth}
	\begin{center}
		\begin{tikzpicture}[line cap=round,line join=round,>=latex,scale=0.5]
			\coordinate  (O) at (0,0);
			
			\draw[thick, draw=black] (-5,5.5) rectangle (5,-3.5);
				\draw[draw=black, dashed, <->] (-5,-4) -- (5,-4);
			\node  at (0,-4.5) {\scriptsize $2\varrho^{1-\epsilon}$};
			\node  [anchor=west] at (5,-3.5) {\scriptsize \textcolor{black}{$t=-\varrho^s$}};
			\node  [anchor=west] at (5,5.5) {\scriptsize \textcolor{black}{$t=0$}};
			\draw[thick, draw=blue, pattern color=blue, pattern=north east lines, opacity=0.1] (-4,5.5) rectangle (4,-2.5);
			\draw[thick, draw=blue] (-4,5.5) rectangle (4,-2.5);
			\node [anchor=west] at (5,-2.5) {\color{blue}\scriptsize $t=- \lbr\tfrac{\bsom}{\mathbf{a}} \rbr^{\mfd-p}(c_o\varrho)^{sp}$};
			\draw[draw=blue, dashed, <->] (-4,5.8) -- (4,5.8);
			\node  at (0,6.5) {\color{blue}\scriptsize $2\lbr \tfrac{\bsom}{\mathbf{a}}\rbr^{\frac{\mfd-2}{sp}}(c_o\varrho)$};
			\draw[thick, draw=orange, dashed] (-2,-2.5) -- (-2,5.5);
			\draw[thick, draw=orange, dashed] (2,-2.5) -- (2,5.5);
			\draw[draw=orange, dashed, <->] (-2,5) -- (2,5);
			\node  at (0,4.4) {\color{orange}\scriptsize $2\lbr \tfrac{\bsom}{4}\rbr^{\frac{\mfd-2}{sp}}(c_o\varrho)$};
			
			\draw[thick, draw=orange, pattern color=orange,pattern=north east lines, opacity=0.3] (-2,0) rectangle (2,2);
			
			%			\draw[thick, draw=blue,pattern=north east lines, pattern color=blue, opacity=1] (-0.5,2) rectangle (0.5,1.5);
			\node  [anchor=west] at (5,2) {\scriptsize \textcolor{orange}{$t=\bar{t}$}};
			%			\draw[dotted,->] (1,2) -- (3,2);
			%			\draw[dotted,->] (1,1) -- (3,1);
			\node  [anchor=west] at (5,0) {\scriptsize \textcolor{orange}{$t=\bar{t}-\lbr \tfrac{\bsom}{4}\rbr^{\mfd-p}(c_o\varrho)^{sp}$}};
			%			\node [anchor=east] at (-1,1.5) {\scriptsize{$\mcq_{c_o\varrho}$}};
			%			\node [anchor=east] at (-3,0) {{$\mcq_o$}};

			\draw[thick, draw=black, dashed,->] (-2,1) -- (5,1);
			\node  [anchor=west] at (5,1) {\scriptsize \textcolor{black}{$t=\bar{t}-\tfrac12\nu\lbr \tfrac{\bsom}{4}\rbr^{\mfd-p}(c_o\varrho)^{sp}$}};
			
			\draw[thick, draw=orange, dashed,<-] (-5,0.5) -- (2,0.5);
			\node  [anchor=east] at (-5,0.5) {\scriptsize \textcolor{orange}{$t=t^{\ast}$}};
			%			\draw[thick, draw=orange, dashed] (1,-3.5) -- (1,6);
			\draw[draw=black,thick] (-3,-3.5) -- (3,-3.5);
			
			\draw[draw=black,thick] (-3,5.5) -- (3,5.5);
			\node [anchor=east] at (-5,5) {\scriptsize{$\mcQ_o$}};
			\node [anchor=east] at (-3.8,4) {\color{blue}\scriptsize{$\mcQ_{\mathbf{a}}$}};
		\node [anchor=east] at (-2,1.5) {\color{orange}\scriptsize{$\mcQ(\bar{t})$}};
			%			\node  at (0,-2.5) {\scriptsize $c_o\varrho$};
			
		\end{tikzpicture}
	\end{center}
	\caption{Defining the geometry}
	\label{figtwoalt}
\end{figure}

\subsection{Reduction of oscillation when \texorpdfstring{\descrefnormal{alt1}{DegAlt-$\I$}}. holds }\label{redoscfirstp>2scaled}

In this section, we work with $v$  near its infimum and the proof will proceed in several steps:
\begin{description}[leftmargin=*]
	\descitemnormal{Step 1:}{step1} Let us take $\de_x = \de_t=1$ and define $\nu := \nu_1$ with $\nu_1$ obtained from \cref{lemma3.2} applied with $\ve = \tfrac14$. Then we get
	\[
	u \geq \bsmu^- + \tfrac18 \bsom \txt{on} B_{\lbr \tfrac{\bsom}{\mathbf{4}}\rbr^{\frac{\mfd-2}{sp}} (\frac12c_o\varrho)} \times ( \bar{t}- \lbr \tfrac{\bsom}{\mathbf{4}}\rbr^{\mfd-p} (\tfrac12c_o\varrho)^{sp} , \bar{t}].
	\]
	To ensure the tail alternative holds, we estimate as follows:
	\[
	\begin{array}{rcl}
		c_o^{{sp}}\tailp((u-\bsmu_-)_-, \theta_x\varrho,0,(\bar{t}-\lbr \tfrac{\bsom}{\mathbf{4}}\rbr^{\mfd-p} (c_o\varrho)^{sp},\bar{t})) & \leq & c_o^{sp}\varrho^{\ep s p}\lbr\tfrac{\bsom}{\mathbf{4}}\rbr^{\mfd-2}\tailp((u-\bsmu_-)_-, \varrho^{1-\ep},0,(-\varrho^{sp},0))\\
		& \leq & \tfrac{1}{4^{\mfd-2}}c_o^{sp} \bsom^{p-1},
	\end{array}
	\]
	where we made use of the bound $\varrho^{\ep s p } \leq \bsom^{2-\mfd}$, since otherwise the oscillation is comparable to the radius along with \cref{Tailassumption}. Now we can choose $c_o$ small such that $\tfrac{1}{4^{\mfd-2}}c_o^{sp} \leq  (\tfrac18)^{p-1}$ to ensure the tail alternative is satisfied.

	\descitemnormal{Step 2:}{step2}  Let $\ve \in (0,\tfrac18)$ to be chosen and we take $t_o =\bar{t}- \lbr \tfrac{\bsom}{\mathbf{4}}\rbr^{\mfd-p} (\tfrac12c_o\varrho)^{sp}$ in \cref{lemma3.3}. Then we get 
	\begin{equation*}
		u\geq\bsmu^- + \tfrac{1}2\varepsilon \bsom
		\quad
		\mbox{ on }\quad  B_{\theta_x\frac{c_o\varrho}4} \times (t_o, t_o+\nu_1 (\tfrac{1}{4\ve})^{2-\mfd} (\ve \bsom)^{\mfd-p}(\tfrac{c_o\varrho}4)^{sp}].
	\end{equation*}
  
     We now choose $\ve$ small such that $\nu_1 (\tfrac{1}{4\ve})^{2-\mfd} (\ve \bsom)^{\mfd-p}(\tfrac{c_o\varrho}4)^{sp} \geq \lbr\tfrac{\bsom}{\mathbf{a}} \rbr^{\mfd-p}(c_o\varrho)^{sp}$, i.e., we take $\ve^{p-2} \leq \frac{4^{2-\mfd-sp}\nu_1 }{\mathbf{a}^{p-\mfd}}$. This ensures that the pointwise information is available till the top of the cylinder. With this choice of $\ve$, we further restrict $c_o$ such that $\tfrac{1}{4^{\mfd-2}}c_o^{sp} \leq \ve^{p-1} =\lbr \frac{4^{2-\mfd-sp}\nu_1 }{\mathbf{a}^{p-\mfd}}\rbr^{\frac{p-1}{p-2}}$ and this ensure the tail alternative is satisfied. 
     
     \begin{remark}
     	In trying to extend the degenerate proof to the singular case $p \in (1,2)$, we are unable to choose $\ve$ small enough in \descref{step2}{Step 2} to ensure $\ve^{p-2} \leq \frac{4^{2-\mfd-sp}\nu_1 }{\mathbf{a}^{p-\mfd}}$ holds. 
     \end{remark}
%	
%	
%	For some $\varepsilon_1 \in (0,\tfrac18)$ to be chosen according to \cref{eq6.21}, we  apply \cref{lemma5.45} with $t_o = \bar{t}$ noting that \cref{degeq7.10} satisfies the first hypothesis of \cref{lemma5.45}. Additionally, we apply \cref{lemma5.45} with $c_o\varrho$ replaced by $\tfrac12c_o\varrho$ to get
%	\begin{equation}\label{eq6.20}
%		v \geq \bsmu_v^-+\tfrac12 \varepsilon_1 \bsom_v \quad
%		\mbox{ on }\quad
%		B_{\frac{c_o\varrho}4} \times \left(\bar{t},\bar{t}+\bar\nu_3(\varepsilon_1\bsom_v)^{2-p}\lbr{\tfrac14c_o\varrho}\rbr^{sp}\right],
%	\end{equation}
%	provided $c_o = c_o(\datanb{,\varepsilon_1,\tau,\mreta})$ is chosen small enough to ensure the tail alternative in  \cref{vwLm:3:3:hypothesis2} is satisfied.
%	\item[Step 3:] From \descrefnormal{step2}{Step 2}, we choose $\varepsilon_1$ small such that $\bar\nu_3 (\varepsilon_1\bsom_v)^{2-p}\lbr{\tfrac14c_o\varrho}\rbr^{sp} \geq \lbr \tfrac{\bsom_v}{\mathbf{a}}\rbr^{2-p}(c_o\varrho)^{sp}$, which gives
%	\begin{equation}\label{eq6.21}
%		\varepsilon_1\leq \lbr\frac{\bar\nu_3}{4\mathbf{a}^{p-2}}\rbr^{\frac{1}{p-2}}.
%	\end{equation}
%	In particular, with this choice of $\varepsilon_1$, we have \cref{eq6.20} holds till $t=0$.
%	\item[Step 4:] Combining the estimates, we have 
%	\begin{equation*}
%		\essosc_{\mcq} v \leq \lbr 1-\tfrac12 \varepsilon_1\rbr \bsom_v \quad \text{where} \quad \mcq:= B_{\frac{c_o\varrho}4} \times \left(-\bar\nu_3(\varepsilon_1\bsom_v)^{2-p}\lbr{\tfrac14c_o\varrho}\rbr^{sp},0\right].
%	\end{equation*}
\end{description}

%\begin{remark}
%	Note that the constant $\mathbf{a}$ in \cref{rmk6.9} is yet to be determined and $c_o$ is to be further restricted in the subsequent steps of the proof. In application of \cref{vwLemma5.5}, we see that $c_o$ depends on $\varepsilon_1$ which in turn depends on $\mathbf{a}$ according to \cref{eq6.21}. We will keep track of the constants to ensure suitable choice of $\mathbf{a}$ and $c_o$, see \cref{degconstantsscaled}. 
%\end{remark}
This completes the proof of the reduction of oscillation when \descrefnormal{alt1}{DegAlt-$\I$} holds. 
%\hrule\hrule\hrule\hrule
\subsection{Reduction of oscillation when \texorpdfstring{\descrefnormal{alt2}{DegAlt-$\II$}}. holds}\label{redoscsecondp>2scaled}
In this section, we work with $u$  near its supremum and the proof will proceed in several steps:

\begin{description}%[leftmargin=*]
	\item[Step 1:] We see that \descrefnormal{alt2}{DegAlt-$\II$} can be rewritten as 
	\begin{equation*}
%		\label{secondaltdeg2rewrite}
		|\{u\leq \bsmu^+-\tfrac14 \bsom\}\cap \mcQ(\bar{t})|> \nu|\mcQ(\bar{t})|.
	\end{equation*}
	\item[Step 2:]  The first step involves finding a good time slice, see \cref{figtwoalt} for the geometry.  In particular, we have the following result:
	\begin{claim}\label{claim6.12}
		With notation as in \cref{figtwoalt}, there exists $\bar{t}-\lbr \tfrac{\bsom}{4}\rbr^{\mfd-p}(c_o\varrho)^{sp} \leq t^{\ast} \leq \bar{t}-\tfrac12 \nu\lbr \tfrac{\bsom}{4}\rbr^{\mfd-p}(c_o\varrho)^{sp}$ such that		\begin{equation*}%\label{secondaltdegtimeslice}
			|\{u(\cdot,t^{\ast})\leq \bsmu^+-\tfrac14 \bsom\}\cap B_{\lbr \tfrac{\bsom}{4}\rbr^{\frac{\mfd-2}{sp}}(c_o\varrho)}|\geq \tfrac12\nu|B_{\lbr \tfrac{\bsom}{4}\rbr^{\frac{\mfd-2}{sp}}(c_o\varrho)}|.
		\end{equation*}
	\end{claim}
	\begin{proof}
		Suppose the claim is false, then we have for all $t^{\ast} \in (\bar{t}-\lbr \tfrac{\bsom}{4}\rbr^{\mfd-p}(c_o\varrho)^{sp}, \bar{t}-\tfrac12 \nu\lbr \tfrac{\bsom}{4}\rbr^{\mfd-p}(c_o\varrho)^{sp} )$, we would have
		 \begin{equation*}%\label{Eq3.11}
			|\{u(\cdot,t^{\ast})\leq \bsmu^+-\tfrac14 \bsom\}\cap B_{\lbr \tfrac{\bsom}{4}\rbr^{\frac{\mfd-2}{sp}}(c_o\varrho)}|< \tfrac12\nu|B_{\lbr \tfrac{\bsom}{4}\rbr^{\frac{\mfd-2}{sp}}(c_o\varrho)}|.
		\end{equation*}
		Denoting the value $T:= \lbr \tfrac{\bsom}{4}\rbr^{\mfd-p}(c_o\varrho)^{sp}$,  we have the following sequence of estimates:
		\begin{equation*}
			\begin{array}{rcl}
				|\{u\leq \bsmu^+-\tfrac14 \bsom\}\cap \mcQ(\bar{t})|& = & \int_{\bar{t} - T}^{\bar{t} - \tfrac12 \nu T} |\{u(\cdot,s)\leq \bsmu^+-\tfrac14 \bsom\}\cap B_{\lbr \tfrac{\bsom}{4}\rbr^{\frac{\mfd-2}{sp}}(c_o\varrho)}| \,ds\\
				&& + \int_{\bar{t} - \tfrac12 \nu T}^{\bar{t}} |\{u(\cdot,s)\leq \bsmu^+-\tfrac14 \bsom\}\cap B_{\lbr \tfrac{\bsom}{4}\rbr^{\frac{\mfd-2}{sp}}(c_o\varrho)}| \,ds\\
				& <& \tfrac12 \nu | B_{\lbr \tfrac{\bsom}{4}\rbr^{\frac{\mfd-2}{sp}}(c_o\varrho)} | T (1-\tfrac12 \nu) + \tfrac12 \nu | B_{\lbr \tfrac{\bsom}{4}\rbr^{\frac{\mfd-2}{sp}}(c_o\varrho)} | T\\
				& < &  \nu|\mcQ(\bar{t})|,
			\end{array}
		\end{equation*}
		which is a contradiction.
	\end{proof}
	%\hrule\hrule\hrule\hrule
	\descitemnormal{Step 3:}{step3} In this step, we expand measure information forward in time on one of the intrinsic cylinder by applying \cref{lemma3.4}. Let us take $\theta_x=  \lbr4 \ve \rbr^{\frac{2-\mfd}{sp}}  \lbr \ve \bsom \rbr^{\frac{\mfd-2}{sp}} = \de_x  \lbr \ve \bsom \rbr^{\frac{\mfd-2}{sp}}$,   (i.e., $\de_x= \lbr4 \ve \rbr^{\frac{2-\mfd}{sp}}$) and $\ve \in (0,\tfrac14)$ to be chosen, then from \cref{lemma3.4},  we get
	\begin{equation*}
		|\{
		\bsmu^{+}-u(\cdot, t)\geq \de\ve \bsom\} \cap B_{\theta_xc_o\varrho}|
		\geq\tfrac{\nu}4 |B_{\theta_xc_o\varrho}|
		\quad \text{ for all } \,  t\in(t^{\ast},t^{\ast} + \bar\ve (4\ve)^{2-\mfd} (\ve\bsom)^{\mfd-p}(c_o\varrho)^{sp}].
	\end{equation*}
	We will choose $\ve$ such that $\bar\ve (4\ve)^{2-\mfd}  (\ve\bsom)^{\mfd-p}(c_o\varrho)^{sp} \geq  (\tfrac14\bsom)^{\mfd-p}(c_o\varrho)^{sp}$, i.e., $\ve^{p-2} \leq \bar\ve 4^{2-p}$. This way, the expansion can be claimed all the way to the top of the bad cylinder.  
	
	We further restrict $c_o$ such that $\tfrac{1}{4^{\mfd-2}}c_o^{sp} \leq \ve^{p-1} \leq \lbr \tfrac{1}{4} \bar\ve^{\tfrac{1}{p-2}} \rbr^{p-1}$ and this ensure the tail alternative is satisfied. In particular, we get
	\begin{equation*}
		|\{
		\bsmu^{+}-u(\cdot, t)\geq \de\ve \bsom\} \cap B_{\theta_xc_o\varrho}|
		\geq\tfrac{\nu}4 |B_{\theta_xc_o\varrho}|
		\quad \text{ for all } \,  t\in(- \lbr\tfrac{\bsom}{\mathbf{a}} \rbr^{\mfd-p}
		(c_o\varrho)^{sp},0].
	\end{equation*}
\item[Step 4:] Let $\sigma$ be a fixed constant to be eventually chosen, to apply \cref{lemma3.5}, we need to ensure $\lbr\tfrac{\bsom}{\mathbf{a}} \rbr^{\mfd-p}
(c_o\varrho)^{sp} \geq (\sigma \de\ve \bsom)^{\mfd-p}(c_o\varrho)^{sp}$, which requires $\mathbf{a}^{\mfd-p} \leq (\sigma \de\ve)^{p-\mfd}$. With this choice of $\mathbf{a}$, let us take $\theta_t = (\sigma \de\ve \bsom)^{\mfd-p}$ and $\de_x = 4^{\frac{2-\mfd}{sp}} (\sigma \de \ve)^{\frac{2-\mfd}{sp}}$. Then the cylinder $\mcQ= B_{\theta_x c_o \varrho} \times (- \theta_t (c_o\varrho)^{sp},0) = B_{\lbr \tfrac{\bsom}{4}\rbr^{\frac{\mfd-2}{sp}}(c_o\varrho)} \times (-(\sigma \de\ve \bsom)^{\mfd-p}(c_o\varrho)^{sp},0]$ becomes admissible  to  apply \cref{lemma3.5} (with $\ve$ replaced by $\de \ve$ from \descrefnormal{step3}{Step 3}) to get
	\begin{equation*}
		\begin{array}{rcl}
		|\{+(\bsmu^{+}-u) \leq \tfrac12 \sigma \ve \bsom\} \cap \mcQ| &\leq &  \mopC \lbr \tfrac{\de_x^{sp}}{\de_t}+1\rbr \tfrac{\sigma^{p-1}}{\nu (1-\sigma)^{p-1}} |\mcQ|\\
		& = & \mopC \lbr4^{2-\mfd} (\sigma \de\ve)^{2-\mfd}+1\rbr \tfrac{\sigma^{p-1}}{\nu (1-\sigma)^{p-1}} |\mcQ|\\
		& \leq & \mopC  \tfrac{\sigma^{p-1}}{\nu (1-\sigma)^{p-1}} |\mcQ|.
		\end{array}
	\end{equation*}
Given any $\nu_*$, we can choose $\sigma = \sigma(\datanb{}) \in (0,\tfrac12)$ such that $\mopC  \tfrac{\sigma^{p-1}}{\nu (1-\sigma)^{p-1}}\leq \nu_*$.
 We further restrict $c_o$ such that $\tfrac{1}{4^{\mfd-2}}c_o^{sp} \leq (\sigma \ve \de \bsom)^{p-1}$ and this ensure the tail alternative is satisfied.
 \item[Step 5:] Note that $\sigma$ is still yet to be chosen. With the choice $\theta_t = (\sigma \de\ve \bsom)^{\mfd-p}$ and $\de_x = 4^{\frac{2-\mfd}{sp}} (\sigma \de \ve)^{\frac{2-\mfd}{sp}}$, we can now apply \cref{lemma3.2} to get $\nu_1$ (which plays the role of $\nu_*$) such that 
 \begin{equation*}
 	+(\bsmu^{+}-u)\geq\tfrac{1}2\sigma \de\varepsilon \bsom
 	\quad
 	\mbox{ on }\quad B_{\lbr \tfrac{\bsom}{4}\rbr^{\frac{\mfd-2}{sp}}(\tfrac12c_o\varrho)} \times (-(\sigma \de\ve \bsom)^{\mfd-p}(\tfrac12c_o\varrho)^{sp},0].
 \end{equation*}
\item[Step 6:] The constant $\nu_1$ depends on $\sigma$ and $\sigma$ depends on $\nu_*$, which might lead to an impossibility in making the choices. But we show that this is not a problem as follows:

From \cref{lemma3.2}, we see that $\nu_1 \approx \lbr \tfrac{1}{\de_x^{sp}} + \tfrac{1}{\de_t}\rbr^{-\frac{n+sp}{sp}} (\de_x^n\de_t)^{-1}$, which takes the form $\nu_1 = \mopC  \sigma^{2-\mfd}$ where without loss of generality, we assumed $\de_x \leq 1$.

In the measure shrinking lemma, we need to ensure $\mopC \sigma^{p-1} \leq \nu_* =\mopC \sigma^{2-\mfd}$, which is possible since $1<\mfd < 2 \leq p$ .
%In the measure shrinking lemma, we need to ensure $\mopC \sigma^{p-1} \leq \nu_* =\mopC \sigma^{2-\mfd}$, which is possible provided $p-1 > 2-\mfd$ or $p+\mfd -3 > 0$. We make the choice of $\mfd$ to satisfy this additional restriction. In particular, this shows that the proof of reduction of oscillation holds in the range $\tfrac32 < p < \infty$ and  our estimates are not stable as $p \searrow \tfrac32$ and .  
 
\end{description}
This completes the proof of the reduction of oscillation when \descrefnormal{alt2}{DegAlt-$\II$} holds.

\subsection{Proof of H\"older continuity when \texorpdfstring{$1<p_o<p<2+\ep_{p_o}$}. for any fixed \texorpdfstring{$p_o>1$}.}

\begin{definition}
	For a fixed $p_o \in (1,2)$, let us take $\mfd = 2+\ep_{p_o}$  with $\epsilon_{p_o} := \min\left\{p_o-1, \tfrac{sp_o(p_o-1)}{n}, \tfrac{sp_o(p_o-1)}{n+3s}\right\}$ and consider $p$ in the range $p_o<p<\mfd$. Note that this implies $\mfd > \max\{2,p\}$. 
\end{definition}

Let us consider the reference cylinder $\mcQ_o = B_{R} \times (-R^{sp},0)$ for some fixed $R \in \RR^+$ which is the reference cylinder. Let us take constants
\begin{equation*}
	\bsmu^+=  \esssup_{\mcQ_o}u, \qquad \bsmu^-= \essinf_{\mcQ_o} u \qquad \text{ and } 
	\qquad \bsom  = 2\esssup_{\mcQ_o} |u| + \tail(u,\mcQ_o).	
\end{equation*}
Consider the cylinder 
\[
\mcQ^{\de_x,\de_t} = B_{\de_x \bsom^{\frac{\mfd-2}{sp}} (c_o\varrho)} \times (- \de_t \bsom^{\mfd-p} (c_o\varrho)^{sp},0].
\]
Then we choose $\varrho$ small such that the following inclusion holds:
\begin{equation}\label{sizerhosing}
B_{\de_x \bsom^{\frac{\mfd-2}{sp}} \varrho} \times (- \de_t \bsom^{\mfd-p} \varrho^{sp},0] \subset \mcQ_o \quad \Longrightarrow \quad \varrho := R\min\{\bsom^{\frac{2-\mfd}{sp}},\bsom^{\frac{p-\mfd}{sp}}\},
\end{equation}
for any $\de_x,\de_t \in  (0,1]$. Note that we have $\mfd > \max\{2,p\}$ and hence the cylinders are shrinking in size as $\bsom \searrow 0$.  We will now obtain the reduction of oscillation for this cylinder. 

\begin{proposition}\label{prop3.13}
	Let $u$ be bounded, weak solution of \cref{maineq} with $q=2$ and let $1<p<\mfd$. Suppose for two  constants $\alpha \in (0,1)$ and $\varepsilon \in (0,1)$, the following assumption holds:
	\[
	|\{\pm (\bsmu^{\pm} -u(\cdot,t_o))\geq \varepsilon \bsom\}\cap B_{ (\ve\bsom)^{\frac{\mfd-2}{sp}} (c_o\varrho)}| \geq \alpha |B_{ (\ve\bsom)^{\frac{\mfd-2}{sp}} (c_o\varrho)}|.
	\]
	Then there exists constants $\delta = \delta(\datanb{,\alpha})$, $\eta = \eta(\datanb{,\alpha})$ and $c_o = c_o(\eta,\varepsilon)$ such that
	\[
	c_o^{\frac{sp}{p-1}} \leq \eta\varepsilon,
	\]
	then the following conclusion holds:
	\[
	\pm(\bsmu^{\pm} - u) \geq \eta\varepsilon\bsom \quad \text{on} \quad B_{(\ve\bsom)^{\frac{\mfd-2}{sp}} \frac12c_o \varrho} \times (t_o +\tfrac12 \de (\ve\bsom)^{\mfd-p}(c_o\varrho)^{sp}
	,t_o+ \de (\ve\bsom)^{\mfd-p}(c_o\varrho)^{sp}].
	\]
	Provided $B_{(\ve\bsom)^{\frac{\mfd-2}{sp}} c_o \varrho} \times (t_o
	% +\tfrac12 \de (\ve\bsom)^{\mfd-p}(c_o\varrho)^{sp}
	,t_o+ \de (\ve\bsom)^{\mfd-p}(c_o\varrho)^{sp}] \subset \mcQ_o$.
\end{proposition}
%\hrule
\begin{proof}
	The proof follows in several steps:
	\begin{description}[leftmargin=*]
		\descitemnormal{Step 1:}{step1sing} We apply \cref{lemma3.4} with $\de_x=\de_t =1$ and $\bsom$ replaced with $\ve \bsom$ to get 
		\[
		|\{\pm (\bsmu^{\pm} - u(\cdot,t))\geq \de\varepsilon \bsom\} \cap B_{(\ve\bsom)^{\frac{\mfd-2}{sp}} (c_o\varrho)} | \geq \tfrac12\alpha |B_{ (\ve\bsom)^{\frac{\mfd-2}{sp}} (c_o\varrho)}| \txt{for a.e} t \in (t_o,t_o+\bar\varepsilon(\varepsilon\bsom)^{\mfd-p}(c_o\varrho)^{sp}], 
		\]
		where $\bar\varepsilon$ and $\delta$ (both depending on $\al$) are as obtained in  \cref{lemma3.4}. To ensure the tail alternative holds, we estimate as follows:
		\[
		\begin{array}{rcl}
			c_o^{{sp}}\tailp((u-\bsmu_-)_-, \theta_x\varrho,0,  (t_o,t_o+\bar\varepsilon(\varepsilon\bsom)^{\mfd-p}(c_o\varrho)^{sp})) & \leq & c_o^{sp}\frac{(\ve\bsom)^{\mfd-2}\varrho^{s p}}{R^{sp}}\tailp((u-\bsmu_-)_-, R,0,(-R^{sp},0))\\
			& \overset{\cref{sizerhosing}}{\leq} & c_o^{sp} \bsom^{p-1},
		\end{array}
		\]
		Hence if we choose $c_o^{sp} \leq \ve^{p-1}$, then the tail alternative is satisfied. 
%				\hrule
		\descitemnormal{Step 2:}{step2sing} In this step, we want to apply \cref{lemma3.5} which requires verifying the measure hypothesis is satisfied. Let us choose $\de_x = (\de\sigma)^{\frac{2-\mfd}{sp}}$ and $\de_t = \bar\ve \de^{p-\mfd}$, then we see that 
		\[
		B_{(\ve\bsom)^{\frac{\mfd-2}{sp}} (c_o\varrho)} \times (t_o,t_o+\bar\varepsilon(\varepsilon\bsom)^{\mfd-p}(c_o\varrho)^{sp}) = B_{\de_x (\sigma\de\ve\bsom)^{\frac{\mfd-2}{sp}}(c_o\varrho)} \times (t_o, t_o+ \de_t (\de\ve\bsom)^{\mfd-p}(c_o\varrho)^{sp}).
		\]
		Since $\sigma \in (0,1)$ and $\mfd>p$, we see that $ \de_t (\de\ve\bsom)^{\mfd-p}(c_o\varrho)^{sp} \geq  \de_t (\sigma\de\ve\bsom)^{\mfd-p}(c_o\varrho)^{sp}$. Hence for all $\hat{t} \in (t_o+\de_t (\sigma\de\ve\bsom)^{\mfd-p}(c_o\varrho)^{sp} ,  t_o +\bar\varepsilon(\varepsilon\bsom)^{\mfd-p}(c_o\varrho)^{sp} ]$,  we have the measure information over the cylinder 
		\[
		B_{\de_x (\sigma\de\ve\bsom)^{\frac{\mfd-2}{sp}}(c_o\varrho)} \times (\hat{t} - \de_t (\sigma\de\ve\bsom)^{\mfd-p}(c_o\varrho)^{sp} ,\hat{t}] \txt{with} \de_x =(\de\sigma)^{\frac{2-\mfd}{sp}} \quad \text{and} \quad \de_t = \bar\ve \de^{p-\mfd}.
		\]
		Thus, applying \cref{lemma3.5}, we get
		\begin{equation*}
			|\{\pm(\bsmu^{\pm}-u) \leq \tfrac12 \sigma \ve \bsom\} \cap \mcQ| \leq \mopC \lbr \tfrac{\de_x^{sp}}{\de_t}+1\rbr \tfrac{\sigma^{p-1}}{\al (1-\sigma)^{p-1}} |\mcQ|, % \approx \mopC |\mcQ|,
		\end{equation*}
		where we have denoted $\mcQ = B_{\theta_x c_o \varrho} \times (\hat{t}- \de_t (\sigma\ve\bsom)^{\mfd-p}(c_o\varrho)^{sp},  \hat{t}]$.
		
		We note that the tail alternative holds if we choose $c_o^{sp} \leq ( \sigma \de\ve)^{p-1}$ as in \descrefnormal{step1sing}{Step 1}.
		
		\item[Step 3:] We now apply \cref{lemma3.2} to get $\nu_1 \approx_{\datanb{}} \lbr \tfrac{1}{\de_x^{sp}} + \tfrac{1}{\de_t}\rbr^{-\frac{n+sp}{sp}} (\de_x^n\de_t)^{-1}$ such that the following conclusion follows:
			\begin{equation*}
			\pm\left(\bsmu^{\pm}-u\right)\geq\tfrac{1}4\sigma \de\varepsilon \bsom
			\quad
			\mbox{ on }\quad B_{\frac12\theta_x c_o \varrho} \times (\hat{t} - \de_t (\sigma\de\ve\bsom)^{\mfd-p}(\tfrac12c_o\varrho)^{sp} , \hat{t}].
		\end{equation*}
 Again we can choose $c_o^{sp} \leq (\sigma \de\ve)^{p-1}$ small enough such that the tail alternative holds. 
	\item[Step 4:] We need to choose $\sigma$ such that 
	\[
	\mopC \lbr \tfrac{\de_x^{sp}}{\de_t}+1\rbr \tfrac{\sigma^{p-1}}{\al (1-\sigma)^{p-1}} \leq \mopC \lbr \tfrac{1}{\de_x^{sp}} + \tfrac{1}{\de_t}\rbr^{-\frac{n+sp}{sp}} (\de_x^n\de_t)^{-1}.
	\]
	Recalling $\de_x =(\de\sigma)^{\frac{2-\mfd}{sp}}$ and  $\de_t = \bar\ve \de^{p-\mfd}$, the left hand side has the form
	\[
	\mopC \lbr \tfrac{\de_x^{sp}}{\de_t}+1\rbr \tfrac{\sigma^{p-1}}{\al (1-\sigma)^{p-1}} \approx \mopC (\sigma^{2-\mfd} +1) \sigma^{p-1} = \mopC (\sigma^{p-\mfd+1} + \sigma^{p-1}).
	\]
	The right hand side has the form 
	\[
%	\begin{array}{rcl}
		\mopC \lbr \tfrac{1}{\de_x^{sp}} + \tfrac{1}{\de_t}\rbr^{-\frac{n+sp}{sp}} (\de_x^n\de_t)^{-1}  \approx  \mopC \frac{\sigma^{(\mfd-2)\frac{n}{sp}}}{ \lbr \sigma^{\mfd-2} + 1 \rbr^{1+\frac{n}{sp}}} \\
		 \overset{\sigma\in(0,\frac12)}{\geq}   \mopC \sigma^{(\mfd-2)\frac{n}{sp}}.
%		& = & \mopC \lbr \frac{1}{\sigma^{p-2}+1}\rbr\lbr  \frac{1}{\lbr 1 + \sigma^{2-p} \rbr^{\frac{n}{sp}}}\rbr\\
%		& \leq BAD & \mopC. 
%	\end{array}
	\]
	Based on the above calculations, if we denote $\mfd = 2+\epsilon$, then we need to ensure the following is satisfied:
	\begin{enumerate}[(i)]
		\item $p-\mfd +1 = p-1-\epsilon > 0$.
		\item $p-\mfd +1  = p-1-\epsilon> (\mfd-2)\tfrac{n}{sp} = \epsilon\tfrac{n}{sp}$.
		\item $p-1 > (\mfd-2)\tfrac{n}{sp} = \epsilon\tfrac{n}{sp}$.
	\end{enumerate}
This can be ensured if we select $\epsilon = \min\left\{p-1, \tfrac{sp(p-1)}{n}, \tfrac{sp(p-1)}{n+sp}\right\}$. In particular, if we fix some $1<p_o<2$, then $p_o-1>0$ and we can choose $\epsilon = \epsilon_{p_o} = \min\left\{p_o-1, \tfrac{sp_o(p_o-1)}{n}, \tfrac{sp_o(p_o-1)}{n+3s}\right\}$ such that we have reduction of oscillation in the range $(p_o,2+\ve_{p_o})$. 
\item[Step 5:] From step 3, we have for all $\hat{t} \in (t_o+\de_t (\sigma\de\ve\bsom)^{\mfd-p}(c_o\varrho)^{sp} ,  t_o +\bar\varepsilon(\varepsilon\bsom)^{\mfd-p}(c_o\varrho)^{sp} ]$,
\begin{equation*}
			\pm\left(\bsmu^{\pm}-u(\cdot,  \hat{t})\right)\geq\tfrac{1}4\sigma \de\varepsilon \bsom
			\quad
			\mbox{ on }\quad B_{\frac12\theta_x c_o \varrho}.
		\end{equation*}
	Now we can further choose $\sigma$ small if needed such that 
	$$
\de_t (\sigma\de\ve\bsom)^{\mfd-p}(c_o\varrho)^{sp} = \bar\ve	(\sigma \ve\bsom)^{\mfd-p}(c_o\varrho)^{sp} \leq \tfrac12 \bar\varepsilon(\varepsilon\bsom)^{\mfd-p}(c_o\varrho)^{sp}.
	$$
This can be ensure by taking $\sigma \leq (\tfrac{1}{2})^{\frac{1}{\mfd -p}}$ and we have that 
$$
(t_o + \tfrac12 \bar\varepsilon(\varepsilon\bsom)^{\mfd-p}(c_o\varrho)^{sp} , t_o +\bar\varepsilon(\varepsilon\bsom)^{\mfd-p}(c_o\varrho)^{sp}] \subset (t_o+\de_t (\sigma\de\ve\bsom)^{\mfd-p}(c_o\varrho)^{sp} ,  t_o +\bar\varepsilon(\varepsilon\bsom)^{\mfd-p}(c_o\varrho)^{sp} ].
	$$
	\end{description}
This  completes the proof of the lemma. 
\end{proof}

%\hrule
\begin{lemma}
	\label{reductionoscsingunified}
	There exists two constants $\eta$ and $\boldsymbol{C}_o$ depending only on data, such that if we denote $\bsom_1:= (1-\eta)\bsom$ and $\varrho_1:= \boldsymbol{C}_o^{-1}\varrho$, then we have 
	\[
	\essosc_{\mcQ_2} u \leq (1-\eta) \essosc_{\mcQ_1},
	\]
	where $\mcQ_2 = B_{\bsom_1^{\frac{\mfd-2}{sp}} \varrho_1} \times (- \bsom_1^{\mfd-p}\varrho_1^{sp},0)$ and $\mcQ_1 = B_{\bsom^{\frac{\mfd-2}{sp}} \varrho} \times (- \bsom^{\mfd-p}\varrho^{sp},0]$. 
\end{lemma}
\begin{proof}
We now obtain the reduction of oscillation as follows:
\begin{description}[leftmargin=*]
	\item[Step 1:] Let us define the time level $\tau := -\de (\ve \bsom)^{\mfd-p} (c_o\varrho)^{sp}$ where the constants $\de$ and $c_o$ are as obtained in \cref{prop3.13} with $\ve = \tfrac14$ and $\alpha = \tfrac12$. 
	\item[Step 2:] One of the two alternatives now must hold: 
	\begin{equation}\label{eq3.13}
		\left\{\begin{array}{l}
		|\{+ (\bsmu^{+} -u(\cdot,\tau))\geq \tfrac14 \bsom\}\cap B_{ (\frac14\bsom)^{\frac{\mfd-2}{sp}} (c_o\varrho)}| \geq \alpha |B_{ (\frac14\bsom)^{\frac{\mfd-2}{sp}} (c_o\varrho)}|,\\
		|\{- (\bsmu^{-} -u(\cdot,\tau))\geq \tfrac14 \bsom\}\cap B_{ (\frac14\bsom)^{\frac{\mfd-2}{sp}} (c_o\varrho)}| \geq \alpha |B_{ (\frac14\bsom)^{\frac{\mfd-2}{sp}} (c_o\varrho)}|.
		\end{array}\right.
	\end{equation}
\item[Step 3:] Without loss of generality, let us assume the second alternative in \cref{eq3.13} holds, noting that the other case follows analogously. 
\item[Step 4:] We now apply \cref{prop3.13} to find corresponding constants such that the following conclusion follows:
\[
-(\bsmu^{-} - u) \geq \eta\bsom \quad \text{on} \quad B_{(\frac14\bsom)^{\frac{\mfd-2}{sp}} \frac12c_o \varrho} \times (-\tfrac12 \de (\tfrac14\bsom)^{\mfd-p}(c_o\varrho)^{sp},0].
\]
The choice of $c_o$ is chosen to ensure the tail alternative is satisfied. 
\item[Step 5:] Let us take $\bsom_1 := (1-\eta)\bsom$ and $\varrho_1 = \boldsymbol{C}_o^{-1}\varrho$ with 
\[
\boldsymbol{C}_o^{-1} \leq \min\{\de c_o (1-\eta)^{\frac{2-\mfd}{sp}},  c_o (1-\eta)^{\frac{p-\mfd}{sp}}\},
\]
then we will have 
\[
\essosc_{\mcQ_2} u \leq (1-\eta) \essosc_{\mcQ_1},
\]
where $\mcQ_2 = B_{\bsom_1^{\frac{\mfd-2}{sp}} \varrho_1} \times (- \bsom_1^{\mfd-p}\varrho_1^{sp},0)$ and $\mcQ_1 = B_{\bsom^{\frac{\mfd-2}{sp}} \varrho} \times (- \bsom^{\mfd-p}\varrho^{sp},0]$. 
\end{description}
This completes the proof of the reduction of oscillation. 
\end{proof}

%\hrule\hrule\hrule\hrule\hrule

\section{Four Lemmas for \texorpdfstring{$(p,q)$}. nonlocal quasilinear equation}
%\subsection{Singular case with non-singular intrinsic scaling}

	\subsection{de Giorgi iteration} 
	The first lemma we prove is a de Giorgi iteration:
	\begin{lemma}\label{lemma4.21}
		Let $u$ be a weak solution of \cref{maineq}. Given $\de_x, \de_t, \varepsilon \in (0,1)$, set $\theta_t = \de_t (\varepsilon \bsom)^{\mfd-p}$, $\theta_x = \de_x (\varepsilon \bsom)^{\frac{\mfd-q }{sp}}$ %for some $1<\mfd<\min\{2,p\}$ 
		and denote $\mcQ_{c_o\varrho}^{\theta_x,\theta_t} = B_{\theta_x c_o\varrho} \times (-\theta_t (c_o\varrho)^{sp},0)$. Then there exists a constant $\nu_1 = \nu(\datanb{,\delta_x,\delta_t}) \in (0,1)$ such that if 
		\begin{equation*}
			\left|\left\{
			\pm\left(\bsmu^{\pm}-u\right)\leq \varepsilon \bsom\right\}\cap \mcQ_{c_o\varrho}^{\theta_x,\theta_t}\right|
			\leq
			\nu_1|\mcQ_{c_o\varrho}^{\theta_x,\theta_t}|,
		\end{equation*}
		holds along with the assumption 
		\begin{equation*}%\label{vwLm:3:3:hypothesis}
			c_o^{\frac{sp}{p-1}}\tail((u-\bsmu_-)_-, \theta_x\varrho,0,(-\theta_t \varrho^{sp},0))\leq \varepsilon\bsom,  \quad \text{and} \quad |\bsmu^\pm| \leq 8\ve \bsom,
		\end{equation*}
		then the following conclusion follows:
		\begin{equation*}
			\pm\left(\bsmu^{\pm}-u\right)\geq\tfrac{1}2\varepsilon \bsom
			\quad
			\mbox{ on }\quad \mcQ_{\frac{1}{2}c_o\varrho}^{\theta_x,\theta_t} = B_{\theta_x\frac{c_o\varrho}2} \times \left(-\theta_t\lbr \tfrac12c_o\varrho\rbr^{sp}, 0\right].
		\end{equation*}
	If we denote $\Gamma := \lbr \tfrac{1}{\de_x^{sp}} + \tfrac{1}{\de_t}\rbr^{\frac{n+sp}{p(n+2s)}} (\de_x^n\de_t)^{\frac{s}{(n+2s)}}$, then the relation between $\nu$ and $\Gamma$ takes the form $ \nu_1 \approx \Gamma^{-\frac{n+2s}{s}}$ or in particular, $\nu_1 \approx \lbr \tfrac{1}{\de_x^{sp}} + \tfrac{1}{\de_t}\rbr^{-\frac{n+sp}{sp}} (\de_x^n\de_t)^{-1}$.
	\end{lemma}
	\begin{proof}
		We prove the lemma in the case of $-(\bsmu^--u)$ noting that the other case follows analogously. For $j=0,1,\ldots$, define
		\begin{equation*}%\label{choices:B_n}
%			\left\{
			{\def\arraystretch{1.1}\begin{array}{cccc}
					k_j:=\bsmu^-+\tfrac{\varepsilon \bsom}2+\tfrac{\varepsilon \bsom}{2^{j+1}},& \tilde{k}_j:=\tfrac{k_j+k_{j+1}}2,&
					\varrho_j:=\tfrac{c_o\varrho}2+\tfrac{c_o\varrho}{2^{j+1}},
					&\tilde{\varrho}_j:=\tfrac{\varrho_j+\varrho_{j+1}}2,\\%[5pt]
					B_j:=B_{\varrho_j}^{\theta_x},& \tilde{B}_j:=B_{\tilde{\varrho}_j}^{\theta_x},&
					\mcQ_j:=\mcQ_{\varrho_j}^{\theta_x,\theta_t},&
					\tilde{\mcQ}_j:=\mcQ_{\tilde\varrho_j}^{\theta_x,\theta_t}.
			\end{array}}
%			\right.
		\end{equation*}
		Furthermore, we also define 
		\begin{equation*}%\label{choices:B_nn}
			%	\left\{
			%	{\def\arraystretch{1.1}\begin{array}{cccc}
			\hat{\varrho}_j:=\tfrac{3\varrho_j+\varrho_{j+1}}4, \quad 
			\bar{\varrho}_j:=\tfrac{\varrho_j+3\varrho_{j+1}}4,\quad 
			\hat{\mcQ}_j:=\mcQ_{\hat\varrho_j}^{\theta_x,\theta_t}, \quad
			\bar{\mcQ}_j:=\mcQ_{\bar\varrho_j}^{\theta_x,\theta_t}.
			%	\end{array}}
			%	\right.
		\end{equation*}
		We now consider a cutoff functions $\bar{\zeta_j}$ and $\zeta_j$ such that
		\begin{equation*}%\label{Lm:3.2:cutoff}
			\begin{array}{c}
				\bar{\zeta_j} \equiv 1 \text{ on } B_{j+1}, \quad \bar{\zeta_j} \in C_c^{\infty}(\bar{B}_{j}), \quad  |\nabla\bar{\zeta_j}| \apprle \frac{1}{\theta_x(\bar{\varrho_j} - \varrho_{j+1})} \approx \frac{2^j}{\theta_xc_o\varrho} \quad \text{and} \quad  |\pa_t\bar{\zeta_n}| \apprle \frac{1}{\theta_t(\bar\varrho_j^{sp} - \varrho_{j+1}^{sp})} \approx \frac{2^{jsp}}{\theta_t(c_o\varrho)^{sp}}, \\
				{\zeta_j} \equiv 1 \text{ on } \tilde{B}_{j}, \quad {\zeta_j} \in C_c^{\infty}(\hat{B}_{j}), \quad  |\nabla{\zeta_j}| \apprle \frac{1}{\theta_x(\hat{\varrho_n} - \tilde{\varrho_{j}})}\approx \frac{2^{j}}{\theta_xc_o\varrho} \quad \text{and} \quad  |\pa_t{\zeta_j}| \apprle \frac{1}{\theta_t(\hat\varrho_j^{sp} - \tilde{\varrho_{j}}^{sp})}\approx \frac{2^{jsp}}{\theta_t(c_o\varrho)^{sp}}.
			\end{array}
		\end{equation*}
		Let us estimate each of the terms appearing on the right hand side of \cref{Prop:energy} as follows:
		\begin{description}[leftmargin=*]
			\item[Estimate for the first term:] Since $(u-k_j)-_ \leq \ve \bsom$, we get 
			\begin{multline*}
				\int_{-\theta_t \varrho_j^{sp}}^0\iint_{B_j \times B_j}\hspace*{-0.7cm} \frac{\max\{(u-k_j)_{-}(x,t),(u-k_j)_{-}(y,t)\}^{p}|\zeta(x,t)-\zeta(y,t)|^p}{|x-y|^{n+sp}}\,dx\,dy\,dt\\
				\begin{array}{rcl}
					& \leq & (\ve \bsom)^p  \lbr \frac{2^{j}}{\theta_x(c_o\varrho)} \rbr^p \int_{-\theta_t \varrho_j^{sp}}^0 \int_{B_j}\int_{B_{j}} \frac{\lsb{\chi}{\{u(x,t) < k_j\}}}{|x-y|^{n+(s-1)p}}\,dx\,dy\,dt\\
					& \leq & \mopC (\ve \bsom)^p  \lbr \frac{2^{j}}{\theta_x(c_o\varrho)} \rbr^p \int_{-\theta_t \varrho_j^{sp}}^0 \int_{B_j}\int_{4B_{j}} \frac{\lsb{\chi}{\{u(x,t) < k_j\}}}{|y|^{n+(s-1)p}}\,dy\,dx\,dt\\
					& \leq & \mopC (\ve \bsom)^p   \frac{2^{pj}}{(\theta_x \varrho_j)^{sp}} |A_j| \approx \mopC (\ve \bsom)^p   \frac{2^{pj}}{\theta_x^{sp}(c_o\varrho)^{sp}} |A_j|,
				\end{array}
			\end{multline*}
		where we have denoted $A_j := \{u(x,t) \leq k_j\} \cap \mcQ_j$.
%			\hrule
			\item[Estimate for the second term:]
			First we make following observations
\begin{itemize}
    \item For $\tilde{k} < k$, there holds $(u-k)_- \geq (u-\tilde{k})_- $.
    \item On $A_j = \{u<k_j \} \cap \mcQ_j$, we have $\bsmu^- \leq u \leq k_j \leq \bsmu^- + \ve \bsom$.
     \item On $A_j = \{u<k_j\} \cap \mcQ_j$, we have $|u| + |k_j| \leq 18 \ve \bsom$ since $|\bsmu^-| \leq 8 \ve \bsom $.
      \item On the set $ \{u<\tilde{k}_j \} $, we have $|u| + |k_j| \geq k_j - u \geq k_j - \tilde{k}_j =  \frac{ \ve \bsom}{2^{j+3}}$ .
      \item On the set $A_{j+1}$, we have $(u-k_j )_- \geq (u- \tilde{k}_j )_- = \tilde k_j - u \geq \tilde k_j - k_{j+1} =  \frac{ \ve \bsom}{2^{j+3}}$.
\end{itemize}
We note that by \cref{lem:g},  we have
$$
\mathfrak g_- ^q (u,k)
\le
\bsc  \lbr|u| + |k|\rbr^{q-2}(u-k)_-^2 \le
\bsc  \lbr|u| + |k|\rbr^{q-1}(u-k)_- .
$$
		Now	 it is easily estimated as follows:
			\begin{equation*}
%				\begin{array}{rcl}
					\iint_{\mcQ_j} \mfg_{-}^q (u,k_j) |\partial_t\zeta_j^p| \,dx\,dt  \apprle  (\ve \bsom)^q |A_j| \frac{2^{jsp}}{\theta_t (c_o \varrho)^{sp}}.
%				\end{array}
			\end{equation*}
%		\hrule
		\item[Estimate for the third term:] This term is zero as our cut-off functions are chosen to have zero values on the parabolic boundary of $\mcQ_j$. 
			\item[Estimate for the fourth term:] Since $x \in \spt \zeta_j \Longrightarrow |x| \leq \theta_x \hat{\varrho_j}$ and $|y| \geq \theta_x\varrho_j$, we get
			\[
			\frac{|y-x|}{|y|} \geq \frac{\varrho_j - \hat{\varrho_j}}{\hat{\varrho_j}} = \frac14 \lbr  \frac{\varrho_j - \varrho_{j+1}}{\varrho_j}\rbr \geq \frac{1}{2^{j+4}},
			\]
			which allows us to estimate as follows:
			\begin{multline*}
				\lbr  \underset{\stackrel{t \in (-\theta_t \varrho_j^{sp},0)}{x\in \spt \zeta_j}}{\esssup}\,\int_{\RR^n \setminus B_j}\frac{(u-k_j)_{-}^{p-1}(y,t)}{|x-y|^{n+sp}}\,dy\rbr\iint_{(-\theta_t \varrho_j^{sp},0)\times B_j}\hspace*{-0.6cm} (u-k_j)_{-}(x,t)\zeta_j^p(x,t)\,dx\,dt\\
				\begin{array}{rcl}
					& \leq & \mopC 2^{(n+sp)j}(\ve \bsom) |A_j| \lbr\underset{t \in (-\theta_t \varrho_j^{sp},0)}{\esssup}\,\int_{\RR^n \setminus B_j}\frac{(\ve\bsom + (\bsmu_--u)_+ )^{p-1}(y,t)}{|y|^{n+sp}}\,dy \rbr\\
					& \leq & \mopC 2^{(n+sp)j}(\ve \bsom) |A_j|  \lbr\frac{(\ve \bsom)^{p-1}}{(\theta_x\varrho_j)^{sp}}+\underset{t \in (-\theta_t \varrho_j^{sp},0)}{\esssup}\,\int_{\RR^n \setminus B_j}\frac{(u -\bsmu_- )_-^{p-1}(y,t)}{|y|^{n+sp}}\,dy \rbr\\
					& = & \mopC 2^{(n+sp)j}(\ve \bsom) |A_j|  \lbr\frac{(\ve \bsom)^{p-1}}{(\theta_x\varrho_j)^{sp}}+\frac{1}{(\theta_x \varrho_j)^{sp}}\underset{t \in (-\theta_t \varrho_j^{sp},0)}{\esssup} (\theta_x \varrho_j)^{sp}\int_{\RR^n \setminus B_j}\frac{(u -\bsmu_- )_-^{p-1}(y,t)}{|y|^{n+sp}}\,dy \rbr\\
%					& \leq & \mopC 2^{(n+sp)j}(\ve \bsom) |A_j|  \lbr\frac{(\ve \bsom)^{p-1}}{(\theta_x\varrho_j)^{sp}}+\frac{1}{(\theta_x \varrho_j)^{sp}}\underset{t \in (-\theta_t \varrho^{sp},0)}{\esssup} (\theta_x \varrho_j)^{sp}\int_{\RR^n \setminus B_j}\frac{(u -\bsmu_- )_-^{p-1}(y,t)}{|y|^{n+sp}}\,dy \rbr\\
					& \overred{a.1}{a}{\leq} & \mopC 2^{(n+sp)j}(\ve \bsom) |A_j|  \lbr\frac{(\ve \bsom)^{p-1}}{(\theta_x\varrho_j)^{sp}}+\frac{c_o^{sp}\tail((u-\bsmu_-)_-, \theta_x\varrho,0,(-\theta_t \varrho^{sp},0))^{p-1} }{(\theta_x\varrho_j)^{sp}}\rbr\\
					& \leq & \mopC 2^{(n+sp)j}(\ve \bsom) |A_j|  \frac{(\ve \bsom)^{p-1}}{(\theta_x\varrho_j)^{sp}} \approx  \mopC 2^{(n+sp)j}(\ve \bsom) |A_j|  \frac{(\ve \bsom)^{p-1}}{\theta_x^{sp}(c_o\varrho)^{sp}},
				\end{array}
			\end{multline*}
		where \redref{a.1}{a} follows from \cref{Rmk:5.1}. 
%		\hrule
				\item[Estimate for the first term on left side of inequality:]
				    
       First by \cref{lem:g},  we have that
$$
\mathfrak g_- ^q (u,k)
\geq
\frac{1}{\bsc}  \lbr|u| + |k|\rbr^{q-2}(u-k)_-^2 .
$$
We combine this with the following observation,  when $1<q<2$ we use that on set $A_j = \{u<k_j\} \cap \mcQ_j$, we have $|u| + |k_j| \leq 18 \ve \bsom$ and when $2<q<\infty$
we use that on the set $ \{u<\tilde{k}_j \} $, we have $|u| + |k_j| \geq k_j - u \geq k_j - \tilde{k}_j \geq  \frac{ \ve \bsom}{2^{j+3}}.$  That gives us
$$
\underset{-\theta_t \varrho_j^{sp} < t < 0}{\esssup}\int_{B_j} \mathfrak g_- ^q (u,k_j) \zeta_j^p \geq
\underset{-\theta_t \varrho_j^{sp} < t < 0}{\esssup}\int_{B_j}(|u|+|k_j|)^{q-2}(u-k_j)_{-}^2\zeta_j^p \,dx \geq \frac{b_j} {C} (\ve \bsom)^{q-2}\underset{(-\theta_t \tilde\varrho_j^{sp},0)}{\esssup}
			\int_{\tilde{B}_j} (u-{k}_j)_-^2\,dx,
$$
where $b_j = \min \{ 2^{(2-q)(j+3)} ,  18^{q-2} \}$.
		\end{description}
	Combining the previous five estimates, we get
	\begin{multline}\label{Eq:4.21}
%		\begin{aligned}
		\frac{b_j}{C} (\ve \bsom)^{q-2}	\underset{(-\theta_t \tilde\varrho_j^{sp},0)}{\esssup}
			\int_{\tilde{B}_j} (u-{k}_j)_-^2\,dx
			+\iiint_{\tilde{\mcQ}_j}\frac{|(u-{k}_j)_{-}(x,t)-(u-{k}_j)_{-}|^p}{|x-y|^{n+sp}}\,dx \,dy\,dt
			\\\leq
			\mopC 2^{(n+sp)j} |A_j| \lbr  \frac{(\ve \bsom)^{p}}{\theta_x^{sp}(c_o\varrho)^{sp}} +  \frac{(\ve \bsom)^q}{\theta_t (c_o \varrho)^{sp}}\rbr.
%		\end{aligned}
	\end{multline}
%\hrule
Using Young's inequality, we have
\begin{multline*}%\label{Eq:4.31}
	|(u-{k}_{j})_- \bar{\zeta_j}(x,t) - (u-{k}_{j})_- \bar{\zeta_j}(y,t)|^p \leq c |(u-{k}_{j})_- (x,t) - (u-{k}_{j})_- (y,t)|^p\bar{\zeta_j}^p(x,t) \\
	+ c |(u-{k}_{j})_-(y,t)|^p |\bar{\zeta_j}(x,t) - \bar{\zeta_j}(y,t)|^p,
\end{multline*}
from which we obtain the following sequence of estimates:
\begin{equation*}
	\begin{array}{rcl}
		\frac{\ve \bsom}{2^{j+2}}
		|A_{j+1}|
		&\overred{4.41a}{a}{\leq} &
		\iint_{\tilde\mcQ_{j}}(u-{k}_j)_-\bar{\zeta_j}\,dx\,dt \\
		&\overred{4.41b}{b}{\leq} &
		\lbr\iint_{\tilde\mcQ_{j}}\left[(u-{k}_j)_-\bar{\zeta_j}\right]^{p\frac{n+2s}{n}}
		\,dx\,dt\rbr^{\frac{n}{p(n+2s)}}|A_j|^{1-\frac{n}{p(n+2s)}}\\
		&\overred{4.41c}{c}{\leq} &\mopC 
		\left(\iiint_{\tilde\mcQ_j}\frac{|(u-{k}_j)_{-}(x,t)\bar{\zeta_j}(x,t)-(u-{k}_j)_{-}\bar{\zeta_j}(y,t)|^p}{|x-y|^{n+sp}}\,dx \,dy\,dt\right)^{\frac{n}{p(n+2s)}} 
		\\&&\qquad\times \left(\underset{-\theta_t\tilde\varrho_j^{sp} < t < 0}{\esssup}\int_{\tilde{B}_j}[(u-\tilde{k}_j)_{-}\bar\zeta_j(x,t)]^2\,dx\right)^{\frac{s}{n+2s}}|A_j|^{1-\frac{n}{p(n+2s)}}\\
		&\overred{4.41d}{d}{\leq}&
		\mopC b_o^j   \lbr  \frac{(\ve \bsom)^{p}}{\theta_x^{sp}(c_o\varrho)^{sp}} +  \frac{(\ve \bsom)^q}{\theta_t (c_o \varrho)^{sp}}\rbr^{\frac{n}{p(n+2s)}}  \lbr (\ve \bsom)^{2-q} \left[ \frac{(\ve \bsom)^{p}}{\theta_x^{sp}(c_o\varrho)^{sp}} +  \frac{(\ve \bsom)^q}{\theta_t (c_o \varrho)^{sp}} \right] \rbr^{\frac{s}{(n+2s)}}
		|A_j|^{1+\frac{s}{n+2s}} \\
		&\overred{4.41e}{e}{\leq} &
		\mopC
		\frac{b_o^j}{(c_o\varrho)^\frac{s(n+sp)}{n+2s}}
		(\ve\bsom)^{\lbr (p+q- \mfd) {\frac{n}{p(n+2s)}} + (p+2-\mfd){\frac{s}{(n+2s)}}\rbr}
		\lbr \tfrac{1}{\de_x^{sp}} + \tfrac{1}{\de_t}\rbr^{\frac{n+sp}{p(n+2s)}}|A_j|^{1+\frac{s}{n+2s}}, 
	\end{array}
\end{equation*}
where to obtain \redref{4.41a}{a}, we made use of the observations and enlarged the domain of integration with $\bar{\zeta}_j$; to obtain \redref{4.41b}{b}, we applied H\"older's inequality; to obtain \redref{4.41c}{c}, we applied \cref{fracpoin}; to obtain \redref{4.41d}{d}, we made use of \cref{Eq:4.21}  with $\mopC = \mopC_{\data{}}$ and finally we collected all the terms to obtain \redref{4.41e}{e}, where $b_o = b_o(\datanb{})\geq 1$ is a constant. Setting
$\bsy_j=|A_j|/|\mcQ_j|$ and noting $|\mcQ_{i+1}| \approx |\mcQ_j| \approx \de_x^n \de_t (\ve \bsom)^{\frac{n(\mfd - q)}{sp}+\mfd -p} (c_o\varrho)^{n+sp}$,  we get
\[
\frac{\ve \bsom}{2^{j+2}} \bsy_{j+1} \leq \mopC  \bsy_j^{1+\frac{s}{n+2s}}\frac{b_o^j}{(c_o\varrho)^\frac{s(n+sp)}{n+2s}}
(\ve\bsom)^{\lbr(p+q- \mfd) {\frac{n}{p(n+2s)}} + (p+2-\mfd){\frac{s}{(n+2s)}} \rbr }
\lbr \tfrac{1}{\de_x^{sp}} + \tfrac{1}{\de_t}\rbr^{\frac{n+sp}{p(n+2s)}} |\mcQ_{j}|^{\frac{s}{n+2s}}.
\]
Simplifying this expression and noting that $\lbr \tfrac{(p+q-\mfd )n}{p(n+2s)} +\tfrac{ (p+2-\mfd) s}{n+2s} \rbr + \lbr \tfrac{s}{n+2s}\rbr \lbr \tfrac{n(\mfd - q)}{sp} +  \mfd -p  \rbr = 1$,  we get
\begin{equation*}
	\bsy_{j+1}
	\le
	\bsc \Gamma  \boldsymbol b^j  \bsy_j^{1+\frac{s}{n+2s}},
\end{equation*}
where $\Gamma := \lbr \tfrac{1}{\de_x^{sp}} + \tfrac{1}{\de_t}\rbr^{\frac{n+sp}{p(n+2s)}} (\de_x^n\de_t)^{\frac{s}{(n+2s)}}$ and   $\bsc  = \bsc_{\data{}} \geq 1$,  $\boldsymbol b = \boldsymbol{b}_{\data{}}\geq  1$ are two constants depending only on the data. The conclusion now follows from \cref{geo_con}.
	\end{proof}
%\hrule
\subsection{de Giorgi iteration with quantitative initial data}
The second lemma we prove is a de Giorgi iteration involving quantitative initial data.
\begin{lemma}\label{lemma4.31}
	Let $u$ be a weak solution of \cref{maineq}. Given $\de_x, \varepsilon \in (0,1)$, there exists $\nu_2 = \nu_1(\datanb{})\in (0,1)$ such that if we take $\theta_x = \de_x (\varepsilon \bsom)^{\frac{\mfd - q }{sp}}$, $\de_t = \nu_2 \de_x^{sp}$ and suppose for some time level $t_o$, the following assumptions are satisfied:
	\begin{equation*} 
	\left\{ \begin{array}{l}	\pm(\bsmu^{\pm}-u(\cdot,t_o))\geq \varepsilon \bsom \txt{on} B_{\theta_x c_o\varrho},\\
%	\end{equation*}
%	holds along with the assumption 
%	\begin{equation*}%\label{vwLm:3:3:hypothesis}
		c_o^{\frac{sp}{p-1}}\tail((u-\bsmu_-)_-, \theta_x\varrho,0,(t_o,\rint))\leq \varepsilon\bsom,
		\end{array} \right.
	\end{equation*}
where $\rint := t_o + \de_t (\ve \bsom )^{\mfd -p}(c_o\varrho)^{sp}$. Furthermore we assume 
 $|\bsmu^\pm| \leq 8\ve \bsom$  holds when $1<q<2$,	then the following conclusion follows:
	\begin{equation*}
		\pm(\bsmu^{\pm}-u)\geq\tfrac{1}2\varepsilon \bsom
		\quad
		\mbox{ on }\quad  B_{\theta_x\frac{c_o\varrho}2} \times (t_o, t_o+\nu_2 \de_x^{sp}(\ve \bsom)^{\mfd-p}(c_o\varrho)^{sp}].
	\end{equation*}
%where $\mopt = $.

%	If we denote $\Gamma := \lbr \tfrac{1}{\de_x^{sp}} + \tfrac{1}{\de_t}\rbr^{\frac{n+sp}{p(n+2s)}} (\de_x^n\de_t)^{\frac{s}{(n+2s)}}$, then the relation between $\nu$ and $\Gamma$ takes the form $ \nu \approx \Gamma^{-\frac{n+2s}{s}}$ or in particular, $\nu \approx \lbr \tfrac{1}{\de_x^{sp}} + \tfrac{1}{\de_t}\rbr^{-\frac{n+sp}{sp}} (\de_x^n\de_t)^{-1}$.
\end{lemma}
\begin{proof}
	 We prove the lemma in the case of $-(\bsmu^--u)$ noting that the other case follows analogously. For $j=0,1,\ldots$, define
	\begin{equation*}%\label{choices:B_n}
		%			\left\{
		{\def\arraystretch{1.1}\begin{array}{cccc}
				k_j:=\bsmu^-+\tfrac{\varepsilon \bsom}2+\tfrac{\varepsilon \bsom}{2^{j+1}},& \tilde{k}_j:=\tfrac{k_j+k_{j+1}}2,&
				\varrho_j:=\tfrac{c_o\varrho}2+\tfrac{c_o\varrho}{2^{j+1}},
				&\tilde{\varrho}_j:=\tfrac{\varrho_j+\varrho_{j+1}}2,\\%[5pt]
				B_j:=B_{\varrho_j}^{\theta_x},& \tilde{B}_j:=B_{\tilde{\varrho}_j}^{\theta_x},&
				\hat{\varrho}_j:=\tfrac{3\varrho_j+\varrho_{j+1}}4, &
				\bar{\varrho}_j:=\tfrac{\varrho_j+3\varrho_{j+1}}4.
		\end{array}}
		%			\right.
	\end{equation*}
%	Furthermore, we also define 
%	\begin{align}\label{choices:B_nn}
%		%	\left\{
%		%	{\def\arraystretch{1.1}\begin{array}{cccc}
%		\hat{\varrho}_j:=\tfrac{3\varrho_j+\varrho_{j+1}}4, \quad 
%		\bar{\varrho}_j:=\tfrac{\varrho_j+3\varrho_{j+1}}4,\quad 
%		\hat{\mcQ}_j:=\mcQ_{\hat\varrho_j}^{\theta_x,\theta_t}, \quad
%		\bar{\mcQ}_j:=\mcQ_{\bar\varrho_j}^{\theta_x,\theta_t}.
%		%	\end{array}}
%		%	\right.
%	\end{align}
	We now consider a cutoff functions $\bar{\zeta_j}(x)$ and $\zeta_j(x)$ independent of time such that
	\begin{equation*}%\label{cutoff_size}
		\begin{array}{c}
			\bar{\zeta_j} \equiv 1 \text{ on } B_{j+1}, \quad \bar{\zeta_j} \in C_c^{\infty}(\bar{B}_{j}), \quad  |\nabla\bar{\zeta_j}| \apprle \frac{1}{\theta_x(\bar{\varrho_j} - \varrho_{j+1})} \approx \frac{2^j}{\theta_xc_o\varrho}, \\
			{\zeta_j} \equiv 1 \text{ on } \tilde{B}_{j}, \quad {\zeta_j} \in C_c^{\infty}(\hat{B}_{j}), \quad  |\nabla{\zeta_j}| \apprle \frac{1}{\theta_x(\hat{\varrho_n} - \tilde{\varrho_{j}})}\approx \frac{2^{j}}{\theta_xc_o\varrho}.
		\end{array}
	\end{equation*}
%\hrule
We will apply \cref{Prop:energy} over cylinders of the form $B_j \times (t_o,\rint)$ and denote $\rint := t_o + \tht_t (\ve \bsom )^{2-\mfd}(c_o\varrho)^{sp}$ (note that the time intervals do not change in the iteration). 
	Let us estimate each of the terms appearing on the right hand side of \cref{Prop:energy} as follows:
	\begin{description}[leftmargin=*]
		\item[Estimate for the first term:] Since $(u-k_j)_- \leq \ve \bsom$, we get 
		\begin{multline*}
			\int_{t_o}^{\rint}\iint_{B_j \times B_j}\hspace*{-0.7cm} \frac{\max\{(u-k_j)_{-}(x,t),(u-k_j)_{-}(y,t)\}^{p}|\zeta(x)-\zeta(y)|^p}{|x-y|^{n+sp}}\,dx\,dy\,dt\\
			\begin{array}{rcl}
				& \leq & (\ve \bsom)^p  \lbr \frac{2^{j}}{\theta_x(c_o\varrho)} \rbr^p \int_{t_o}^{\rint} \int_{B_j}\int_{B_{j}} \frac{\lsb{\chi}{\{u(x,t) < k_j\}}}{|x-y|^{n+(s-1)p}}\,dx\,dy\,dt\\
				& \leq & \mopC (\ve \bsom)^p  \lbr \frac{2^{j}}{\theta_x(c_o\varrho)} \rbr^p \int_{t_o}^{\rint} \int_{B_j}\int_{4B_{j}} \frac{\lsb{\chi}{\{u(x,t) < k_j\}}}{|y|^{n+(s-1)p}}\,dy\,dx\,dt\\
				& \leq & \mopC (\ve \bsom)^p   \frac{2^{pj}}{(\theta_x\varrho_j)^{sp}} |A_j| \approx \mopC (\ve \bsom)^p   \frac{2^{pj}}{\theta_x^{sp}(c_o\varrho)^{sp}} |A_j|,
			\end{array}
		\end{multline*}
		where we have denoted $A_j := \{u(x,t) \leq k_j\} \cap (B_j \times (t_o,\rint))$.
		
		\item[Estimate for the second term:] This term is zero since $\pa_t\zeta_j = 0$.
%		$u \geq \bsmu_- + \ve \bsom$ This is easily estimated as follows:
%		\begin{equation}
%			\begin{array}{rcl}
%				\iint_{\mcQ_j}(u-k_j)_{-}^2|\partial_t\zeta_j^p| \,dx\,dt & \leq & (\ve \bsom)^2 |A_j| \frac{2^{jsp}}{\theta_t (c_o \varrho)^{sp}}
%			\end{array}
%		\end{equation}
		\item[Estimate for the third term:] This term is zero since the hypothesis says $u \geq \bsmu_-+\ve\bsom$ at the initial time level $t=t_o$ and hence $(u(\cdot,t_o)-k_j)_-=0$. 
%		\hrule
		\item[Estimate for the fourth term:] Since $x \in \spt \zeta_j \Longrightarrow |x| \leq \theta_x \hat{\varrho_j}$ and $|y| \geq \theta_x\varrho_j$, we get
		\[
		\frac{|y-x|}{|y|} \geq \frac{\varrho_j - \hat{\varrho_j}}{\hat{\varrho_j}} = \frac14 \lbr  \frac{\varrho_j - \varrho_{j+1}}{\varrho_j}\rbr \geq \frac{1}{2^{j+4}},
		\]
		which allows us to estimate as follows:
		\begin{multline*}
			\lbr  \underset{\stackrel{t \in (t_o,\rint)}{x\in \spt \zeta_j}}{\esssup}\,\int_{\RR^n \setminus B_j}\frac{(u-k_j)_{-}^{p-1}(y,t)}{|x-y|^{n+sp}}\,dy\rbr\iint_{B_j\times (t_o,\rint)}\hspace*{-0.6cm} (u-k_j)_{-}(x,t)\zeta_j^p(x)\,dx\,dt\\
			\begin{array}{rcl}
				& \leq & \mopC 2^{(n+sp)j}(\ve \bsom) |A_j| \lbr\underset{t \in (t_o,\rint)}{\esssup}\,\int_{\RR^n \setminus B_j}\frac{(\ve\bsom + (\bsmu_--u)_+ )^{p-1}(y,t)}{|y|^{n+sp}}\,dy \rbr\\
				& \leq & \mopC 2^{(n+sp)j}(\ve \bsom) |A_j|  \lbr\frac{(\ve \bsom)^{p-1}}{(\theta_x\varrho_j)^{sp}}+\underset{t \in (t_o,\rint)}{\esssup}\,\int_{\RR^n \setminus B_j}\frac{(u -\bsmu_- )_-^{p-1}(y,t)}{|y|^{n+sp}}\,dy \rbr\\
				& = & \mopC 2^{(n+sp)j}(\ve \bsom) |A_j|  \lbr\frac{(\ve \bsom)^{p-1}}{(\theta_x\varrho_j)^{sp}}+\frac{1}{(\theta_x \varrho_j)^{sp}}\underset{t \in (t_o,\rint)}{\esssup} (\theta_x \varrho_j)^{sp}\int_{\RR^n \setminus B_j}\frac{(u -\bsmu_- )_-^{p-1}(y,t)}{|y|^{n+sp}}\,dy \rbr\\
%				& \leq & \mopC 2^{(n+sp)j}(\ve \bsom) |A_j|  \lbr\frac{(\ve \bsom)^{p-1}}{(\theta_x\varrho_j)^{sp}}+\frac{(\theta_x \varrho_j)^{sp-sp}}{(\theta_x \varrho)^{sp}}\underset{t \in (t_o,\rint)}{\esssup} (\theta_x \varrho)^{sp}\int_{\RR^n \setminus B_j}\frac{(u -\bsmu_- )_-^{p-1}(y,t)}{|y|^{n+sp}}\,dy \rbr\\
				& \leq & \mopC 2^{(n+sp)j}(\ve \bsom) |A_j|  \lbr\frac{(\ve \bsom)^{p-1}}{(\theta_x\varrho_j)^{sp}}+\frac{c_o^{sp}\tail((u-\bsmu_-)_-, \theta_x\varrho,0,(t_o,\rint))^{p-1} }{(\theta_x\varrho_j)^{sp}}\rbr\\
				& \leq & \mopC 2^{(n+sp)j}(\ve \bsom) |A_j|  \frac{(\ve \bsom)^{p-1}}{(\theta_x\varrho_j)^{sp}} \approx  \mopC 2^{(n+sp)j}(\ve \bsom) |A_j|  \frac{(\ve \bsom)^{p-1}}{\theta_x^{sp}(c_o\varrho)^{sp}}.
			\end{array}
		\end{multline*}
				\item[Estimate for the first term on left side of inequality:]
				    
       First by \cref{lem:g},  we have that
$$
\mathfrak g_- ^q (u,k)
\geq
\frac{1}{\bsc}  \lbr|u| + |k|\rbr^{q-2}(u-k)_-^2 .
$$
We combine this with the following observation,  when $1<q<2$ we use that on set $A_j = \{u<k_j\} \cap \mcQ_j$, we have $|u| + |k_j| \leq 18 \ve \bsom$ and when $2<q<\infty$
we use that on the set $ \{u<\tilde{k}_j \} $, we have $|u| + |k_j| \geq k_j - u \geq k_j - \tilde{k}_j \geq  \frac{ \ve \bsom}{2^{j+3}}$ .  That gives us
$$
\underset{( t_o , \rint )}{\esssup}\int_{B_j} \mathfrak g_- ^q (u,k_j) \zeta_j^p \geq 
\underset{t_o < t < \rint }{\esssup}\int_{B_j}(|u|+|k_j|)^{q-2}(u-k_j)_{-}^2\zeta_j^p \,dx \geq \frac{b_j}{C} (\ve \bsom)^{q-2} \underset{(t_o,  \rint)}{\esssup}
			\int_{\tilde{B}_j} (u-{k}_j)_-^2\,dx,
$$
where $b_j = \min \{ 2^{(2-q)(j+3)} ,  18^{q-2} \}$.		
	\end{description}
	Combining the previous three estimates along with the observations, we have
	\begin{equation}\label{Eq:4.61}
		%		\begin{aligned}
	\frac{b_j}{C} (\ve \bsom)^{q-2}	\underset{(t_o,\rint)}{\esssup}
		\int_{\tilde{B}_j} (u-{k}_j)_-^2\,dx
		+\iiint_{\tilde{\mcQ}_j}\frac{|(u-{k}_j)_{-}(x,t)-(u-{k}_j)_{-}|^p}{|x-y|^{n+sp}}\,dx \,dy\,dt
				\leq
		\mopC 2^{(n+sp)j} |A_j| \lbr  \frac{(\ve \bsom)^{p}}{\theta_x^{sp}(c_o\varrho)^{sp}}\rbr,
		%		\end{aligned}
	\end{equation}
where we have denoted $\tilde\mcQ_j = \tilde{B}_j  \times (t_o,\rint)$. 
	%\hrule
	Using Young's inequality, we have
	\begin{multline}\label{Eq:4.71}
		|(u-{k}_{j})_- \bar{\zeta_j}(x,t) - (u-{k}_{j})_- \bar{\zeta_j}(y,t)|^p \leq c |(u-{k}_{j})_- (x,t) - (u-{k}_{j})_- (y,t)|^p\bar{\zeta_j}^p(x,t) \\
		+ c |(u-{k}_{j})_-(y,t)|^p |\bar{\zeta_j}(x,t) - \bar{\zeta_j}(y,t)|^p,
	\end{multline}
	from which we obtain the following sequence of estimates:
	\begin{equation*}
		\begin{array}{rcl}
			\frac{\ve \bsom}{2^{j+2}}
			|A_{j+1}|
			&\overred{4.81a}{a}{\leq} &
			\iint_{\tilde\mcQ_{j}}(u-{k}_j)_-\bar{\zeta_j}\,dx\,dt \\
			&\overred{4.81b}{b}{\leq} &
			\lbr\iint_{\tilde\mcQ_{j}}\left[(u-{k}_j)_-\bar{\zeta_j}\right]^{p\frac{n+2s}{n}}
			\,dx\,dt\rbr^{\frac{n}{p(n+2s)}}|A_j|^{1-\frac{n}{p(n+2s)}}\\
			&\overred{4.81c}{c}{\leq} &\mopC 
			\left(\iiint_{\tilde\mcQ_j}\frac{|(u-{k}_j)_{-}(x,t)\bar{\zeta_j}(x,t)-(u-{k}_j)_{-}\bar{\zeta_j}(y,t)|^p}{|x-y|^{n+sp}}\,dx \,dy\,dt\right)^{\frac{n}{p(n+2s)}} 
			\\&&\qquad\times \left(\underset{t\in (t_o,\rint)}{\esssup}\int_{\tilde{B}_j}[(u-\tilde{k}_j)_{-}\bar\zeta_j(x,t)]^2\,dx\right)^{\frac{s}{n+2s}}|A_j|^{1-\frac{n}{p(n+2s)}}\\
			&\overred{4.81d}{d}{\leq}&
			\mopC b_o^j   \lbr  \frac{(\ve \bsom)^{p}}{\theta_x^{sp}(c_o\varrho)^{sp}} \rbr^{\frac{n}{p(n+2s)}}
			 \lbr  \frac{(\ve \bsom)^{p+2-q}}{\theta_x^{sp}(c_o\varrho)^{sp}} \rbr^{\frac{s}{(n+2s)}}
			|A_j|^{1+\frac{s}{n+2s}} \\
			&\overred{4.81e}{e}{\leq} &
			\mopC
			\frac{b_o^j}{(c_o\varrho)^\frac{s(n+sp)}{n+2s}}
			(\ve\bsom)^{\lbr (p+q- \mfd) {\frac{n}{p(n+2s)}} + (p+2-\mfd){\frac{s}{(n+2s)}}\rbr}
			\lbr \tfrac{1}{\de_x^{sp}}\rbr^{\frac{n+sp}{p(n+2s)}}|A_j|^{1+\frac{s}{n+2s}}, 
		\end{array}
	\end{equation*}
	where to obtain \redref{4.81a}{a}, we made use of the observations and enlarged the domain of integration; to obtain \redref{4.81b}{b}, we applied H\"older's inequality; to obtain \redref{4.81c}{c}, we applied \cref{fracpoin}; to obtain \redref{4.81d}{d}, we made use of \cref{Eq:4.61} along with \cref{Eq:4.71} with $\mopC = \mopC_{\data{}}$ and finally we collected all the terms to obtain \redref{4.81e}{e}, where $b_o = b_o(\datanb{})\geq 1$ is a  constant. Setting
	$\bsy_j=|A_j|/|\mcQ_j|$  and noting $|\mcQ_{j+1}| \approx |\mcQ_j| \approx \de_x^n \de_t (\ve \bsom)^{\frac{n(\mfd -q)}{sp}+\mfd -p } (c_o\varrho)^{n+sp}$ where we have denoted $\mcQ_j := B_{j} \times (t_o,\rint)$,  we get
	\[
	\frac{\ve \bsom}{2^{j+2}} \bsy_{j+1} \leq \mopC  \bsy_j^{1+\frac{s}{n+2s}}\frac{b_o^j}{(c_o\varrho)^\frac{s(n+sp)}{n+2s}}
	(\ve\bsom)^{\lbr (p+q-\mfd) {\frac{n}{p(n+2s)}} + (p+2-\mfd){\frac{s}{(n+2s)}}\rbr}
	\lbr \tfrac{1}{\de_x^{sp}} \rbr^{\frac{n+sp}{p(n+2s)}} |\mcQ_{j}|^{\frac{s}{n+2s}}.
	\]
	Simplifying this expression noting that $\lbr \tfrac{(p+q-\mfd) n}{p(n+2s)} +\tfrac{ (p+2- \mfd) s}{n+2s} \rbr + \lbr \tfrac{s}{n+2s}\rbr \lbr \tfrac{n(\mfd-q)}{sp} +  \mfd-p  \rbr = 1$,  we get
	\begin{equation*}
		\bsy_{j+1}
		\le
		\bsc \Gamma  \boldsymbol b^j  \bsy_j^{1+\frac{s}{n+2s}},
	\end{equation*}
	where $\Gamma := \lbr \tfrac{1}{\de_x^{sp}}\rbr^{\frac{n+sp}{p(n+2s)}} (\de_x^n\de_t)^{\frac{s}{(n+2s)}}$ and   $\bsc  = \bsc_{\data{}} \geq 1$,  $\boldsymbol b\geq  1$ are two constants depending only on the data. We apply  \cref{geo_con} to see if $\bsy_0 \leq \bsc^{-1/\alpha}\Gamma^{-1/\alpha} \boldsymbol b^{1/\alpha^2}$ with $\alpha = \tfrac{s}{n+2s}$, then $\bsy_{\infty} = 0$ which is the desired conclusion.  We make the choice of $\de_t$ to satisfy
	\[
	1 = \bsc^{-1/\alpha}\Gamma^{-1/\alpha} \boldsymbol b^{1/\alpha^2} = \bsc^{-1/\alpha}\boldsymbol b^{1/\alpha^2} \de_x^{sp} \de_t^{-1} \quad \Longrightarrow\quad  \de_t = \bsc^{-1/\alpha}\boldsymbol b^{1/\alpha^2} \de_x^{sp} =: \nu_2 \de_x^{sp},
	\]
	which completes the proof of the lemma.
\end{proof}
%\hrule

\subsection{Propagation of measure information forward in time}

%\hrule

\begin{lemma}\label{lemma4.41}
	Let $\alpha \in(0,1)$ be given, then there exist constants $\de = \de(n,p,q,s,\La,  \alpha) \in (0,1)$ and $\bar\ve  = \bar\ve(n,p,q,s,\La,\alpha)\in (0,1)$,
	 such that whenever $u$ satisfies
	\begin{equation*}
		\left|\left\{
		\pm\lbr\bsmu^{\pm}-u(\cdot, t_o)\rbr\geq \ve\bsom
		\right\}\cap B_{\theta_xc_o\varrho}(x_o)\right|
		\geq\alpha \left|B_{\theta_xc_o\varrho}\right|,
	\end{equation*}
where $\theta_x = \de_x (\varepsilon \bsom)^{\frac{\mfd - q }{sp}}$ 	and the bounds  
	\begin{equation*}%\label{sec:exptime:hyp}
		c_o^{\frac{sp}{p-1}}\tail((u-\bsmu^{\pm})_{\pm};\theta_x\varrho,0,(t_o,t_o+\bar\ve \de_x^{sp}(\ve\bsom)^{\mfd-p}(c_o\varrho)^{sp}]) \leq  \ve\bsom,   \quad \text{and} \quad |\bsmu^\pm| \leq 8\ve \bsom,
	\end{equation*}
	then the following conclusion follows:
	\begin{equation*}%\label{Eq:4:9k}
		|\{
		\pm\left(\bsmu^{\pm}-u(\cdot, t)\right)\geq \de\ve \bsom\} \cap B_{\theta_xc_o\varrho} |
		\geq\tfrac{\alpha}2 |B_{\theta_xc_o\varrho}|
		\quad \text{ for all } \,  t\in(t_o,t_o + \bar\ve \de_x^{sp}(\ve\bsom)^{\mfd-p}(c_o\varrho)^{sp}].
	\end{equation*}
\end{lemma}
\begin{proof} 
	We prove the case of supersolutions only because the case for subsolutions is analogous. We set
	\[
	A_{M}^{\varrho}(t) := \{u(\cdot,t)  - \bsmu^-< M\} \cap B_{\theta_xc_o\varrho}.
	\]
	By hypothesis, we have 
	\begin{equation}\label{Eq:4:7}
	|A_{\ve\bsom}^{\varrho}(t_o)|\leq (1-\alpha)|B_{\theta_xc_o\varrho}|.
	\end{equation}
	We apply the energy estimate  from \cref{Prop:energy} for $(u-k)_{-}$ over the cylinder $B_{\theta_xc_o\varrho}\times(t_o,\mft]$  with $k=\bsmu^-+\ve\bsom$. Note that $(u-k)_{-} \leq \ve\bsom$ in $B_{\theta_xc_o\varrho}$. For $\sigma \in (0,\tfrac18]$ to be chosen later, we take a cutoff function $\zeta = \zeta(x)\geq 0$,  such that it is supported in $B_{\theta_x c_o\varrho(1-\frac{\sigma}{2})}$ with $\zeta \equiv 1$ on $B_{(1-\sigma)\theta_xc_o\varrho}$ and $|\nabla \zeta| \apprle \tfrac{1}{\sigma\theta_xc_o\varrho}$. We now make use of  \cref{Prop:energy} over $(t_o,\mft)$ and estimate analogously as \cref{lemma4.31}  and using $u\geq\bsmu^-$ to get
%	\hrule
	\begin{equation}\label{Eq:4.10k}
			\begin{array}{rcl}
		\int_{B_{(1-\sigma)\theta_xc_o\varrho}\times\{\mft\}} \mathfrak g_-^q(u,k) \zeta^p\,dx  
		&\leq&   \mopC\frac{ (\ve\bsom)^p}{\sigma^{n+sp} \theta_x^{sp}(c_o\varrho)^{sp}}|\mft-t_o||B_{\theta_xc_o\varrho}|
		+ \int_{B_{\theta_xc_o\varrho}\times\{t_o\}} \mathfrak g_-^q(u,k) \zeta^p \,dx.
	\end{array}
\end{equation}
Now we estimate $\int_{B_{\theta_xc_o\varrho}\times\{t_o\}} \mathfrak g_-^q(u,k) \zeta^p \,dx$ from above as follows:
\begin{equation*}
	\begin{array}{rcl}
	\int_{B_{\theta_xc_o\varrho}\times\{t_o\}} \mathfrak g_-^q(u,k) \zeta^p \,dx
		& \overset{\cref{defgpm}}{=} &  \int_{B_{\theta_xc_o\varrho}\times\{t_o\}} (q-1) \int_{u}^{k} |s|^{q-2}(s-k)_-\,ds \, dx  \\
		& \leq &  |A_{\ve\bsom}^{\varrho}(t_o) |(q-1) \int_{\bsmu^-}^{k} |s|^{q-2}(s-k)_-\,ds \\
		& \overset{\cref{Eq:4:7}}{\leq} &  (1-\alpha)|B_{\theta_xc_o\varrho}| (q-1) \int_{\bsmu^-}^{k} |s|^{q-2}(s-k)_-\,ds ,
		    \end{array}
	\end{equation*}
	where $\mft >t_o$ is any time level yet to be chosen. 
	Setting $k_\de:=\bsmu^-+\de \ve\bsom$ for some $ \de \in (0,\tfrac12)$ to be chosen later, we estimate the term on the left hand side of \cref{Eq:4.10k} from below to get
	\begin{equation*}
		\begin{array}{rcl}
	\int_{B_{(1-\sigma)\theta_xc_o\varrho}\times\{\mft\}} g_-^q(u,k) \zeta^p \,dx 
	    &  \overset{\cref{defgpm}}{=}& \int_{B_{(1-\sigma)\theta_xc_o\varrho}\times\{\mft\}} (q-1)\int_{u}^k |s|^{q-2}(s-k)_-\,ds \,\zeta^p\,dx  \\
%		& \geq & \int_{B_{(1-\sigma)\theta_xc_o\varrho}\times\{\mft\} \cap\{u<k_{\de}\}}\hspace*{-1cm} (q-1)\int_{u}^k |s|^{q-2}(s-k)_-\,ds \,\zeta^p\,dx  \\
		&\geq & |A_{k_\de}^{(1-\sigma)\varrho}(\mft)| (q-1) \int_{k_\de}^k |s|^{q-2}(s-k)_-\,ds.
		\end{array}
	\end{equation*}
%	Making use of the bound 
%	$\tfrac{1}{2}M\leq(1-\epsilon)M=k-k_\epsilon\leq |k_\epsilon|+|k|\leq 2(|\bsmu^-|+ M)\leq 18 M$, we further get
%	\begin{align}\label{prop_1}
%		\int_{k_\epsilon}^k|s|^{p-2}(s-k)_-\,ds
%		=
%		\tfrac{1}{p-1}\, \mathfrak g_-(k_\epsilon,k)
%		\geq 
%		\tfrac{1}{\bsc  (p)}
%		\lbr|k_\epsilon|+|k|\rbr^{p-2} (k-k_\epsilon)^2
%		\geq
%		\tfrac{1}{\bsc  (p)} M^p.
%	\end{align}
Furthermore, we use the bound 
$\tfrac{1}{2} (\ve \bsom ) \leq(1-\de) (\ve \bsom ) =k-k_\de \leq |k_\de|+|k|\leq 2(|\bsmu^-|+ (\ve \bsom ))\leq  18 (\ve \bsom ) $ to get
\begin{align}\label{prop_1}
	(q-1)\int_{k_\de}^k|s|^{q-2}(s-k)_-\,ds
	=
	 \mathfrak g_-^q(k_\de ,k)
	\overlabel{lem:g}{\geq} 
	\tfrac{1}{\bsc  (q)}
	\lbr|k_\de|+|k|\rbr^{q-2} (k- k_\de)^2
	\geq
	\tfrac{1}{\bsc  (q)} (\ve \bsom)^q.
\end{align}
Next, we note that the following decomposition holds:
	\begin{align*}
		|A_{k_\de}^{\varrho}(\mft)| \leq 
		|A_{k_\de}^{(1-\sigma)\varrho}(\mft)|+n\sigma |B_{\theta_xc_o\varrho}|.
	\end{align*}
	Combining  all the above estimates and recalling the definition of $\theta_x$, we get
	\begin{align*}
		|\{-(\bsmu^- -u(\cdot,\mft))< \de\ve\bsom\} \cap B_{\theta_xc_o\varrho}|
		\leq 
		\lbr \frac{\int_{\bsmu^-}^k |s|^{q-2}(s-k)_-\,ds }{\int_{k_\de}^k |s|^{q-2}(s-k)_-\,ds } (1-\alpha) + \mopC \frac{(\ve\bsom)^{p-\mfd }|\mft-t_o|}{\sigma^{n+sp} \de_x^{sp} (c_o\varrho)^{sp}} +n\sigma\rbr |B_{\theta_xc_o\varrho}|,
	\end{align*}
%\hrule
	for a universal constant $\mopC  > 0$. 
Proceeding analogously as in \cite[Proof of Lemma 4.1]{bogeleinHolderRegularitySigned2021}, we write
\[
\frac{\int_{\bsmu^-}^k |s|^{q-2}(s-k)_{-} \,ds }{\int_{k_\de}^k |s|^{q-2}(s-k)_{-} \,ds } = 1 + \frac{\int_{\bsmu^-}^{k_\de} |s|^{q-2}(s-k)_{-} \,ds}{\int_{k_\de}^k  |s|^{q-2}(s-k)_{-} \,ds} = 1 + I_{\de}.
\]	
Using $|\bsmu^-|\leq 8 (\ve \bsom )$, $|k_\de |\leq 9 (\ve \bsom )$ and applying \cref{lem:g},  we get
\[
	\int_{\bsmu^-}^{k_\de} |\tau|^{q-2}(\tau-k)_-\,d\tau
	\leq 
	(\ve  \bsom) \int_{\bsmu^-}^{\bsmu^-+\de  (\ve  \bsom)}|\tau|^{q-2}\,d\tau
	=
	(\ve  \bsom) |s|^{q-2} s\Big|_{\bsmu^-}^{\bsmu^-+ \de (\ve \bsom)}
	\leq
	\bsc  (q) (\ve \bsom)^q \de.
\]
Making use of  \cref{prop_1},  we obtain
\begin{align*}
	I_\de
	\leq
	\bsc  (q) \de .
\end{align*}	
	We now choose $\de\in (0,1)$ and $\sigma$ small enough such that
	\begin{equation*}
	 (1-\alpha) (1+ \bsc  (q) \de )	\leq {1-\frac{3\alpha}{4}} \txt{and} \sigma = \frac{\alpha}{8n}.
	\end{equation*}
With these choices, let us take $\mft > t_o$ to satisfy
\begin{equation}\label{4.11k}
|\mft- t_o| \leq \lbr\frac{\alpha}{8}\rbr \lbr\frac{\sigma^{n+sp} \de_x^{sp} (\ve\bsom)^{\mfd - p}(c_o\varrho)^{sp}}{\mopC}\rbr =: \bar\ve \de_x^{sp} (\ve\bsom)^{ \mfd -p}(c_o\varrho)^{sp}.
\end{equation}
	%and make the choice $\sigma=\tfrac{\alpha}{8n}$.
	%Finally, we choose $\delta\in (0,1)$ small enough so that $\tfrac{\bsc\delta}{\sigma^p}\leq\tfrac{\alpha}{8}$, 
Since $\mft$ was arbitrary, we get the desired conclusion on the full time interval provided \cref{4.11k} holds.
\end{proof}
\subsection{Measure shrinking lemma}
%\hrule
\begin{lemma}\label{lemma4.51}
	Let $\de_x, \de_t$ be given and assume for some $\sigma \in (0,1)$, the following is satisfied:
	\begin{equation*}
		\left|\left\{
		\pm\lbr\bsmu^{\pm}-u(\cdot, t)\rbr\geq \ve\bsom
		\right\}\cap B_{\theta_xc_o\varrho}(x_o)\right|
		\geq\alpha \left|B_{\theta_xc_o\varrho}\right| \txt{for all} t \in (t_o, t_o + \theta_t (c_o\varrho)^{sp}),
	\end{equation*}
and 
	\begin{equation*}%\label{Lm:3:2:hyp}
		c_o^{\frac{sp}{p-1}}\tail((u-\bsmu^{\pm})_{\pm};2\theta_x\varrho,0,(t_o,t_o + \theta_t (c_o\varrho)^{sp}]) \leq \sigma \ve \bsom,   \quad \text{and} \quad |\bsmu^\pm| \leq \sigma  \ve \bsom,
	\end{equation*}
where $\theta_x = \de_x (\sigma \ve\bsom)^{\frac{\mfd - q}{sp}}$ and $\theta_t = \de_t (\sigma \ve\bsom)^{\mfd-p}$.
	Then there exists $\mopC >0$ depending only on data such that the following holds:
	\begin{equation*}
		|\{\pm(\bsmu^{\pm}-u) \leq \tfrac12 \sigma \ve \bsom\} \cap \mcQ| \leq \mopC \lbr \tfrac{\de_x^{sp}}{\de_t}+1\rbr \tfrac{\sigma^{p-1}}{\al (1-\sigma)^{p-1}} |\mcQ|,
	\end{equation*}
where we have denoted $\mcQ = B_{\theta_x c_o \varrho} \times (t_o, t_o + \theta_t (c_o\varrho)^{sp})$.  
\end{lemma}
\begin{proof}
	We prove the case of super-solutions only because the case for sub-solutions is analogous.  We write the energy estimate from \cref{Prop:energy} over the cylinder $B_{\theta_x c_o\varrho}\times (t_o,\mft]$ with $\mft = t_o + \theta_t (c_o\varrho)^{sp}$ and  $k = \bsmu^-+\sigma \ve \bsom$.
	We choose a test function $\zeta = \zeta(x)$ such that $\zeta\equiv 1$ on $B_{\theta_xc_o\varrho}$, it is supported in $B_{\theta_x\frac32 c_o\varrho}$ and $|\nabla \zeta|\apprle \tfrac{1}{\theta_xc_o\varrho}$.
%	\hrule
%	
	Using $(u-k)_{-} \leq \sigma \ve \bsom$ locally and  by our choice of the test function, applying \cref{Prop:energy}, we get
	%\hrule
	\begin{multline*}
		\int_{t_o}^{\mft}\int_{B_{\theta_xc_o\varrho}}|(u-k)_-(x,t)|\int_{B_{\theta_xc_o\varrho}}\frac{|(u-k)_+(y,t)|^{p-1}}{|x-y|^{n+sp}}\,dx\,dy\,dt
		\\
		\begin{array}{rcl}
		&\leq&  \mopC (\sigma\ve \bsom)^p   \frac{1}{\theta_x^{sp}(c_o\varrho)^{sp}} |B_{\theta_xc_o\varrho}||\mft-t_o|
		+ \int_{B_{\theta_xc_o\varrho}\times t_0 } \mathfrak g_-^q (u,k) \,dx \\
		& \leq &  \mopC (\sigma\ve \bsom)^{q}   \tfrac{\de_t}{\de_x^{sp}} |B_{\theta_xc_o\varrho}|
		+ \int_{B_{\theta_xc_o\varrho}\times t_0 } \mathfrak g_-^q (u,k) \,dx.
		\end{array}
%		\underset{1.1ackrel{t \in (-\delta\varrho^{sp},0]}{ x\in \spt \zeta}}{\esssup}\int_{B_{8\varrho}^c}\frac{(u-k_j)_{-}(y,t)}{|x-y|^{n+sp}}\,dy
%		+ C\int_{B_{8\varrho}\times\{-\de \varrho^{sp}\}} \mathfrak g_- (u,k_j) \,dx,
	\end{multline*}
%\hrule	
Now to deal with the last integral, we make use of \cref{lem:g} to  get 
\begin{equation*}
	\mathfrak g_-^q (u,k)
	\le
	\bsc_q  \lbr|u|+|k|\rbr^{q-2}(u-k)_-^2.
\end{equation*}
\begin{description}
	\item[Case $1<q\leq 2$:] In this case, we make use of $(u-k)_-\le |u|+|k|$ along with $u\geq \bsmu^-$ to get  \begin{equation*}
		\mathfrak g_-^q (u,k)
		\leq
		\bsc_q  \lbr u-k \rbr_-^q \chi_{\{u\leq k\}}
		\leq
		\bsc_q \left(\sigma \ve \bsom \right)^{q}.
	\end{equation*}
	\item[Case $q \geq 2$:] In this case, we make use of the bounds $(u-k)_-\le |u|+|k|$, $u\geq \bsmu^-$ and $|\bsmu^{-}|\leq  ( \sigma \ve \bsom)$ to get
	\begin{equation*}
		\mathfrak g_-^q(u,k)
		\leq
		\bsc_q  \lbr |u|+|k|\rbr^q \chi_{\{u\leq k\}}
		\leq
		\bsc_q \left( \sigma \ve \bsom \right)^{q}.
	\end{equation*}
\end{description}
 In particular,  for all $1<q<\infty$, we get
\begin{equation*}%\label{4.12k}
	\int_{B_{\theta_xc_o\varrho}\times t_0 } \mathfrak g_-^q (u,k) \,dx .
	\leq \bsc_q \left( \sigma \ve \bsom \right)^{q} 
	|B_{\theta_xc_o\varrho}|.
\end{equation*}	
%	\hrule

	We now invoke \cref{lem:shrinking} with $l = \bsmu^- + \sigma \ve\bsom$, $k = \bsmu^- + \tfrac12 \sigma \ve\bsom$ and $m = \bsmu^- + \ve\bsom$ to get
%	\hrule
	\[
	(\tfrac12 \sigma \ve \bsom)(\ve\bsom)^{p-1}(1-\sigma)^{p-1} \frac{\al |B_{\theta_x c_o\varrho}|}{(\theta_xc_o\varrho)^{n+sp}} |\{-(\bsmu^--u) \leq \tfrac12 \sigma \ve \bsom\} \cap \mcQ| \leq \mopC (\sigma\ve \bsom)^{q}   \lbr 1+\tfrac{\de_t}{\de_x^{sp}}\rbr |B_{\theta_xc_o\varrho}|.
	\]
	In particular, we get
	\[
	\al \tfrac{\sigma}{2} (\ve\bsom)^p (1-\sigma)^{p-1} \frac{|B_{\theta_x c_o\varrho}|}{|\mcQ|} \frac{\theta_t}{\theta_x^{sp}}
	|\{-(\bsmu^--u) \leq \tfrac12 \sigma \ve \bsom\} \cap \mcQ| \leq \mopC (\sigma\ve \bsom)^{q}   \lbr 1+\tfrac{\de_t}{\de_x^{sp}}\rbr |B_{\theta_xc_o\varrho}|,
	\]
	which becomes
	\[
	\al \tfrac{\sigma}{2} (\ve\bsom)^p (1-\sigma)^{p-1} \frac{|B_{\theta_x c_o\varrho}|}{|\mcQ|} \frac{\de_t}{\de_x^{sp}}(\sigma \ve \bsom)^{q-p}
	|\{-(\bsmu^--u) \leq \tfrac12 \sigma \ve \bsom\} \cap \mcQ| \leq \mopC (\sigma\ve \bsom)^{q}   \lbr 1+\tfrac{\de_t}{\de_x^{sp}}\rbr |B_{\theta_xc_o\varrho}|.
	\]
	This gives
	\[
	|\{-(\bsmu^--u) \leq \tfrac12 \sigma \ve \bsom\} \cap \mcQ| \leq \mopC \lbr \tfrac{\de_x^{sp}}{\de_t}+1\rbr \tfrac{\sigma^{p-1}}{\al (1-\sigma)^{p-1}} |\mcQ|,
	\]
which completes the proof. %	holds for a constant $\mopC>0$ depending only on data. 
\end{proof}
\begin{remark}
	In \cref{lemma4.21}, \cref{lemma4.31}, \cref{lemma4.41} and \cref{lemma4.51}, if we make the choice $\de_t = \de_x^{sp}$, then all the constants will become independent of $\de_x$ and $\de_t$. 
\end{remark}

\section{Reduction of oscillation for \texorpdfstring{\cref{holderparabolic}}.}
%\subsection{Proof of H\"older continuity when $1<q<2+\ep_{p_o}$ and  $1<p_o<p<2+\ep_{p_o}$ for any fixed $p_o>1$}

%\begin{definition}\label{Def:6.1}
%	For a fixed $p_o \in (1,2)$, let us take $\mfd = 2+\ep_{p_o}$  with $\epsilon_{p_o} := \tfrac{sp_o  (p_o -1)}{n+3s}$ and consider $p,q$ in the range $1<q<\mfd$ and $p_o<p< \mfd$.
%\end{definition}

Let us consider the reference cylinder $\mcQ_o = B_{R} \times (-R^{sp},0)$ for some fixed $R \in \RR^+$.  Without loss of generality, we have taken $(x_o, t_o) = (0, 0)$.
Let us take constant
\begin{equation}\label{defbsom}
	\bsom  = 2\esssup_{\mcQ_o} |u| + \tail(|u|,\mcQ_o).	
\end{equation}
Consider the cylinder 
\[
\mcQ_1 = \mcQ_{\varrho} (\bsom) = B_{ \bsom^{\frac{\mfd-q}{sp}} \varrho} \times (-  \bsom^{\mfd-p} \varrho^{sp},0].
\]
Then we choose $\varrho$ small such that the following inclusion holds:
\begin{equation*}%\label{sizerhosing6}
B_{ \bsom^{\frac{\mfd-q}{sp}} \varrho} \times (-  \bsom^{\mfd-p} \varrho^{sp},0] \subset \mcQ_o \quad \Longrightarrow \quad \varrho := R\min\{\bsom^{\frac{q-\mfd}{sp}},\bsom^{\frac{p-\mfd}{sp}}\}.
\end{equation*}
 Note that we have $\mfd > \max\{q,p\}$  and hence the cylinders are shrinking in size as $\bsom \searrow 0$. 
 From \cref{defbsom}, we trivially have
 \begin{equation*}%\label{intrelation}
 \essosc_{\mcQ_{\varrho}(\bsom)} u %= \essosc_{B_{ \bsom^{\frac{\mfd-q}{sp}} (c_o\varrho)} \times (-  \bsom^{\mfd-p} (c_o\varrho)^{sp},0]} u
  \leq \bsom .
 \end{equation*}
  Let us now define 
\[  
\bsmu^+:=  \esssup_{\mcQ_1}u \qquad\text{and} \qquad \bsmu^-:= \essinf_{\mcQ_1} u.
\]

We can easily deduce following tail estimate on $\mcQ_1$ by the use of definitions of $\bsom, \varrho$ and the $\tail$.
\begin{multline*}
\tailp((u-\bsmu^{\pm})_{\pm};\mcQ_1) = \bsom^{\mfd-q} \varrho^{sp}\underset{t \in (-\bsom^{\mfd - p} \varrho^{sp},0)}{\esssup}\,\int_{\RR^n \setminus B_{ \bsom^{\frac{\mfd-q}{sp}} \varrho} }\frac{(u -\bsmu^{\pm} )_{\pm}^{p-1}(y,t)}{|y|^{n+sp}}\,dy\\
\begin{array}{rcl}
&\leq&\mopC_{n,p,s} \left[ \bsom^{p-1} +  \frac{\bsom^{\mfd-q} \varrho^{sp}}{R^{sp}}   \tailp(u,\mcQ_o) +  \bsom^{\mfd-q} \varrho^{sp} \underset{t \in (-\bsom^{\mfd - p} \varrho^{sp},0)}{\esssup}\,\int_{B_R \setminus B_{ \bsom^{\frac{\mfd-q}{sp}} \varrho} }\frac{|u|^{p-1}(y,t)}{|y|^{n+sp}}\,dy \right] \\
&\overset{\cref{defbsom}}{\leq} &\mopC_{n,p,s} \bsom^{p-1}.
\end{array}
\end{multline*}
%	\hrule

We will now obtain the reduction of oscillation in $\mcQ_1$ cylinder.  For simplicity we will always assume that all the time intervals mentioned below are always contained inside $(- \bsom^{\mfd -p} \varrho^{sp},0]$. 

\begin{remark}
The two main results in this section that we will show are reduction of oscillation from \cref{prop6.13} when $|\bsmu^{\pm}| \leq \tau \bsom$ holds and \cref{prop6.14} when $\tau \bsom \leq \pm \bsmu^{\pm} \leq 2 \bsom$ holds. The remaining case  corresponds to the away from zero situation which is dealt with in \cref{section7}. 

\end{remark}

\begin{proposition}\label{prop6.13}
	Let $u$ be locally bounded, weak solution of \cref{maineq} and $\mfd$ satisfy $\max\{p,q\} < \mfd < q + \tfrac{p-1}{1+\frac{n}{sp}}$. Furthermore, for a given $\alpha \in (0,1)$,  we assume the following is satisfied:
	\[
	|\{\pm (\bsmu^{\pm} -u(\cdot,t_o))\geq \tfrac{\bsom}{4} \}\cap B_{ (\frac{\bsom}{4})^{\frac{\mfd-q}{sp}} (c_o\varrho)}| \geq \alpha |B_{ (\frac{\bsom}{4})^{\frac{\mfd-q}{sp}} (c_o\varrho)}|.
	\]
	Then there exists constants $\delta_o = \delta_o(\datanb{,\alpha}) \in (0,\tfrac14)$, $\eta_o = \eta_o(\datanb{,\alpha}) \in (0,\tfrac14)$ and $ c_o(\eta_o) \in (0,\tfrac14)$ satisfying
	\[
	c_o^{\frac{sp}{p-1}} \leq \eta_o  \qquad  \text{and} \qquad |\bsmu^{\pm}| \leq \eta_o  \bsom \; ,
	\]
	such that  the following conclusion holds:
	\[
	\pm(\bsmu^{\pm} - u) \geq \eta_o \bsom \quad \text{on} \quad B_{(\frac{\bsom}{4})^{\frac{\mfd-q }{sp}} \tfrac{c_o \varrho}{2}} \times (t_o + \tfrac{1}{2} \de_o (\tfrac{\bsom}{4})^{\mfd-p}(c_o\varrho)^{sp} ,t_o+ \de_o (\tfrac{\bsom}{4})^{\mfd-p}(c_o\varrho)^{sp}].
	\]

\end{proposition}
%\hrule
\begin{proof}
	It is easy to see that $B_{(\frac{\bsom}{4})^{\frac{\mfd-q }{sp}} (c_o \varrho)} \times (t_o  ,t_o+ \de_o (\tfrac{\bsom}{4})^{\mfd-p}(c_o\varrho)^{sp}] \subset \mcQ_1$ provided  we choose $t_o$ such that $(t_o  ,t_o+ \de_o (\tfrac{\bsom}{4})^{\mfd-p}(c_o\varrho)^{sp}] \subset (-  \bsom^{\mfd-p} \varrho^{sp},0]$. This is possible since $\de_o (\tfrac{\bsom}{4})^{\mfd-p}(c_o\varrho)^{sp} \leq  \bsom^{\mfd-p} \varrho^{sp}$ holds from the fact that $\de_o, c_o \in (0,\tfrac14)$. 
	The proof now follows in several steps:
%	\hrule
	\begin{description}[leftmargin=*]
		\descitemnormal{Step 1:}{6step1sing} We apply \cref{lemma4.41} with $\de_x=1$ and $\ve =\tfrac14$ to get 
		\[
		|\{\pm (\bsmu^{\pm} - u(\cdot,t))\geq \de \tfrac{\bsom}{4}\} \cap B_{(\tfrac{\bsom}{4})^{\frac{\mfd-q}{sp}} (c_o\varrho)} | \geq \tfrac12\alpha |B_{ (\tfrac{\bsom}{4})^{\frac{\mfd-q}{sp}} (c_o\varrho)}| \txt{for a.e} t \in (t_o,t_o+\bar\varepsilon(\tfrac{\bsom}{4})^{\mfd-p}(c_o\varrho)^{sp}], 
		\]
		where $\bar\varepsilon$ and $\delta$ (both depending on $\al$ and data) are as obtained in  \cref{lemma4.41}. Note that the condition $|\bsmu^{\pm}| \leq 8\ve\bsom$ in \cref{lemma4.41} is automatically satisfied when $\ve = \tfrac14$.    To ensure the tail alternative holds, we estimate as follows (see \cref{Rmk:5.1}) :
		\[
		\begin{array}{rcl}
			c_o^{{sp}}\tailp((u-\bsmu^{\pm})_{\pm},  (\tfrac{\bsom}{4})^{\frac{\mfd -q}{sp}} \varrho, (t_o,t_o+\bar\varepsilon(\tfrac{\bsom}{4})^{\mfd-p}(c_o\varrho)^{sp}) ) & \leq & c_o^{sp} 4^{q-\mfd} \tailp((u-\bsmu^{\pm})_{\pm}, \mcQ_1)\\
			&\leq & \mopC_{n,p,s} c_o^{sp} 4^{q-\mfd}\bsom^{p-1},
		\end{array}
		\]
		Hence if we choose $c_o^{sp} \leq \tfrac{1}{\mopC_{n,p,s}} 4^{1-p}$, then the tail alternative in \cref{lemma4.41} is also satisfied. 
%				\hrule
		\descitemnormal{Step 2:}{6step2sing} In this step, we want to apply \cref{lemma4.51}.  Let us choose $\de_x = (\de\sigma)^{\frac{q-\mfd}{sp}}$ and $\de_t = \bar\ve \de^{p-\mfd}$,  and we will first show that the measure hypothesis is satisfied.  Notice that
		\[
		B_{(\tfrac{\bsom}{4})^{\frac{\mfd-q}{sp}} (c_o\varrho)} \times (t_o,t_o+\bar\varepsilon(\tfrac{\bsom}{4})^{\mfd-p}(c_o\varrho)^{sp}) = B_{\de_x (\sigma\de \tfrac{\bsom}{4})^{\frac{\mfd-q}{sp}}(c_o\varrho)} \times (t_o, t_o+ \de_t (\de \tfrac{\bsom}{4})^{\mfd-p}(c_o\varrho)^{sp}).
		\]
		Since $\sigma \in (0,1)$ and $\mfd>p$, we see that $\de_t (\de \tfrac{\bsom}{4})^{\mfd-p}(c_o\varrho)^{sp} \geq \de_t (\sigma\de \tfrac{\bsom}{4})^{\mfd-p}(c_o\varrho)^{sp}$. Hence   for all $\hat{t} \in (t_o+\de_t (\sigma\de \tfrac{\bsom}{4})^{\mfd-p}(c_o\varrho)^{sp} ,  t_o +\bar\varepsilon(\tfrac{\bsom}{4})^{\mfd-p}(c_o\varrho)^{sp} ]$, we have the measure information over the cylinder 
		\[
		\mcQ = B_{\de_x (\sigma\de \tfrac{\bsom}{4})^{\frac{\mfd-q}{sp}}(c_o\varrho)} \times (\hat{t} - \de_t (\sigma\de \tfrac{\bsom}{4})^{\mfd-p}(c_o\varrho)^{sp} ,  \hat t] \txt{with} \de_x =(\de\sigma)^{\frac{q-\mfd}{sp}} \quad \text{and} \quad \de_t = \bar\ve \de^{p-\mfd}.
		\]
		Thus, applying \cref{lemma4.51}, we get
		\begin{equation*}
			|\{\pm(\bsmu^{\pm}-u) \leq \tfrac12 \sigma \de \tfrac{\bsom}{4}\} \cap \mcQ| \leq \mopC \lbr \tfrac{\de_x^{sp}}{\de_t}+1\rbr \tfrac{\sigma^{p-1}}{\al (1-\sigma)^{p-1}} |\mcQ| \approx \mopC |\mcQ| .
		\end{equation*}
		%where we have denoted $\mcQ = B_{\theta_x c_o \varrho} \times (t_o, t_o+ \de_t (\sigma\ve\bsom)^{\mfd-p}(c_o\varrho)^{sp})$.
		
	As in \descrefnormal{6step1sing}{Step 1},	we take $|\bsmu^{\pm}| \leq \tfrac14 \sigma \de \bsom $ and note that the tail alternative holds if we choose $c_o^{sp} \leq  \tfrac{1}{\mopC_{n,p,s}} (\tfrac14 \sigma \de)^{p-1}$.
%		\hrule
		\item[Step 3:] We now apply \cref{lemma4.21} to get $\nu_1 \approx_{\datanb{}} \lbr \tfrac{1}{\de_x^{sp}} + \tfrac{1}{\de_t}\rbr^{-\frac{n+sp}{sp}} (\de_x^n\de_t)^{-1}$ such that the following conclusion follows:
			\begin{equation*}
			\pm\left(\bsmu^{\pm}-u\right)\geq\tfrac{1}{16}\sigma \de \bsom
			\quad
			\mbox{ on }\quad B_{ \de_x (\sigma\de \tfrac{\bsom}{4})^{\frac{\mfd-q}{sp}} \frac{ c_o \varrho}{2}} \times (\hat t - \de_t (\sigma\de \tfrac{\bsom}{4})^{\mfd-p}(\tfrac12c_o\varrho)^{sp},  \hat t].
		\end{equation*}
	%Again we can choose $c_o^{sp} \leq \sigma \de\ve$ small enough such that the tail alternative does not hold. 
	\item[Step 4:] We need to choose $\sigma$ such that 
	\[
	\mopC \lbr \tfrac{\de_x^{sp}}{\de_t}+1\rbr \tfrac{\sigma^{p-1}}{\al (1-\sigma)^{p-1}} \leq \mopC \lbr \tfrac{1}{\de_x^{sp}} + \tfrac{1}{\de_t}\rbr^{-\frac{n+sp}{sp}} (\de_x^n\de_t)^{-1}.
	\]
	Recalling $\de_x =(\de\sigma)^{\frac{q-\mfd}{sp}}$ and  $\de_t = \bar\ve \de^{p-\mfd}$, the left hand side has the form
	\[
	\mopC \lbr \tfrac{\de_x^{sp}}{\de_t}+1\rbr \tfrac{\sigma^{p-1}}{\al (1-\sigma)^{p-1}} \approx \mopC (\sigma^{q-\mfd} +1) \sigma^{p-1} = \mopC (\sigma^{p+q-\mfd-1} + \sigma^{p-1}).
	\]
	The right hand side has the form 
	\[
%	\begin{array}{rcl}
		\mopC \lbr \tfrac{1}{\de_x^{sp}} + \tfrac{1}{\de_t}\rbr^{-\frac{n+sp}{sp}} (\de_x^n\de_t)^{-1}  \approx  \mopC \frac{\sigma^{(\mfd-q)\frac{n}{sp}}}{ \lbr \sigma^{\mfd-q} + 1 \rbr^{1+\frac{n}{sp}}} \\
		 \overset{\sigma\in(0,\frac12)}{\geq}   \mopC \sigma^{(\mfd-q)\frac{n}{sp}}.
%		& = & \mopC \lbr \frac{1}{\sigma^{p-2}+1}\rbr\lbr  \frac{1}{\lbr 1 + \sigma^{2-p} \rbr^{\frac{n}{sp}}}\rbr\\
%		& \leq BAD & \mopC. 
%	\end{array}
	\]
	Based on the above calculations, we need to ensure the following are satisfied:
	\begin{enumerate}[(i)]
		\item $p+q-\mfd -1 > 0$ which is equivalent to $\mfd < p+q-1$.
		\item $p+q-\mfd -1 > (\mfd-q)\tfrac{n}{sp} $ which is equivalent to $\mfd < q + \frac{p-1}{1+\frac{n}{sp}}$. 
		\item $p-1 >  (\mfd-q)\tfrac{n}{sp} $ which is equivalent to $\mfd < q + \frac{p-1}{\frac{n}{sp}}$.
	\end{enumerate}
In particular, we need to ensure 
\begin{equation*}%\label{conditionond}
	\max\{p,q\} < \mfd < q + \frac{p-1}{1+\frac{n}{sp}}.
\end{equation*}
%	This can be ensured if we select $\epsilon =  \tfrac{sp(p-1)}{n+sp} $. In particular, if we fix some $1<p_o<2$, then $p_o-1>0$ and we can choose $\epsilon = \epsilon_{p_o} =  \tfrac{sp_o  (p_o -1)}{n+3s}$ such that we have reduction of oscillation in the range $p \in (p_o,2+\ve_{p_o})$ and $q \in (2,2+\ve_{p_o})$. 
%This can be ensured if we select $\epsilon = \min\left\{p-1, \tfrac{sp(p-1)}{n}, \tfrac{sp(p-1)}{n+sp}\right\}$. In particular, if we fix some $1<p_o<2$, then $p_o-1>0$ and we can choose $\epsilon = \epsilon_{p_o} = \min\left\{p_o-1, \tfrac{sp_o(p_o-1)}{n}, \tfrac{sp_o(p_o-1)}{n+3s}\right\}$ such that we have reduction of oscillation in the range $(p_o,2+\ve_{p_o})$. 
\item[Step 5:] From step 3, we have for all $\hat{t} \in (t_o+\de_t (\sigma\de\tfrac{\bsom}{4})^{\mfd-p}(c_o\varrho)^{sp} ,  t_o +\bar\varepsilon(\tfrac{\bsom}{4})^{\mfd-p}(c_o\varrho)^{sp} ]$,
\begin{equation*}
			\pm\left(\bsmu^{\pm}-u(\cdot,  \hat{t})\right)\geq\tfrac{1}{16}\sigma \de \bsom
			\quad
			\mbox{ on }\quad B_{ \de_x (\sigma\de \tfrac{\bsom}{4})^{\frac{\mfd-q}{sp}} \frac{ c_o \varrho}{2}}.
		\end{equation*}
	Now we can further choose $\sigma$ small if needed such that 
	$$
\de_t (\sigma\de\tfrac{\bsom}{4})^{\mfd-p}(c_o\varrho)^{sp} = \bar\ve	(\sigma \tfrac{\bsom}{4})^{\mfd-p}(c_o\varrho)^{sp} \leq \tfrac12 \bar\varepsilon(\tfrac{\bsom}{4})^{\mfd-p}(c_o\varrho)^{sp}.
	$$
this can be ensure by taking $\sigma \leq (\tfrac{1}{2})^{\frac{1}{\mfd -p}}$ and we have that 
$$
(t_o + \tfrac12 \bar\varepsilon(\tfrac{\bsom}{4})^{\mfd-p}(c_o\varrho)^{sp} , t_o +\bar\varepsilon(\tfrac{\bsom}{4})^{\mfd-p}(c_o\varrho)^{sp}] \subset (t_o+\de_t (\sigma\de\tfrac{\bsom}{4})^{\mfd-p}(c_o\varrho)^{sp} ,  t_o +\bar\varepsilon(\tfrac{\bsom}{4})^{\mfd-p}(c_o\varrho)^{sp} ].
	$$
	\end{description}
This  completes the proof of the lemma. 
\end{proof}
%\hrule\hrule\hrule\hrule

Next we focus on the case when $\tau \bsom \leq \pm \bsmu ^{\pm} \leq 2 \bsom$ for some fixed $\tau \in (0,\eta_o)$ yet to be chosen depending only on data, $\eta_o$ is from \cref{prop6.13} and we will prove the reduction of oscillation.  We start by writing down three main lemmas, the first being the propagation of measure information forward in time. 

%\hrule

\begin{lemma}\label{Lm:6:2}
	Suppose $ \tau \bsom \leq \pm \bsmu ^\pm \leq 2 \bsom$ for some fixed $\tau$ yet to be chosen with $\mfd > 2$ and 
	\begin{equation*}%\label{Eq:6.31}
		|\{u(\cdot,t^{\ast})\leq \bsmu^\pm \mp \tfrac14 \bsom\}\cap B_{\lbr \tfrac{\bsom}{4}\rbr^{\frac{\mfd-q}{sp}}(c_o\varrho)}|\geq \nu |B_{\lbr \tfrac{\bsom}{4}\rbr^{\frac{\mfd-q}{sp}}(c_o\varrho)}|. 
	\end{equation*}	
	%Let $\mft$ and $\htt$ be as in \cref{claim3},  
	Then there exists  $\varepsilon = \varepsilon(\datanb{,\nu,  \tau })\in(0,\tfrac{\tau}{2})$ and $c_o = c_o(n,p,s,\varepsilon)$,  such that the following two conclusion holds:
	\[	\begin{cases*}
		\left|\left\{ u(\cdot, t)\le\bsmu^\pm \mp \varepsilon\boldsymbol \om\right\}\cap B_{\lbr \tfrac{\bsom}{4} \rbr^{\frac{\mfd - q}{sp}}c_o\varrho}\right|
		\ge
		\tfrac12\nu |B_{\lbr \tfrac{\bsom}{4} \rbr^{\frac{\mfd - q}{sp}} c_o\varrho}|
		\txt{for all} t\in (t^{\ast} , t^{\ast} +   \lbr \ve \bsom\rbr^{\mfd-p}(c_o\varrho)^{sp}],\\
		c_o^{\frac{sp}{p-1}} \tail(( u-\bsmu^{\pm})_{\pm}; \lbr \tfrac{\bsom}{4} \rbr^{\frac{\mfd - q}{sp}} \varrho  ; (t^{\ast} , t^{\ast} +   \lbr \ve \bsom\rbr^{\mfd-p}(c_o\varrho)^{sp}] ) \leq   \varepsilon \bsom.
	\end{cases*}
	\]
	
\end{lemma}
\begin{proof}
	
	%	The tail estimate follows from Step 1 from  \cref{5.2.1} and so we are only left to prove the propagation of measure conclusion.
	We prove the case for $\bsmu^{+}$ only because the case for $\bsmu^-$ is analogous noting that the proof is very similar to \cite[Lemma 5.13]{adimurthi2025localholderregularitybounded}.
	%Proof of this lemma is same as \cref{Lm:5:2}	So we will outline the proof for the sake of complicity.
	To ensure the tail alternative holds, we estimate as follows:
	\[
	\begin{array}{rcl}
		c_o^{\frac{sp}{p-1}} \tail(( u-\bsmu^{+})_{+}; \lbr \tfrac{\bsom}{4} \rbr^{\frac{\mfd - q}{sp}} \varrho  ; (t^{\ast} , t^{\ast} +   \lbr \ve \bsom\rbr^{\mfd-p}(c_o\varrho)^{sp}] ) & \leq & c_o^{\frac{sp}{p-1}}\frac{(\tfrac{\bsom}{4})^{\mfd-q}\varrho^{s p}}{R^{sp}}\tail((u-\bsmu^{\pm})_{\pm}, R,(-R^{sp},0))\\
		&\leq & c_o^{\frac{sp}{p-1}} 4^{q-\mfd}\bsom,
	\end{array}
	\]
	Hence if we choose $c_o^{\frac{sp}{p-1}} \leq 4^{\mfd-q} \ve$, then the tail alternative is automatically satisfied.  
%	\hrule
	
	Let us now  take $k = \bsmu^+ - \varepsilon\bsom$ for some $\varepsilon\in (0,\tfrac{\tau}{2})$ to be eventually chosen. From the choice of $\varepsilon$ and $\bsmu^+ \geq \tau \bsom$, we see that $k \geq \tfrac{1}{2}\tau \bsom$.  We apply the energy estimate  from \cref{Prop:energy} for $(u-k)_{+}$ over the cylinder $B_{\theta_x c_o\varrho}\times(t^{\ast},\bar{t}]$ where $\theta_x =\lbr \frac{\bsom}{4} \rbr^{\frac{\mfd - q}{sp}}$ and $\bar{t}$ yet to be chosen. 
	% Now following step by step as in \cref{Lm:5:2} to get
	%We apply the energy estimate  from \cref{Prop:energy} for $(u-k)_{+}$ over the cylinder $B_{\theta_x c_o\varrho}\times(t^{\ast},\bar{t}]$.  %with $k=\bsmu^-+\ve\bsom$ where $\delta>0$ will be chosen later in the proof. 
	 For $\sigma \in (0,\tfrac18]$ to be chosen later, we take a cutoff function $\zeta = \zeta(x)\geq 0$,  such that it is supported in $B_{ \theta_x c_o\varrho(1-\frac{\sigma}{2})}$ with $\zeta \equiv 1$ on $B_{(1-\sigma)\theta_xc_o\varrho}$ and $|\nabla \zeta| \apprle \tfrac{1}{\sigma\theta_xc_o\varrho}$. Making use of  \cref{Prop:energy}  and estimate analogously as \cref{lemma4.31} ,  \cref{lemma4.41}   and using $u\leq\bsmu^+$ along with $(u-k)_{+} \leq \ve\bsom$ in $B_{\theta_xc_o\varrho}$, we get
%	 \hrule	
	\begin{equation}\label{Eq:6.10A}
		\begin{array}{rcl}
			\int_{B_{(1-\sigma)\theta_xc_o\varrho}\times\{\bar{t} \}} \mathfrak g_+^q(u,k) \zeta^p\,dx  
			&\leq&   \mopC\frac{ (\ve\bsom)^p}{\sigma^{n+sp} \theta_x^{sp}(c_o\varrho)^{sp}}|\bar{t}-t^{\ast}||B_{\theta_xc_o\varrho}|
			+ \int_{B_{\theta_xc_o\varrho}\times\{t^{\ast}\}} \mathfrak g_+^q(u,k) \zeta^p \,dx.
			\end{array}
		\end{equation}
	Now we estimate the last term appearing on the right hand side of \cref{Eq:6.10A} as follows:
	\begin{equation*}
		\begin{array}{rcl}
		\int_{B_{\theta_xc_o\varrho}\times\{t^{\ast}\}} \mathfrak g_+^q(u,k) \zeta^p \,dx	& = & \int_{B_{\theta_xc_o\varrho}\times\{t^{\ast}\}} (q-1) \int_{k}^{u} s^{q-2}(s-k)_+\,ds \, dx  \\
			%& \leq & \mopC\frac{ (\ve\bsom)^p}{\sigma^{n+p} \theta_x^{sp}(c_o\varrho)^{sp}}|\mft-t_o||B_{\theta_xc_o\varrho}|
			%+  |A_{\ve\bsom}^{\varrho}(t_o) |(q-1) \int_{\bsmu^-}^{k} |s|^{q-2}(s-k)_-\,ds \\
			& \leq &  (1- \nu )|B_{\theta_xc_o\varrho}| (q-1) \int^{\bsmu^+}_{k} s^{q-2}(s-k)_+\,ds,
		\end{array}
	\end{equation*}
where to obtain the last inequality, we made use of the measure hypothesis. 
	We now estimate the term appearing on the left hand side of \cref{Eq:6.10A} as follows:
	\begin{align*}
		\int_{B_{(1-\sigma)\theta_xc_o\varrho}\times\{\bar{t} \}} \mathfrak g_+^q(u,k) \zeta^p\,dx    &= \int_{B_{(1-\sigma)\theta_xc_o\varrho}\times\{\bar{t} \}} (q-1) \int_{k}^{u} s^{q-2}(s-k)_+\,ds \, dx  \\
		&\geq \left|\left\{ u(\cdot, \bar t)\ge k_{\tilde{\ve}}  \right\}\cap B_{(1-\sigma)\theta_x c_o\varrho}\right| (q-1) \int_{k}^{k_{\tilde{\ve}}} s^{q-2}(s-k)_+\,ds,
	\end{align*}
	where $k_{\tilde\varepsilon} := \bsmu^+ - \varepsilon\tilde\varepsilon\bsom$ for some $\tilde\varepsilon\in(0,\tfrac12)$. Recalling $k \geq \tfrac{\tau \bsom}{2}$,
	%the we are in the case $\tau \bsom \leq \bsmu^+ \leq 2\bsom$, 
	we further get	
	\begin{equation*}
		\int_k^{k_{\tilde\varepsilon}} s^{q-2}(s-k)_+\,ds \geq \lbr \frac{\tau \bsom}{2}\rbr^{q-2} \int_k^{k_{\tilde\varepsilon}} (s-k)_+ \,ds =  \lbr\frac{\tau \bsom}{2} \rbr^{q-2} \frac{(\varepsilon\bsom)^2 (1-\tilde\varepsilon)^2}2 .
	\end{equation*} 
%	
	%with $k=\bsmu^-+\ve\bsom$ where $\delta>0$ will be chosen later in the proof. 
	%Note that $(u-k)_{+} \leq \ve\bsom$ in $B_{\theta_xc_o\varrho}$. For $\sigma \in (0,\tfrac18]$ to be chosen later, we take a cutoff function $\zeta = \zeta(x)\geq 0$,  such that it is supported in $B_{ \theta_x c_o\varrho(1-\frac{\sigma}{2})}$ with $\zeta \equiv 1$ on $B_{(1-\sigma)\theta_xc_o\varrho}$ and $|\nabla \zeta| \apprle \tfrac{1}{\sigma\theta_xc_o\varrho}$. We now make use of  \cref{Prop:energy}  and estimate analogously as \cref{lemma4.31} ,  \cref{lemma4.41}   and using $u\leq\bsmu^+$ to get	
%\hrule
	As in the proof of \cref{lemma4.41}, we have
	\begin{multline}\label{deg6.29}
		|\{u(\cdot,\bar t) >k_{\tilde\varepsilon}\}\cap B_{ \theta_x  c_o\varrho }| 
		\leq 
		\frac{\int_k^{\bsmu^+} s^{q-2}(s-k)_+\,ds }{\int_k^{k_{\tilde\varepsilon}} s^{q-2}(s-k)_+\,ds }
		(1- \nu)|B_{\theta_x c_o\varrho}|\\
		+
		\mopC\frac{4^{\mfd - 2}  \ve^{p-2} \bsom^{p-\mfd} |\bar{t}-t^{\ast}||B_{\theta_xc_o\varrho}| }{ \min \{ \tau^{q-2}, 2^{q-2}\} \sigma^{n+p} (q-1)(1-\tilde{\ve})^2 (c_o\varrho)^{sp}}  
		+n\sigma |B_{\theta_x c_o\varrho}| .
	\end{multline}
	 Furthermore, proceeding as in \cite[Lemma 6.2]{liaoHolderRegularitySigned2022},  we use $\tau \bsom  \leq \bsmu^+ \leq 2\bsom$ and $k \geq \frac{ \tau \bsom}{2}$ to get
	\begin{equation}\label{deg6.30}
		\frac{\int_k^{\bsmu^+} s^{q-2}(s-k)_+\,ds }{\int_k^{k_{\tilde\varepsilon}} s^{q-2}(s-k)_+\,ds }\leq 1 + \bsc_{q,\tau}\ve \tilde\varepsilon.
	\end{equation}
	Combining \cref{deg6.29} and \cref{deg6.30} and taking $\bar t \in (t^{\ast} , t^{\ast} +   \lbr \ve \bsom\rbr^{\mfd-p}(c_o\varrho)^{sp}]$, we get
	\begin{multline*}%\label{deg6.31}
		|\{u(\cdot,\bar t) >k_{\tilde\varepsilon}\}\cap B_{ \theta_x  c_o\varrho }| 
		\leq \lbr
		(1+\bsc_{q,\tau} \ve\tilde\varepsilon )
		(1-\nu)
		+
		\mopC\frac{ (4 \ve)^{\mfd-2} }{ \min\{\tau^{q-2}, 2^{q-2} \} \sigma^{n+p} (q-1)(1-\tilde{\ve})^2 }  
		+n \sigma  \rbr |B_{\theta_x c_o\varrho}|.
	\end{multline*}
	We first choose $\tilde\varepsilon$ small such that $(1+\bsc_{q,\tau} \tilde\varepsilon)
	(1-\nu) \leq 1-\tfrac34 \nu$ and then choose $\sigma = \tfrac{\nu}{8n}$ followed by $\varepsilon$ small such that $\mopC\frac{  (4 \ve)^{\mfd-2} }{\min\{ \tau^{q-2} , 2^{q-2} \} \sigma^{n+p} (q-1)(1-\tilde{\ve})^2 }   \leq \tfrac{\nu}{8}$ which is possible since $\mfd >2$.  Noting that $\mfd > p$, we can redefine $\tilde{\ve} \ve$ to be $\ve$ which completes the proof of the lemma. 
\end{proof}

%\hrule
We now prove measure shrinking lemma which is similar to \cite[Lemma 5.16]{adimurthi2025localholderregularitybounded}. % We will skip the proof since it follows similarly as in \cref{Lm:5.12}.
%Notice that we have not use the lower bound of $\pm \bsmu^\pm$ in the proof of \cref{Lm:5.12} thus same result holds true here also.
\begin{lemma}\label{Lm:6.12}
	Suppose $\tau \bsom\le \pm \bsmu^\pm \le2\bsom$ and for some positive $\sigma \leq \tfrac{1}{2}$ and $\ve \in (0,\tfrac{1}{8}) $ following is satisfied, 
	\begin{equation*}%\label{Eq66}
		\left|\left\{ u(\cdot, t)\le  \bsmu^\pm  \mp \varepsilon\boldsymbol \om\right\}\cap B_{\lbr \frac{\bsom}{4} \rbr^{\frac{\mfd - q}{sp}} c_o\varrho}\right|
		\ge
		\tfrac12\nu |B_{\lbr \frac{\bsom}{4} \rbr^{\frac{\mfd - q}{sp}} c_o\varrho}|
		\txt{for all} t\in (   t^{\ast} -  \lbr \sigma \ve \bsom\rbr^{\mfd-p} (c_o\varrho)^{sp},  t^{\ast}],
	\end{equation*}
along with 
\begin{equation*}%\label{deglm6.2tail}
	c_o^{\frac{sp}{p-1}} \tail(( u-\bsmu^{\pm})_{\pm}; \lbr \tfrac{\bsom}{4} \rbr^{\frac{\mfd - q}{sp}}  \varrho  ;  t^{\ast}-   \lbr \sigma \ve \bsom\rbr)^{\mfd-p} (c_o\varrho)^{sp}, t^{\ast}]) \leq   \sigma\varepsilon\bsom,
\end{equation*}
	then we have the following conclusion:
	\begin{equation*}
		|\{ \pm( \bsmu^{\pm}-u) \leq \tfrac12 \sigma \ve \bsom\} \cap \mcQ| \leq \mopC \lbr (\sigma \ve)^{2-\mfd} +1\rbr \max\{\tau^{q-2} ,1 \} \tfrac{\sigma^{p-1}}{\nu (1-\sigma)^{p-1}} |\mcQ|,
	\end{equation*}
	where $\mcQ = B_{ \lbr \frac{\bsom}{4} \rbr^{\frac{\mfd - q}{sp}} c_o \varrho} \times ( t^{\ast} -  \lbr \sigma \ve \bsom\rbr^{\mfd-p} (c_o\varrho)^{sp}, t^{\ast}]$
	and $\mopC = \mopC(\datanb{})$. %and %$\ve = \ve(\datanb{,\nu,\tau})$ is from \cref{Lm:4:2}  
%	provided the following is satisfied:
%	\begin{equation}\label{deglm6.2tail}
%		c_o^{\frac{sp}{p-1}} \tail(( u-\bsmu^{\pm})_{\pm}; \lbr \tfrac{\bsom}{4} \rbr^{\frac{\mfd - q}{sp}}  \varrho  ;  t^{\ast}-   \lbr \sigma \ve \bsom\rbr)^{\mfd-p} (c_o\varrho)^{sp}, t^{\ast}]) \leq   \sigma\varepsilon\bsom.
%	\end{equation}
\end{lemma}
\begin{proof}
	We prove the case for $\bsmu^{+}$ only because the case for $\bsmu^-$ is analogous. 	Let $\theta_x = \lbr \tfrac{\bsom}{4} \rbr^{\frac{\mfd - q}{sp}}  $ and $\theta_t = \lbr \sigma \ve \bsom\rbr^{\mfd-p}$.  We choose a test function $\zeta = \zeta(x)$ such that $\zeta\equiv 1$ on $B_{ \theta_x c_o \varrho}$, it is supported in $ B_{ \theta_x \frac32 c_o \varrho}$ and $|\nabla \zeta|\apprle \tfrac{1}{\theta_xc_o\varrho}$.  We write the energy estimate from  \cref{Prop:energy} over the cylinder $B_{ \theta_x \frac32 c_o \varrho} \times ( t^\ast -  \theta_t (c_o\varrho)^{sp},t^\ast]$ for the functions $(u-k)_+$ where $k = \bsmu^+-\sigma \ve \bsom$.
	
%	\hrule
	
	Using $(u-k)_{+} \leq \sigma \ve \bsom$ locally and  by our choice of the test function, applying \cref{Prop:energy}, we get
	%\hrule
	\begin{multline*}
		\int_{t^\ast -  \theta_t (c_o\varrho)^{sp}}^{t^\ast}\int_{B_{\theta_x \frac32 c_o\varrho}}|(u-k)_+(x,t)|\int_{B_{\theta_xc_o\varrho}}\frac{|(u-k)_-(y,t)|^{p-1}}{|x-y|^{n+sp}}\,dx\,dy\,dt
		\\
		\begin{array}{rcl}
			&\leq&  \mopC (\sigma\ve \bsom)^p   \frac{\theta_t (c_o\varrho)^{sp} }{\theta_x^{sp}(c_o\varrho)^{sp}} |B_{\theta_xc_o\varrho}|
			+ \int_{B_{\theta_x 2 c_o\varrho}\times   t^\ast -  \theta_t (c_o\varrho)^{sp} } \mathfrak g_+^q (u,k) \,dx \\
			& \leq &  \mopC \frac{(\sigma\ve \bsom)^{\mfd} \bsom^{q-\mfd}}{4^{q - \mfd}}  |B_{\theta_xc_o\varrho}|
			+ \int_{B_{\theta_x 2 c_o\varrho}\times t^\ast -  \theta_t (c_o\varrho)^{sp} } \mathfrak g_+^q (u,k) \,dx.
		\end{array}
		%		\underset{1.1ackrel{t \in (-\delta\varrho^{sp},0]}{ x\in \spt \zeta}}{\esssup}\int_{B_{8\varrho}^c}\frac{(u-k_j)_{-}(y,t)}{|x-y|^{n+sp}}\,dy
		%		+ C\int_{B_{8\varrho}\times\{-\de \varrho^{sp}\}} \mathfrak g_- (u,k_j) \,dx,
	\end{multline*}
	
	Now to deal with the last integral, we make use of \cref{lem:g} and $\tau \bsom \leq \bsmu^+ \leq 2\bsom$ to get
	%\begin{equation*}
	%	\mathfrak g_-^q (u,k)
	%	\le
	%	\bsc(q)  \lbr|u|+|k|\rbr^{q-2}(u-k)_-^2.
	%\end{equation*}
	% for $1< q < 2$, we  use $(u-k)_-\le |u|+|k|$, $u\geq \bsmu^-$
	% \begin{equation*}
	%		\mathfrak g_-^q (u,k)
	%		\leq
	%		\bsc_q  \lbr u-k \rbr_-^q \chi_{\{u\leq k\}}
	%		\leq
	%		\bsc_q \left(\sigma \ve \bsom \right)^{q}.
	%	\end{equation*}
	%and for $q \geq 2$, we further use $(u-k)_-\le |u|+|k|$, $u\geq \bsmu^-$ and $|\bsmu^{-}|\leq  ( \sigma \ve \bsom)$ to get
	%	\begin{equation*}
	%		\mathfrak g_-^q(u,k)
	%		\leq
	%		\bsc_q  \lbr |u|+|k|\rbr^q \chi_{\{u\leq k\}}
	%		\leq
	%		\bsc_q \left( \sigma \ve \bsom \right)^{q}.
	%	\end{equation*}
	% In particular,  for all $1<q<\infty$, we get
	\begin{equation*}
		\int_{B_{\theta_x 2 c_o\varrho}\times t^\ast -  \theta_t (c_o\varrho)^{sp} } \mathfrak g_+^q (u,k) \,dx .
		\leq \bsc_q \left( \sigma \ve \bsom \right)^{2} \max\{\tau^{q-2},1\} \bsom^{q-2} 
		|B_{\theta_xc_o\varrho}|.
		% \frac{\bsc_q }{\delta \Theta \varrho^{sp}}
		% \left(\frac{\epsilon \gh{\xi \bsom}}{2^j}\right)^{q} |\widehat{ \mcq}| = \frac{\bsc_q }{\delta  \varrho^{sp}} \left(\frac{\epsilon \gh{\xi \bsom}}{2^j}\right)^{p-q}
		% \left(\frac{\epsilon \gh{\xi \bsom}}{2^j}\right)^{q} |\widehat{ \mcq}| = \frac{\bsc_q }{\delta  \varrho^{sp}} 
		% \left(\frac{\epsilon \gh{\xi \bsom}}{2^j}\right)^{p} |\widehat{ \mcq}|.
	\end{equation*}

	We now invoke \cref{lem:isop} for $\bsmu^+$ 
%	\[
%	(l-k)^{p-1} (m-l) \abs{[u>m]\cap B_{\theta_x c_o \varrho}}\abs{[u<k]\cap B_{\theta_x c_o \varrho}} \leq 
%	C(\theta_x c_o\varrho)^{n+sp}\int_{B_{\theta_x \frac32 c_o\varrho}} (u-l)_{+}(x)\int_{B_{\theta_x c_o \varrho}}\frac{(u-l)_{-}^{p-1}(y)}{|x-y|^{n+sp}}\,dy\,dx,
%	\]	
	with $l = \bsmu^+ -\frac12 \sigma \ve\bsom$, $k = \bsmu^+ -  \ve\bsom$ and $m = \bsmu^+ - \sigma \ve\bsom$ to get
	%	\hrule
	\[
	(\tfrac12 \sigma \ve \bsom)(\ve\bsom)^{p-1}(1-\sigma)^{p-1} \frac{\nu |B_{\theta_x c_o\varrho}|}{(\theta_xc_o\varrho)^{n+sp}} |\{\bsmu^+ -u  \leq  \sigma \ve \bsom\} \cap \mcQ| \leq \mopC \max\{\tau^{q-2},1\}  \bsom^{q}   \lbr (\sigma\ve)^{\mfd} + (\sigma\ve)^2 \rbr |B_{\theta_xc_o\varrho}|.
	\]
	In particular, we get
	\[
	\nu \tfrac{\sigma}{2} (\ve\bsom)^p (1-\sigma)^{p-1} \frac{|B_{\theta_x c_o\varrho}|}{|\mcQ|} \frac{\theta_t}{\theta_x^{sp}}
	|\{\bsmu^+ -u  \leq  \sigma \ve \bsom\} \cap \mcQ| \leq \mopC \max\{\tau^{q-2},1\} \bsom^{q}   \lbr (\sigma\ve)^{\mfd} + (\sigma\ve)^2 \rbr |B_{\theta_xc_o\varrho}|,
	\]
	which becomes
	\[
	\nu \tfrac{\sigma}{2} (\ve\bsom)^p (1-\sigma)^{p-1}  \frac{(\sigma \ve \bsom)^{\mfd -p}}{(\tfrac{\bsom}{4} )^{\mfd - q}}
	|\{ \bsmu^+ -u  \leq  \sigma \ve \bsom\} \cap \mcQ| \leq \mopC \max\{\tau^{q-2},1\} \bsom^{q}   \lbr (\sigma\ve)^{\mfd} + (\sigma\ve)^2 \rbr |\mcQ|.
	\]
	This gives
	\[
	|\{ \bsmu^+ -u \leq  \sigma \ve \bsom\} \cap \mcQ| \leq \mopC \max\{\tau^{q-2},1\}  \lbr (\sigma\ve)^{2-\mfd} + 1 \rbr \tfrac{\sigma^{p-1}}{\nu (1-\sigma)^{p-1}} |\mcQ|,
	\]
	which holds for a constant $\mopC>0$ depending only on data. 
\end{proof}

Now we prove the de Giorgi iteration.  %We will skip the proof since it follows similarly as in \cref{Lm:5.10}.
\begin{lemma}\label{Lm:6.10}
	Let $\tau\bsom\le \pm \bsmu^\pm \le2\bsom$ and given some positive $\sigma \leq \tfrac{1}{2}$ and $\ve \in (0,\tfrac{\tau}{2}) $,  set $\theta_t =  (\sigma \varepsilon \bsom)^{\mfd - p}$, $\theta_x =  (\tfrac{\bsom}{4} )^{\frac{\mfd -q }{sp}}$  and denote 
	$\mcQ_{c_o\varrho}^{\theta_x,\theta_t} = B_{\theta_x c_o\varrho} \times (t^{\ast} - \theta_t (c_o\varrho)^{sp},  t^{\ast}]$. Then there exists a constant $\nu_1 = \nu_1(\datanb{}) \in (0,1)$ such that if 
	\begin{equation*}
		\left|\left\{
		\pm (\bsmu^{\pm}-u) \leq \tfrac12 \sigma \varepsilon \bsom\right\}\cap \mcQ_{c_o\varrho}^{\theta_x,\theta_t}\right|
		\leq
		\nu_1 \tfrac{ \min \{ \tau^{q-2} , \tau^{(2-q)(1+\frac{n}{sp})} \} (\sigma \ve)^{2-\mfd}}{(1+(\sigma \ve)^{2-\mfd})^{\frac{n+sp}{sp}}}| \mcQ_{c_o\varrho}^{\theta_x,\theta_t}|,
	\end{equation*}
	holds along with the assumption 
	\begin{equation*}%\label{vwLm:3:3:hypothesis}
		c_o^{\frac{sp}{p-1}}\tail((u-\bsmu^\pm)_\pm ; \theta_x\varrho;  (t^{\ast}-\theta_t (c_o\varrho)^{sp},  t^{\ast}])\leq \sigma\varepsilon\bsom, 
	\end{equation*}
	then the following conclusion follows:
	\begin{equation*}
		\pm (\bsmu^{\pm}-u ) \geq\tfrac{1}4 \sigma \varepsilon \bsom
		\quad
		\mbox{ on }\quad \mcQ_{\tfrac{c_o\varrho}{2}}^{\theta_x,\theta_t} = B_{\theta_x\frac{ c_o\varrho}{2}} \times \left(  t^{\ast}-\theta_t (\tfrac{c_o\varrho}{2})^{sp}, t^{\ast} \right].
	\end{equation*}
\end{lemma}
\begin{proof}
	For $j=0,1,\ldots$, define
	\begin{equation*}%\label{choices:B_n}
		%			\left\{
		{\def\arraystretch{1.1}\begin{array}{cccc}
				k_j:=\bsmu^+-\tfrac{\sigma \ve \bsom}2 - \tfrac{\sigma \ve \bsom}{2^{j+1}},& \tilde{k}_j:=\tfrac{k_j+k_{j+1}}2,&
				\varrho_j:=\tfrac{c_o\varrho}2+\tfrac{c_o\varrho}{2^{j+1}},
				&\tilde{\varrho}_j:=\tfrac{\varrho_j+\varrho_{j+1}}2,\\%[5pt]
				B_j:=B_{\varrho_j}^{\theta_x},& \tilde{B}_j:=B_{\tilde{\varrho}_j}^{\theta_x},&
				\mcQ_j:=\mcQ_{\varrho_j}^{\theta_x,\theta_t},&
				\tilde{\mcQ}_j:=\mcQ_{\tilde\varrho_j}^{\theta_x,\theta_t}.
		\end{array}}
		%			\right.
	\end{equation*}
	Furthermore, we also define 
	\begin{equation*}%\label{choices:B_nn}
		%	\left\{
		%	{\def\arraystretch{1.1}\begin{array}{cccc}
		\hat{\varrho}_j:=\tfrac{3\varrho_j+\varrho_{j+1}}4, \quad 
		\bar{\varrho}_j:=\tfrac{\varrho_j+3\varrho_{j+1}}4,\quad 
		\hat{\mcQ}_j:=\mcQ_{\hat\varrho_j}^{\theta_x,\theta_t}, \quad
		\bar{\mcQ}_j:=\mcQ_{\bar\varrho_j}^{\theta_x,\theta_t}.
		%	\end{array}}
		%	\right.
	\end{equation*}
	We now consider a cutoff functions $\bar{\zeta_j}$ and $\zeta_j$ such that
	\begin{equation*}%\label{Lm:3.2:cutoff}
		\begin{array}{c}
			\bar{\zeta_j} \equiv 1 \text{ on } B_{j+1}, \quad \bar{\zeta_j} \in C_c^{\infty}(\bar{B}_{j}), \quad  |\nabla\bar{\zeta_j}| \apprle \frac{1}{\theta_x(\bar{\varrho_j} - \varrho_{j+1})} \approx \frac{2^j}{\theta_xc_o\varrho} \quad \text{and} \quad  |\pa_t\bar{\zeta_n}| \apprle \frac{1}{\theta_t(\bar\varrho_j^{sp} - \varrho_{j+1}^{sp})} \approx \frac{2^{jsp}}{\theta_t(c_o\varrho)^{sp}}, \\
			{\zeta_j} \equiv 1 \text{ on } \tilde{B}_{j}, \quad {\zeta_j} \in C_c^{\infty}(\hat{B}_{j}), \quad  |\nabla{\zeta_j}| \apprle \frac{1}{\theta_x(\hat{\varrho_n} - \tilde{\varrho_{j}})}\approx \frac{2^{j}}{\theta_xc_o\varrho} \quad \text{and} \quad  |\pa_t{\zeta_j}| \apprle \frac{1}{\theta_t(\hat\varrho_j^{sp} - \tilde{\varrho_{j}}^{sp})}\approx \frac{2^{jsp}}{\theta_t(c_o\varrho)^{sp}}.
		\end{array}
	\end{equation*}
	We now make the following observations:
	\begin{itemize}
		\item Since $\tau \bsom \leq \bsmu^+ \leq 2\bsom$ and $\varepsilon \in (0,\tfrac\tau2)$ and $0<\sigma \leq \tfrac12$,   we see that $$|u|+|k_j| \geq |k_j| \geq \tau \bsom - \frac{\varepsilon\bsom}{2} - \frac{\varepsilon\bsom}{2^{j+1}} \geq  \tau \bsom - \frac{\tau\bsom}{4} - \frac{\tau\bsom}{4} \geq \frac12 \tau \bsom . $$
		In particular, we have $\mathfrak g^q_+(u,\tilde{k}_j)  \apprge \min\{\tau^{q-2},1\} \bsom^{q-2}  (u-\tilde{k}_j)_+^2$.
		\item Let us denote $A_j := \{u(x,t) \geq k_j\} \cap \mcQ_j$.
		\item Using $\tau \bsom \leq \bsmu^+ \leq 2\bsom$, we  have $(u-k_j)_+ \leq \sigma \ve  \bsom$ and $\mathfrak g^q(u,k_j)_+ \apprle \max\{\tau^{q-2},1\} \bsom^{q-2} (\sigma \ve \bsom)^2$.
	\end{itemize}
	%\hrule
	Now, we  apply the energy estimates from \cref{Prop:energy} over  $\mcQ_j $ to $ (u-\tilde{k}_j)_{+}$ and $\zeta_j$  and estimate analogously as \cref{lemma4.21} and using above observation we get
	
	%\begin{multline}\label{degEq:3.6en_i}
	%	\tau^{p-2}\bsom^{p-2}	\underset{-\theta\varrho_i^{sp} < t < 0}{\esssup}\int_{B_i}(u-\tilde{k}_i)_{+}^2\zeta_i^p(x,t)\,dx  
	%	+ {\iiint_{\mcq_i}}\frac{|(u-\tilde{k}_i)_{+}(x,t)\zeta_i(x,t)-(u-\tilde{k}_i)_{+}\zeta_i(y,t)|^p}{|x-y|^{n+sp}}\,dx \,dy\,dt
	%	\\ 
	%	\begin{array}{rcl} &\leq_{\data{}} & \frac{2^{isp}}{(c_o\varrho)^{sp}}M^p |A_i| +
	%	\frac{2^{isp}}{\tht(c_o\varrho)^{sp}}M^2\bsom^{p-2} |A_i|\\
	%	&&+ \lbr \underset{\stackrel{-\theta\varrho_i^{sp} < t < 0;}{ x\in \spt \zeta_i}}{\esssup}\int_{\RR^n \setminus B_i}\frac{(u-\tilde{k}_i)_{+}^{p-1}(y,t)}{|x-y|^{n+sp}}\,dy\rbr   M |A_i|.
	%				\end{array}
	%\end{multline}		

	\begin{multline}\label{Eq:5.20A}
		%		\begin{aligned}
		\min\{\tau^{q-2},1\} \bsom^{q-2}	\underset{(-\theta_t \tilde\varrho_j^{sp},0)}{\esssup}
		\int_{\tilde{B}_j} (u-{k}_j)_+^2\,dx
		+\iiint_{\tilde{\mcQ}_j}\frac{|(u-{k}_j)_{+}(x,t)-(u-{k}_j)_{+}|^p}{|x-y|^{n+sp}}\,dx \,dy\,dt
		\\\leq
		\mopC 2^{(n+sp)j} |A_j| \lbr  \frac{( \sigma \ve \bsom)^{p}}{\theta_x^{sp}(c_o\varrho)^{sp}} +  \frac{\max\{\tau^{q-2},1\} \bsom^{q-2} (\sigma \ve \bsom)^2}{\theta_t (c_o \varrho)^{sp}}\rbr.
		%		\end{aligned}
	\end{multline}
	%\hrule
	Using Young's inequality, we have
	\begin{multline*}%\label{Eq:4.31A}
		|(u-{k}_{j})_+ \bar{\zeta_j}(x,t) - (u-{k}_{j})_+ \bar{\zeta_j}(y,t)|^p \leq c |(u-{k}_{j})_+(x,t) - (u-{k}_{j})_+ (y,t)|^p\bar{\zeta_j}^p(x,t) \\
		+ c |(u-{k}_{j})_+(y,t)|^p |\bar{\zeta_j}(x,t) - \bar{\zeta_j}(y,t)|^p,
	\end{multline*}
	from which we obtain the following sequence of estimates:
	\begin{equation*}
		\begin{array}{rcl}
			\frac{\sigma \ve \bsom}{2^{j+2}}
			|A_{j+1}|
			&\overred{6.41a}{a}{\leq} &
			\iint_{\tilde\mcQ_{j}}(u-{k}_j)_+\bar{\zeta_j}\,dx\,dt \\
			&\overred{6.41b}{b}{\leq} &
			\lbr\iint_{\tilde\mcQ_{j}}\left[(u-{k}_j)_+\bar{\zeta_j}\right]^{p\frac{n+2s}{n}}
			\,dx\,dt\rbr^{\frac{n}{p(n+2s)}}|A_j|^{1-\frac{n}{p(n+2s)}}\\
			&\overred{6.41c}{c}{\leq} &\mopC 
			\left(\iiint_{\tilde\mcQ_j}\frac{|(u-{k}_j)_{+}(x,t)\bar{\zeta_j}(x,t)-(u-{k}_j)_{+}\bar{\zeta_j}(y,t)|^p}{|x-y|^{n+sp}}\,dx \,dy\,dt\right)^{\frac{n}{p(n+2s)}} 
			\\&&\qquad\times \left(\underset{-\theta_t\tilde\varrho_j^{sp} < t < 0}{\esssup}\int_{\tilde{B}_j}[(u-\tilde{k}_j)_{+}\bar\zeta_j(x,t)]^2\,dx\right)^{\frac{s}{n+2s}}|A_j|^{1-\frac{n}{p(n+2s)}}\\
			&\overred{6.41d}{d}{\leq}&
			\mopC b_o^j   \lbr  \frac{(\sigma\ve \bsom)^{p}}{\theta_x^{sp}(c_o\varrho)^{sp}} +  \frac{\max\{\tau^{q-2},1\}\bsom^{q-2}(\sigma\ve \bsom)^2}{\theta_t (c_o \varrho)^{sp}}\rbr^{\frac{n}{p(n+2s)}} 
			\\&&\qquad\times \lbr \frac{1}{\min\{\tau^{q-2},1\}}\bsom^{2-q} \left[ \frac{(\sigma\ve \bsom)^{p}}{\theta_x^{sp}(c_o\varrho)^{sp}} +  \frac{\max\{\tau^{q-2},1\}\bsom^{q-2}(\sigma\ve \bsom)^2}{\theta_t (c_o \varrho)^{sp}} \right] \rbr^{\frac{s}{(n+2s)}}
			|A_j|^{1+\frac{s}{n+2s}} \\
			&\overred{6.41e}{e}{\leq} &
			\mopC
			\frac{b_o^j}{(c_o\varrho)^\frac{s(n+sp)}{n+2s}}
			(\sigma\ve\bsom)^{\lbr  {\frac{n(p+q-\mfd)}{p(n+2s)}} + {\frac{s(p+2-\mfd)}{(n+2s)}}\rbr}
			\lbr (\sigma\ve)^{\mfd-q} + (\sigma\ve)^{2-q} \rbr^{\frac{n}{p(n+2s)}} 
			\\&&\qquad\times \lbr \frac{1 + (\sigma \ve)^{\mfd -2}}{ \min \{ \tau^{q-2} , \tau^{(2-q)(1+\frac{n}{sp})} \} } \rbr^{\frac{s}{n+2s}}	|A_j|^{1+\frac{s}{n+2s}}, 
		\end{array}
	\end{equation*}
	where to obtain \redref{6.41a}{a}, we made use of the observations and enlarged the domain of integration with $\bar{\zeta}_j$; to obtain \redref{6.41b}{b}, we applied H\"older's inequality; to obtain \redref{6.41c}{c}, we applied \cref{fracpoin}; to obtain \redref{6.41d}{d}, we made use of \cref{Eq:5.20A}  with $\mopC = \mopC_{\data{}}$ and finally we collected all the terms to obtain \redref{6.41e}{e} and fact that $\min\{\tau^{q-2},1\} = \tau^{q-2}$ then $\max\{\tau^{q-2},1\}=1$ and visa versa.  Here $b_o = b_o(\datanb{})\geq 1$ is a constant.  We set
	$\bsy_j=|A_j|/|\mcQ_j|$ to get
	\begin{multline*}
	\frac{\sigma\ve \bsom}{2^{j+2}} \bsy_{j+1}  \leq \mopC  \bsy_j^{1+\frac{s}{n+2s}}\frac{b_o^j}{(c_o\varrho)^\frac{s(n+sp)}{n+2s}}
	(\sigma\ve\bsom)^{\lbr  {\frac{n(p+q-\mfd)}{p(n+2s)}} + {\frac{s(p+2-\mfd)}{(n+2s)}}\rbr}
	 \lbr (\sigma\ve)^{\mfd-q} + (\sigma\ve)^{2-q} \rbr^{\frac{n}{p(n+2s)}}\\ \times\lbr \tfrac{1 + (\sigma \ve)^{\mfd -2}}{ \min \{ \tau^{q-2} , \tau^{(2-q)(1+\frac{n}{sp})} \} } \rbr^{\frac{s}{n+2s}}
	|\mcQ_{j}|^{\frac{s}{n+2s}}.
	\end{multline*}
	Now using $|\mcQ_{j+1}| \approx |\mcQ_j| \approx  ( \bsom)^{\frac{n(\mfd-q)}{sp}} (\sigma\ve\bsom)^{\mfd - p } (c_o\varrho)^{n+sp}$,  we get
	\begin{multline*}
		\frac{\sigma\ve \bsom}{2^{j+2}}  \bsy_{j+1} \leq \mopC b_o^j \bsy_j^{1+\frac{s}{n+2s}}
		(\sigma\ve\bsom)^{\lbr  {\frac{n(p+q-\mfd)}{p(n+2s)}} + {\frac{s(p+2-\mfd)}{(n+2s)}} + \frac{s}{n+2s} \left[ \frac{n(\mfd-q)}{sp} + \mfd -p \right] \rbr} \\
		 \times \lbr (\sigma\ve)^{\mfd-q} + (\sigma\ve)^{2-q} \rbr^{\frac{n}{p(n+2s)}} \lbr 1 + (\sigma \ve)^{\mfd -2} \rbr^{\frac{s}{n+2s}} (\sigma \ve)^{\frac{n(q-\mfd)}{p(n+2s)}} \lbr  \min \{ \tau^{q-2} , \tau^{(2-q)(1+\frac{n}{sp})} \}  \rbr^{\frac{-s}{n+2s}} .
	\end{multline*}
	%$ \lbr \tfrac{1}{\de_x^{sp}} + \tfrac{1}{\de_t}\rbr^{\frac{n+sp}{p(n+2s)}} (\de_x^n\de_t)^{\frac{s}{(n+2s)}}$ 
%	
	Simplifying this expression and noting that $$\lbr  {\frac{n(p+q-\mfd)}{p(n+2s)}} + {\frac{s(p+2-\mfd)}{(n+2s)}} + \frac{s}{n+2s} \left[ \frac{n(\mfd-q)}{sp} + \mfd -p \right] \rbr = 1,$$  we get
	\begin{equation*}
		\bsy_{j+1}
		\le
		\bsc \Gamma  \boldsymbol b^j  \bsy_j^{1+\frac{s}{n+2s}},
	\end{equation*}
	where $\bsc  = \bsc_{\data{}} \geq 1$,  $\boldsymbol b\geq  1$ are two constants depending only on the data and $\Gamma$ is defined as
	\begin{equation*}
		\begin{array}{rcl}
		\Gamma &=& \lbr (\sigma\ve)^{\mfd-q} + (\sigma\ve)^{2-q} \rbr^{\frac{n}{p(n+2s)}} \lbr 1 + (\sigma \ve)^{\mfd -2} \rbr^{\frac{s}{n+2s}} (\sigma \ve)^{\frac{n(q-\mfd)}{p(n+2s)}} \lbr  \min \{ \tau^{q-2} , \tau^{(2-q)(1+\frac{n}{sp})} \}  \rbr^{\frac{-s}{n+2s}}  \\
		&=& (\sigma \ve)^{\frac{s}{n+2s} (\mfd -2)} (1+(\sigma\ve)^{2-\mfd})^{\frac{n+sp}{p(n+2s)}} \lbr  \min \{ \tau^{q-2} , \tau^{(2-q)(1+\frac{n}{sp})} \}  \rbr^{\frac{-s}{n+2s}}.
		\end{array} 
	\end{equation*}

	We can now apply  \cref{geo_con} to see that if $\bsy_0 \leq \bsc^{-1/\alpha}\Gamma^{-1/\alpha} \boldsymbol b^{1/\alpha^2}$ holds with $\alpha = \tfrac{s}{n+2s}$, then $\bsy_{\infty} = 0$ which is the desired conclusion.  Thus we define
	$$ 
	\bsc^{-1/\alpha} \boldsymbol b^{1/\alpha^2}  \Gamma^{-1/\alpha}= \nu_1  \min \{ \tau^{q-2} , \tau^{(2-q)(1+\frac{n}{sp})} \}  \frac{(\sigma \ve)^{2-\mfd}}{(1+(\sigma \ve)^{2-\mfd})^{\frac{n+sp}{sp}}} ,
	$$
	which completes the proof.
\end{proof}	
%\hrule\hrule\hrule\hrule\hrule\hrule

Now we are ready to show that oscillation is reduced.
\begin{proposition}\label{prop6.14}
	Let $u$ be bounded, weak solution of \cref{maineq} and let $1<\{p,q,2\}<\mfd$. Suppose for two  constants $\alpha \in (0,1)$ and $\tau \in (0,\tfrac14)$, the following assumption holds:
%	
%	
%	
%	\begin{equation*}%\label{Eq:6.41}
%		
%	\end{equation*}
%	and
	\[
	\tau \bsom \leq \pm \bsmu ^\pm \leq 2 \bsom \qquad \text{and} \qquad |\{ \pm (\bsmu^{\pm} -u(\cdot,t^{\ast}))\geq \tfrac{ \bsom}{4}\}\cap B_{ (\tfrac{ \bsom}{4})^{\frac{\mfd-q}{sp}} (c_o\varrho)}| \geq \alpha |B_{ (\tfrac{ \bsom}{4})^{\frac{\mfd-q}{sp}} (c_o\varrho)}|.
	\]
	Then there exists constants $\delta_1 = \delta_1(\datanb{,\alpha, \tau})$, $\eta_1 = \eta_1(\datanb{,\alpha,\tau})$ and $c_o = c_o(\eta_1)$ such that
	\[
	c_o^{\frac{sp}{p-1}} \leq \eta_1 , 
	\]
	and the following conclusion holds:
	\[
	\pm (\bsmu^{\pm} - u) \geq \eta_1 \bsom \quad \text{on} \quad B_{(\frac{\bsom}{4})^{\frac{\mfd-q }{sp}} \tfrac{c_o \varrho}{2}} \times (t^{\ast} + \tfrac12  \de_1 (\tfrac{\bsom}{4})^{\mfd-p}(c_o\varrho)^{sp} , t^{\ast}+ \de_1 (\tfrac{\bsom}{4})^{\mfd-p}(c_o\varrho)^{sp}],
	\]
	provided $B_{(\frac{\bsom}{4})^{\frac{\mfd-q }{sp}} (c_o \varrho)} \times (t^{\ast}  ,  t^{\ast}+ \de_1 (\tfrac{\bsom}{4})^{\mfd-p}(c_o\varrho)^{sp}] \subset \mcQ_1 .$
\end{proposition}
\begin{proof}
	We prove the case for $\bsmu^{+}$ only because the case for $\bsmu^-$ is analogous. 
	The proof follows in several steps:
	\begin{description}[leftmargin=*]
		\descitemnormal{Step 1:}{6step1singk} We apply \cref{Lm:6:2}  to get 
		\[
		\left|\left\{ u(\cdot, t)\le\bsmu^+-\varepsilon\boldsymbol \om\right\}\cap B_{\lbr \tfrac{\bsom}{4} \rbr^{\frac{\mfd - q}{sp}}c_o\varrho}\right|
		\ge
		\tfrac12\alpha |B_{\lbr \tfrac{\bsom}{4} \rbr^{\frac{\mfd - q}{sp}} c_o\varrho}|
		\txt{for all} t\in (t^{\ast} , t^{\ast} +   \lbr \ve \bsom\rbr^{\mfd-p}(c_o\varrho)^{sp}],
		\]
		where $\varepsilon = \varepsilon(\datanb{,\alpha,  \tau })\in(0,\tfrac{\tau}{2})$ and we have imposed the condition  that  $c_o^{sp} \leq 4^{\mfd-q} \ve^{p-1}$.	 
%						\hrule
		\descitemnormal{Step 2:}{6step2singk} In this step, we want to apply \cref{Lm:6.12}.  
		Since $\sigma \in (0,1)$ and $\mfd>p$, we see that $( \ve \bsom )^{\mfd-p}(c_o\varrho)^{sp} \geq (\sigma \ve \bsom)^{\mfd-p}(c_o\varrho)^{sp}$. Hence for all  $\hat{t} \in (t^{\ast}+ (\sigma\ve\bsom)^{\mfd-p}(c_o\varrho)^{sp} ,  t^{\ast} +(\varepsilon\bsom)^{\mfd-p}(c_o\varrho)^{sp} ]$ we have the measure information over the cylinder 
		\[
		\mcQ = B_{ \lbr \frac{\bsom}{4} \rbr^{\frac{\mfd - q}{sp}} c_o \varrho} \times ( \hat t -  \lbr \sigma \ve \bsom\rbr)^{\mfd-p} (c_o\varrho)^{sp} , \hat t] .
		\]
		Thus, applying \cref{Lm:6.12}, we get
		\begin{equation*}
			|\{ \bsmu^{+}-u \leq \tfrac12 \sigma \ve \bsom\} \cap \mcQ| \leq \mopC \lbr (\sigma \ve)^{2-\mfd} +1\rbr  \max\{\tau^{q-2} ,1 \} \tfrac{\sigma^{p-1}}{\alpha (1-\sigma)^{p-1}} |\mcQ|.
		\end{equation*}
		%where we have denoted $\mcQ = B_{\theta_x c_o \varrho} \times (t_o, t_o+ \de_t (\sigma\ve\bsom)^{\mfd-p}(c_o\varrho)^{sp})$.
		Hence if we choose $c_o^{sp} \leq 4^{\mfd-q} (\sigma \ve)^{p-1}$ ,  then the tail alternative is satisfied.
%		\hrule
		\item[Step 3:] Now We make a choice of $\sigma$.  We need to choose $\sigma$ such that following inequality holds.
		\begin{equation}\label{Eq:6.09}
			\mopC  \lbr (\sigma \ve)^{2-\mfd} +1\rbr  \max\{\tau^{q-2} ,1 \}  \tfrac{\sigma^{p-1}}{\al (1-\sigma)^{p-1}} \leq \nu_1 \min \{ \tau^{q-2} , \tau^{2-q(1+\frac{n}{sp})} \} \tfrac{(\sigma \ve)^{2-\mfd}}{(1+(\sigma \ve)^{2-\mfd})^{\frac{n+sp}{sp}}} ,
		\end{equation}
		
		where	$\mopC = \mopC(\datanb{}) >0$ and $\nu_1 = \nu_1(\datanb{}) \in (0,1)$ depends only on data.  
		Recalling $\varepsilon = \varepsilon(\datanb{,\alpha,  \tau })\in(0,\tfrac{\tau}{2})$ is a fixed constant,  the left hand side has the form
		\[
		\mopC  \lbr (\sigma \ve)^{2-\mfd} +1\rbr   \max\{\tau^{q-2} ,1 \} 
		\tfrac{\sigma^{p-1}}{\al (1-\sigma)^{p-1}} \approx \mopC (\sigma^{2-\mfd} +1) \sigma^{p-1} = \mopC (\sigma^{p+1-\mfd} + \sigma^{p-1}).
		\]
		The right hand side has the form 
		\[
		%	\begin{array}{rcl}
		\nu_1 \min \{ \tau^{q-2} , \tau^{2-q(1+\frac{n}{sp})} \} \tfrac{(\sigma \ve)^{2-\mfd}}{(1+(\sigma \ve)^{2-\mfd})^{\frac{n+sp}{sp}}}   \approx  \mopC \frac{\sigma^{(\mfd-2)\frac{n}{sp}}}{ \lbr \sigma^{\mfd-2} + 1 \rbr^{1+\frac{n}{sp}}} \\
		\overset{\sigma\in(0,\frac12)}{\geq}   \mopC \sigma^{(\mfd-2)\frac{n}{sp}}.
		%		& = & \mopC \lbr \frac{1}{\sigma^{p-2}+1}\rbr\lbr  \frac{1}{\lbr 1 + \sigma^{2-p} \rbr^{\frac{n}{sp}}}\rbr\\
		%		& \leq BAD & \mopC. 
		%	\end{array}
		\]
		Based on the above calculations, we need to ensure the following is satisfied:
		\begin{enumerate}[(i)]
			\item $p+1-\mfd   > 0$ which is equivalent to $\mfd < p+1$.
			\item $p+1-\mfd  >  (\mfd-2)\tfrac{n}{sp}$ which is equivalent to $\mfd < 2 + \frac{p-1}{1+\frac{n}{sp}}$.
			\item $p-1 >   (\mfd-2)\tfrac{n}{sp}$ which is equivalent to $\mfd < 2 + \frac{p-1}{\frac{n}{sp}}$.
		\end{enumerate}
	Thus we need to ensure 
	\begin{equation*}%\label{conditionond}
		\max\{p,q,2\} < \mfd < 2 + \frac{p-1}{1+\frac{n}{sp}}.
	\end{equation*}
%		This can be ensured if we select $\mfd $. In particular, if we fix some $1<p_o<2$, then $p_o-1>0$ and we can choose $\epsilon = \epsilon_{p_o} =  \tfrac{sp_o  (p_o -1)}{n+3s}$ such that we have inequality \cref{Eq:6.09} in the range $p \in (p_o,2+\ve_{p_o})$ and $q \in (2,2+\ve_{p_o})$  for $\sigma = \sigma(\datanb{,\alpha, \tau})$ small enough. 
%		%This can be ensured if we select $\epsilon = \min\left\{p-1, \tfrac{sp(p-1)}{n}, \tfrac{sp(p-1)}{n+sp}\right\}$. In particular, if we fix some $1<p_o<2$, then $p_o-1>0$ and we can choose $\epsilon = \epsilon_{p_o} = \min\left\{p_o-1, \tfrac{sp_o(p_o-1)}{n}, \tfrac{sp_o(p_o-1)}{n+3s}\right\}$ such that we have reduction of oscillation in the range $(p_o,2+\ve_{p_o})$. 
		\item[Step 4:] Now we can apply \cref{Lm:6.10} on the cylinder $
		\mcQ = B_{ \lbr \frac{\bsom}{4} \rbr^{\frac{\mfd - q}{sp}} c_o \varrho} \times (\hat t - \lbr \sigma \ve \bsom\rbr)^{\mfd-p} (c_o\varrho)^{sp}, \hat t] $ to get
		\begin{equation*}
			\bsmu^{+}-u \geq\tfrac{1}4 \sigma \varepsilon \bsom
			\quad
			\mbox{ on }\quad \mcQ = B_{ \lbr \frac{\bsom}{4} \rbr^{\frac{\mfd - q}{sp}} \tfrac{c_o \varrho}{2}} \times (\hat t -  \lbr \sigma \ve \bsom\rbr)^{\mfd-p} (\tfrac{c_o\varrho}{2})^{sp},  \hat t].
		\end{equation*}
		%Again we can choose $c_o^{sp} \leq \sigma \de\ve$ small enough such that the tail alternative does not hold. 
		\item[Step 5:] From step 4, we have for all $\hat{t} \in (t^{\ast}+ (\sigma\ve\bsom)^{\mfd-p}(c_o\varrho)^{sp} ,  t^{\ast} +(\varepsilon\bsom)^{\mfd-p}(c_o\varrho)^{sp} ]$,
		\begin{equation*}
			\pm\left(\bsmu^{\pm}-u(\cdot,  \hat{t})\right)\geq\tfrac{1}4\sigma \de\varepsilon \bsom
			\quad
			\mbox{ on }\quad  B_{ \lbr \frac{\bsom}{4} \rbr^{\frac{\mfd - q}{sp}} \tfrac{c_o \varrho}{2}}.
		\end{equation*}
		Now we can further choose $\sigma$ small if needed such that 
		$$
		(\sigma \ve\bsom)^{\mfd-p}(c_o\varrho)^{sp} \leq \tfrac12 (\varepsilon\bsom)^{\mfd-p}(c_o\varrho)^{sp}.
		$$
		this can be ensure by taking $\sigma \leq (\tfrac{1}{2})^{\frac{1}{\mfd -p}}$ and we have that 
		$$
		(t^{\ast} + \tfrac12 (\varepsilon\bsom)^{\mfd-p}(c_o\varrho)^{sp} , t^\ast +(\varepsilon\bsom)^{\mfd-p}(c_o\varrho)^{sp}] \subset (t^\ast+\de_t (\sigma\de\ve\bsom)^{\mfd-p}(c_o\varrho)^{sp} ,  t^\ast +(\varepsilon\bsom)^{\mfd-p}(c_o\varrho)^{sp} ].
		$$
	\end{description}
	We take $\eta_1 = \tfrac{1}4 \sigma \varepsilon$ and $\delta_1 =( 4 \ve )^{\mfd-p} $ to complete the proof of the lemma. 
\end{proof}

%\hrule
\subsection{Reduction of Oscillation near zero}

Now we are ready to show the first step of reduction of oscillation.
For some positive constant $\tau \in (0,\tfrac14)$ that will be made precise later on in the proof of \cref{reductionoscsingunifiedk}, we have two cases
\refstepcounter{equation}
\def\myname{\theequation}
\begin{align}[left=\empheqlbrace]
	&\mbox{when $u$ is \emph{near} zero:}  \bsmu^-\le\tau\bsom  \,\,\text{and} \,\,
	\bsmu^+\ge-\tau\bsom,\tag*{(\myname$_{a}$)}\label{Eq:6Hp-main1}\\%[5pt]
	&\mbox{when $u$ is \emph{away} from zero:} \bsmu^- >\tau\bsom \,\,\text{or}\,\, \bsmu^+<-\tau\bsom.\tag*{(\myname$_{b}$)}\label{Eq:6Hp-main2}
\end{align}
Furthermore, we will always assume the following is satisfied $\bsmu^+ -\bsmu^- >\tfrac12\bsom$, 
since in the other case, we trivially get the reduction of oscillation.
We see that  \tlcref{Eq:6Hp-main1} gives
\begin{description}
	\item[Lower bound for $\bsmu^-$:] In this case, we have
	\begin{equation*}%\label{deg6.3a}
		\bsmu^- = \bsmu^- - \bsmu^+ + \bsmu^+ \geq -\bsom + \bsmu^+ \geq -(1+\tau)\bsom.
	\end{equation*}
	\item[Upper bound for $\bsmu^+$:] In this case, we have
	\begin{equation*}%\label{deg6.3b}
		\bsmu^+ = \bsmu^+ - \bsmu^- + \bsmu^- \leq \bsom + \bsmu^- \leq (1+\tau)\bsom.
	\end{equation*}
\end{description}
In particular, \tlcref{Eq:6Hp-main1} implies
\begin{equation*}%\label{6boundmu}
	|\bsmu^{\pm}| \leq (1+\tau)\bsom \leq 2\bsom.
\end{equation*}

%\hrule\hrule\hrule\hrule
\begin{lemma}
	\label{reductionoscsingunifiedk}
	Let $u$ be bounded, weak solution of \cref{maineq} and let $\mfd$ satisfy the condition from \cref{holderparabolic}.  There exists constants $\eta$,  $\de$, $\tau$ and $c_o $  depending only on data, such that if \tlcref{Eq:6Hp-main1} holds,  then one of the following two conclusion holds:
	\[
	\quad \bsmu^{+}-u \geq \eta \bsom
	\quad
	\mbox{ on }\quad  B_{ \lbr \frac{\bsom}{4} \rbr^{\frac{\mfd - q}{sp}} \tfrac{c_o \varrho}{2}} \times  (- \de (\tfrac{\bsom}{4})^{\mfd-p}(\tfrac{c_o\varrho}{2})^{sp},  0],
	\]
	or
	\[
	-(\bsmu^{-}-u) \geq \eta \bsom
	\quad
	\mbox{ on }\quad  B_{ \lbr \frac{\bsom}{4} \rbr^{\frac{\mfd - q}{sp}} \tfrac{c_o \varrho}{2}} \times  (- \de (\tfrac{\bsom}{4})^{\mfd-p}(\tfrac{c_o\varrho}{2})^{sp},  0].
	\]	
	Furthermore,	if we denote $\bsom_1:= (1-\eta)\bsom$, $\boldsymbol{C}_o = \min\{  \tfrac12 c_o (4(1-\eta))^{\frac{q-\mfd}{sp}},  \tfrac{1}{2} \de c_o (4(1-\eta))^{\frac{p-\mfd}{sp}}\}$ and $\varrho_1:= \boldsymbol{C}_o^{-1}\varrho$, then we have 
	\[
	\essosc_{\mcQ_{\varrho_1}(\bsom_1)} u \leq \bsom_1,
	\]
	where $\mcQ_{\varrho_1}(\bsom_1) = B_{\bsom_1^{\frac{\mfd-q}{sp}} \varrho_1} \times (- \bsom_1^{\mfd-p}\varrho_1^{sp},0]$.
	% $\mcQ_2 = B_{\bsom_1^{\frac{\mfd-2}{sp}} \varrho_1} \times (- \bsom_1^{\mfd-p}\varrho_1^{sp},0)$ and $\mcQ_1 = B_{\bsom^{\frac{\mfd-2}{sp}} \varrho} \times (- \bsom^{\mfd-p}\varrho^{sp},0]$. 
\end{lemma}
\begin{proof}
	We now obtain the reduction of oscillation as follows:
	\begin{description}[leftmargin=*]
		\item[Step 1:] Let us fix a time level $t^{\ast}$ chosen later in $\mcQ_1 = \mcQ_{\varrho} (\bsom) = B_{ \bsom^{\frac{\mfd-q}{sp}} (\varrho)} \times (-  \bsom^{\mfd-p} (\varrho)^{sp},0]. $
		Noting that $\bsmu^+ -\bsmu^- >\tfrac12\bsom$, we see that one of the two alternatives now must hold: 
		\refstepcounter{equation}
		\def\mynamen{\theequation}
		\begin{align}[left=\empheqlbrace]
			&|\{+ (\bsmu^{+} -u(\cdot,t^{\ast})\geq \tfrac14 \bsom\}\cap B_{ (\frac14\bsom)^{\frac{\mfd-q}{sp}} (c_o\varrho)}| \geq \tfrac12 |B_{ (\frac14\bsom)^{\frac{\mfd-q}{sp}} (c_o\varrho)}|,\tag*{(\mynamen$_{a}$)}\label{eq6.13.1}\\%[5pt]
			&|\{- (\bsmu^{-} -u(\cdot,t^{\ast}))\geq \tfrac14 \bsom\}\cap B_{ (\frac14\bsom)^{\frac{\mfd-q}{sp}} (c_o\varrho)}| \geq \tfrac12 |B_{ (\frac14\bsom)^{\frac{\mfd-q}{sp}} (c_o\varrho)}|.\tag*{(\mynamen$_{b}$)}\label{eq6.13.2}
		\end{align}
	
%		\begin{equation}\label{eq6.13}
%			\left\{\begin{array}{l}
%				|\{+ (\bsmu^{+} -u(\cdot,t^{\ast})\geq \tfrac14 \bsom\}\cap B_{ (\frac14\bsom)^{\frac{\mfd-q}{sp}} (c_o\varrho)}| \geq \tfrac12 |B_{ (\frac14\bsom)^{\frac{\mfd-q}{sp}} (c_o\varrho)}|\\
%				|\{- (\bsmu^{-} -u(\cdot,t^{\ast}))\geq \tfrac14 \bsom\}\cap B_{ (\frac14\bsom)^{\frac{\mfd-q}{sp}} (c_o\varrho)}| \geq \tfrac12 |B_{ (\frac14\bsom)^{\frac{\mfd-q}{sp}} (c_o\varrho)}|
%			\end{array}\right.
%		\end{equation}
		\item[Step 2:] Without loss of generality, let us assume  \tlcref{eq6.13.1} holds, noting that the other case follows analogously.  Now we apply \cref{prop6.13} and \cref{prop6.14} with $\alpha = \tfrac12$ to get the constants $\eta_o$ and $\de_o$  depending only on data. 
%		\hrule
		\item[Step 3:] We fix $\tau = \eta_o$ then we have two cases,
		\begin{description}
			\item[Case $-\tau \bsom \leq \bsmu^+ \leq \tau \bsom$:] In this case we apply \cref{prop6.13} to get 
			\begin{equation}\label{eq6.17}
				\bsmu^+ - u \geq \eta_o \bsom \quad \text{on} \quad B_{(\frac{\bsom}{4})^{\frac{\mfd-q }{sp}} \tfrac{c_o \varrho}{2}} \times (t^{\ast} ,t^{\ast}+ \de_o (\tfrac{\bsom}{4})^{\mfd-p}(\tfrac{c_o\varrho}{2})^{sp}].
			\end{equation}
			\item[Case $\tau \bsom \leq \bsmu^+ \leq 2 \bsom$:] In this case we apply \cref{prop6.14} to get 
			\begin{equation}\label{eq6.18}
				\bsmu^{+}-u \geq \eta_1 \bsom
				\quad
				\mbox{ on }\quad  B_{ \lbr \frac{\bsom}{4} \rbr^{\frac{\mfd - q}{sp}} \tfrac{c_o \varrho}{2}} \times  (t^{\ast} ,t^{\ast}+ \de_1 (\tfrac{\bsom}{4})^{\mfd-p}(\tfrac{c_o\varrho}{2})^{sp}].
			\end{equation}
		\end{description}
%		Case 1: When $-\tau \bsom \leq \bsmu^+ \leq \tau \bsom$, 
%		
%		Case 2: When $\tau \bsom \leq \bsmu^+ \leq 2 \bsom$, 
		The choice of $c_o$ is chosen to ensure the tail alternative is satisfied in both cases which can be done by taking $c_o < \min \{ \eta_o^{\frac{p-1}{sp}} ,  \eta_1^{\frac{p-1}{sp}} \}$.  
		\item[Step 4:] We now define $\de = \min \{ \de_o, \de_1,\tfrac12 \}$ and $\eta= \min  \{ \eta_o, \eta_1 \}$.  If we take $t^\ast = -\de (\tfrac{\bsom}{4})^{\mfd-p}(c_o\varrho)^{sp}$, then using \eqref{eq6.17} and \eqref{eq6.18}, the  following conclusion follows:
		%apply \cref{prop3.13} to find corresponding constants such that the following conclusion follows:
		\[
		\bsmu^{+}-u \geq \eta \bsom
		\quad
		\mbox{ on }\quad  B_{ \lbr \frac{\bsom}{4} \rbr^{\frac{\mfd - q}{sp}} \tfrac{c_o \varrho}{2}} \times  (- \de (\tfrac{\bsom}{4})^{\mfd-p}(\tfrac{c_o\varrho}{2})^{sp},  0].
		\]
%		\hrule
		%The choice of $c_o$ is chosen to ensure the tail alternative is satisfied. 
		\item[Step 5:] Let us take $\bsom_1 := (1-\eta)\bsom$ and $\varrho_1 = \boldsymbol{C}_o^{-1}\varrho$ with 
		\[
		\boldsymbol{C}_o^{-1} = \min\{  \tfrac12 c_o (4(1-\eta))^{\frac{q-\mfd}{sp}},  \tfrac{1}{2} \de c_o (4(1-\eta))^{\frac{p-\mfd}{sp}}\},
		\]
		and define cylinder $\mcQ_{\varrho_1}(\bsom_1) = B_{\bsom_1^{\frac{\mfd-q}{sp}} \varrho_1} \times (- \bsom_1^{\mfd-p}\varrho_1^{sp},0]$ and $\mcQ_{\varrho}(\bsom) = B_{\bsom^{\frac{\mfd-q}{sp}} \varrho} \times (- \bsom^{\mfd-p}\varrho^{sp},0]$. 
		Then we will have 
		$
		\mcQ_{\varrho_1}(\bsom_1) \subset \mcQ_{\varrho} (\bsom)
		$
		and
		$
		\essosc_{\mcQ_{\varrho_1}(\bsom_1)} u \leq \bsom_1.
		$
	\end{description}
	This completes the proof of the reduction of oscillation. 
\end{proof}

%\hrule
%\section{Reduction of oscillation for a scaled parabolic fractional \texorpdfstring{$p$}.-Laplace type equations - Preliminary lemmas}\label{section7}
\section{Reduction of Oscillation away from  zero}\label{section7}
%\hrule\hrule\hrule\hrule
We are in the case \tlcref{Eq:6Hp-main2} holds and without loss of generality, we consider the case $\bsmu^- > \tau \bsom$ noting that the other case follows by replacing $u$ with $-u$. The calculations in this section follow verbatim as in \cite[Sections 7,8,9]{adimurthi2025localholderregularitybounded}, because once the change of variable is achieved, then the equation structure is similar to \cref{maineq} with $q=2$. Thus we only provide a very rough sketch of the reduction of oscillation and refer to \cite{adimurthi2025localholderregularitybounded} for all the details. 

\begin{assumption}\label{assump3}
	Without loss of generality, we consider the case $\bsmu^- > \tau \bsom$. Let us further assume that the following bound holds: 
	\[
	\bsmu^- \leq \frac{1+\tau}{1-\mathring{\eta}} \bsom,
	\]
	for some $\mathring{\eta} = \mathring{\eta}(\datanb{}) \in (0,1)$ satisfying $\frac{1+\tau}{1-\mathring{\eta}} \geq 1$.  See \cref{rmkdeg7.15} on how this assumption is always satisfied.
\end{assumption}

\begin{remark}\label{rmkdeg7.15}
	In order to obtain H\"older regularity, we iteratively assume that \tlcref{Eq:6Hp-main1} holds and denote $i_o$ to be the first time when \tlcref{Eq:6Hp-main2} occurs. If $i_o=0$, then we trivially have $\bsmu^- \leq \bsom$ from the choice of \cref{defbsom}. On the other hand, if $i_o \geq 1$, then \tlcref{Eq:6Hp-main2} implies the following holds:
	\[
	\text{either} \quad\bsmu_{i_o}^-> \tau\bsom_{i_o}\;\quad
	\text{or}\quad
	\;\bsmu_{i_o}^+<-\tau\bsom_{i_o}.
	\]
	We work with $\bsmu_{i_o}^->\tau\bsom_{i_o}$ with the other case being analogous.
	Since $i_o$ is the first index for this to happen,
	we have $\bsmu_{i_o-1}^-\le \tau\bsom_{i_o-1}$ (here is where we need to assume $i_o \geq 1$) and
	\[
	\bsmu_{i_o}^-
	\leq \bsmu_{i_o}^+
	\leq
	\bsmu_{i_o-1}^{+}
	=
	\bsmu_{i_o-1}^{-} + \bsmu_{i_o-1}^{+} - \bsmu_{i_o-1}^{-}
	\le
	\bsmu_{i_o-1}^{-} + \bsom_{i_o-1}
	\le
	(1+\tau)\bsom_{i_o-1} = \frac{1+\tau}{1-\mreta}\bsom_{i_o},
	\]
	where we assumed $\bsom_{i_o} = (1-\mreta) \bsom_{i_o-1} =(1-\mreta)^2 \bsom_{i_o-2} \ldots = (1-\mreta)^{i_o} \bsom $. 
	As a result, we have
	\begin{equation}\label{Eq:7:4}
		\tau\bsom_{i_o}
		<
		\bsmu_{i_o}^-\le\frac{1+\tau}{1-\mreta}\bsom_{i_o}.
	\end{equation}
	In particular, \cref{assump3} is always satisfied.  Also we will be dealing with the cylinder of the form $\mcQ_{i_o} = B_{\de_x (\bsom_{i_o})^{\frac{\mfd - q}{sp}}  R_{i_o}} \times ( -\de_t (\bsom_{i_o})^{\mfd-p} R_{i_o}^{sp} ,0]$.
\end{remark}

\begin{remark}
	By an abuse of notation, we shall suppress keeping track of $i_o$ in the next section and instead denote $R_{i_o}, $ $\bsmu_{i_o}, $ and $\bsom_{i_o}$ by $R, $ $\bsmu$ and $\bsom$ respectively.   
\end{remark}
%\subsection{Proof of oscillation decay in the degenerate case}

\begin{definition}\label{defvw}
	Let $\mcQ_{R} = B_{\de_x  R} \times ( -\de_t R^{sp} ,0]$ and  ${\mcQ}_o={\mcQ}_{\varrho} = B_{\varrho}(x_o) \times (t_o-\bar A\varrho^{sp},t_o] \subset \mcQ_R $ where $\bar A \geq 1$ will be chosen later in terms of $\data$. Let $\hat{u} (x,t) = u((\bsom)^{\frac{\mfd-q}{sp}} x ,  (\bsom)^{\mfd -p} t )$ and $w:=\tfrac{\hat u}{\bsmu^-}$,  then making use of \cref{assump3}, we see that $v$ satisfies
	\begin{equation}\label{Eq:7:9}
		1 \leq w \leq \frac{1+\tau}{\tau} \qquad \text{ a.e in }\quad {\mcQ_R},
	\end{equation}
	where we made use of the following bound:
	\[
	\bsmu^- \leq u \leq \bsmu^+ - \bsmu^- + \bsmu^- \leq \boldsymbol \om + \bsmu^- \leq (1+\tfrac{1}{\tau})\bsmu^-. 
	\]
	Moreover,  from $u$ solving \cref{maineq},  we can easily deduced that  $w$ also solves %\cref{maineq}. 
	\begin{equation*}%\label{7maineq}
	\partial_t (w^{q-1}) + \text{P.V.}\int_{\RR^n} |w(x,t)-w(y,t)|^{p-2}(w(x,t)-w(y,t)) \hat{K}(x,y,t)\,dy=0,  \quad \text{in} \; \mcQ_R,
\end{equation*} 
where $\hat{K}:\RR^n\times\RR^n\times \RR \to [0,\infty)$ is a symmetric measurable function satisfying
\begin{equation*}%\label{7boundsonKernel}
\left( \frac{\bsmu^-}{\bsom} \right)^{p-q}	\frac{(1-s)}{\Lambda|x-y|^{n+sp}}\leq \hat{K}(x,y,t)\leq \left( \frac{\bsmu^-}{\bsom} \right)^{p-q} \frac{(1-s)\Lambda}{|x-y|^{n+sp}}. 
\end{equation*}
\end{definition}

\begin{definition}\label{vwdef7.2}
 Let us define $v=w^{q-1}$ in $\mcQ_o$.	Let us denote $\bsmu^-_w  =  \frac{\essinf_{\mcQ_o} \hat u}{\bsmu^-}=	\essinf_{\mcQ_o} w$ and  $\bsmu^+_w   = \frac{\esssup_{\mcQ_o} \hat u}{\bsmu^-}=	\esssup_{\mcQ_o} w$. Analogously, denote $\bsmu^-_v = (\bsmu^-_w)^{q-1}=	(\essinf_{\mcQ_o} w)^{q-1} =\essinf_{\mcQ_o} v$ and  $\bsmu^+_v = (\bsmu^+_w)^{q-1}= 	(\esssup_{\mcQ_o} w)^{q-1} =\esssup_{\mcQ_o} v$.
	We shall denote $\bsom_w = \frac{\bsom}{\bsmu^-}$ using $\bsom \geq \essosc_{\mcQ_o} \hat u$  it is easy to see that 
	\[
	\bsom \geq \bsmu^+ - \bsmu^- \qquad \Longrightarrow \qquad \bsom_w \geq \bsmu^+_w - \bsmu^-_w.
	\]
	
	Since $w$ satisfies \cref{Eq:7:9}, then for any $x,y \in \mcQ_R$ with $v(x)\geq v(y)$, we have
	\[
	v(x) - v(y) = w^{q-1}(x) - w^{q-1}(y) \overlabel{alg_lem}{\leq} (q-1)\lbr \tfrac{1+\tau}{\tau}\rbr^{q-1} (w(x)- w(y)).
	\]
	In particular, this implies that if we take $\bsom_v := (q-1)\lbr \tfrac{1+\tau}{\tau}\rbr^{qq-1}  \bsom_w$, then we have
	\[
	\bsom_v \geq 	\bsmu^+_v - \bsmu^-_v.
	\]
%	\hrule
	Moreover, making use of \cref{Eq:7:4} and $\bsom_w = \frac{\bsom}{\bsmu^-}$, we also have
	\begin{equation*}%\label{boundsonomega}
		\mathfrak{C}_1:= (q-1)\frac{1-\mathring{\eta}}{1+\tau}\left(\frac{1+\tau}{\tau}\right)^{q-1}\leq {\bsom}_v\leq \frac{q-1}{\tau}\left(\frac{1+\tau}{\tau}\right)^{q-1}=:\mathfrak{C}_2.
	\end{equation*}
%\hrule
\end{definition}
%\subsection{Energy estimates }

Let us  first state standard energy estimate for $v$,   the proof of which follows by imitating the proof of \cite[Lemma 7.6]{adimurthi2025localholderregularitybounded} and fact that ratio $\tfrac{\bsom}{\bsmu^-}$ is bounded above and below by a constant depending only on data (see \eqref{Eq:7:4}).
%\hrule\hrule\hrule\hrule
\begin{lemma}\label{lemma7.3}
	With $w$  and  $v$ be as in \cref{defvw} and \cref{vwdef7.2} and $k \in \RR$ be given, then
	there exists a constant $\bsc  >0$ depending only on the data such that
	for all  every non-negative, piecewise smooth cut-off function
	$\zeta(x,t) = \zeta_1(x)\zeta_2(t)$ vanishing on $\partial_p \mcQ_o$,  there holds
	\begin{multline*}
		\esssup_{t_o-\bar A \varrho^{sp}<t<t_o}\int_{B_{\varrho}\times\{t\}}	
		\zeta^p(v-k)_{\pm}^2\,dx +
		\iint_{\mcQ_o} (v-k)_{\pm}(x,t)\zeta_1^p(x)\zeta_2^p(t)\lbr \int_{B_{\varrho}}\frac{(v-k)_{\mp}^{p-1}(y,t)}{|x-y|^{n+sp}}\,dy\rbr\,dx\,dt 
		\\
		+\iiint_{\mcQ_o} \frac{|(v-k)_{\pm}(x,t)-(v-k)_{\pm}(y,t)|^p}{|x-y|^{n+ps}}\max\{\zeta_1(x),\zeta_1(y)\}^p\zeta_2^p(t)\,dx\,dy\,dt \\ 
		\begin{array}{rcl}
			&\leq&
			\boldsymbol{C} \iiint_{\mcQ_o} \frac{\max\{(v-k)_{\pm}(x,t),(v-k)_{\pm}(y,t)\}^{p}|\zeta(x,t)-\zeta(y,t)|^p}{|x-y|^{n+sp}}\,dx\,dy\,dt
			\\
			&&+
			\iint_{\mcq}\mathfrak (v-k)_\pm^2 (u,k)|\partial_t\zeta^p| \,dx\,dt
			+\int_{B_{\varrho}\times \{t=t_o-\bar A\varrho^{sp}\}} \zeta^p (v-k)_\pm \,dx 
			\\
			&&+ \,\boldsymbol{C}\iint_{\mcQ_o}\zeta_2^p(t)\zeta_1^p(x)(v-k)_{\pm}(x,t)\lbr \esssup\limits_{\stackrel{t\in (t_o - \bar A \varrho^{sp},t_o)}{x \in \spt \zeta_1}}\left[\int_{{\RR^n\setminus B_{\varrho}}} \frac{(w- k^{\frac{1}{q-1}})_{\pm}(y,t)^{p-1}}{|x-y|^{n+sp}}\,dy\right]\rbr\,dx\,dt.
		\end{array}
	\end{multline*}
\end{lemma}

Now  we can obtain reduction of oscillation considering the degenerate and singular cases separately as done in \cite{adimurthi2025localholderregularitybounded}. The first is the reduction of oscillation in the degenerate case $p>2$ which follows verbatim as in \cite[Proposition 8.5]{adimurthi2025localholderregularitybounded}:

%We will eventually choose $\bar{A} = \lbr \frac{\bsom_v}{\mathbf{a}}\rbr^{2-p}$, for some $\mathbf{a} \gg 1
%	$ to be chosen depending only on data satisfying \cref{defdega1} and \cref{defdega2} (see \cref{degconstantsscaled}). In what follows, we are interested in obtaining reduction of oscillation in the cylinder 
%	\begin{equation*}
%		\mcq_{c_o\varrho}^{\bar{A}} := B_{c_o\varrho} \times (-\bar{A}(c_o\varrho)^{sp},0).
%	\end{equation*}

\begin{proposition}\label{propdegaway}
	Let $p>2$ and $\mcQ_{c_o\varrho}^{\bar{A}} := B_{c_o\varrho} \times (-\lbr\tfrac{\bsom_v}{\mathbf{a}} \rbr^{2-p}(c_o\varrho)^{sp},0)$ be a given cylinder  with $\mathbf{a} = \mathbf{a}(\datanb{})$,  then there exists a constant $\mathbf{d} = \mathbf{d}(\datanb{,\tau,\mreta}) $, $\eta_d= \eta_d(\datanb{,\tau,\mreta})$ and $c_o = c_o(\datanb{,\tau,\mreta,\tfrac{1}{\mathbf{a}}})$ such following conclusions hold:
%	\begin{description}
%		\descitemnormal{Conclusion 1:}{conc1deg} We have
%		\begin{equation*}
%			\essosc_{\mcq} v \leq \lbr 1-\tfrac12 \varepsilon_1\rbr \bsom_v \quad \text{where} \quad \mcq:= B_{\frac{c_o\varrho}4} \times \left(-\bar\nu_3(\varepsilon_1\bsom_v)^{2-p}\lbr{\tfrac14c_o\varrho}\rbr^{sp},0\right],
%		\end{equation*}
%	and 
%	\[
%	c_o^{\frac{sp}{p-1}}\tail((u-\bsmu^{-})_{-};0,c_o\varrho,(-\lbr\tfrac{\bsom_v}{\mathbf{a}} \rbr^{2-p}(c_o\varrho)^{sp},0)) \leq \min\{\tfrac14,\varepsilon_1,\bar\sigma_2\varepsilon_2\} \bsom.
%	\]
%	 \descitemnormal{Conclusion 2:}{conc2deg} Or
%	 \[
%	 v \leq \bsmu^+_v - \tfrac12 \bar{\sigma_2}\bar\varepsilon_1\varepsilon_4 \bsom_v\qquad \text{in} \quad B_{\frac14c_o\varrho}\times (-(\bar{\sigma_2}\bar\varepsilon_1\varepsilon_4 \bsom_v)^{p-2}(\tfrac14c_o\varrho)^{sp},0).
%	 \]
	% and 
	% \[
	% c_o^{\frac{sp}{p-1}}\tail((u-\bsmu^{+})_{+};0,c_o\varrho,(-\lbr\tfrac{\bsom_v}{\mathbf{a}} \rbr^{2-p}(c_o\varrho)^{sp},0)) \leq \min\{\tfrac14,\bar\varepsilon_1,\bar\varepsilon_1\varepsilon_4, \bar\sigma_2\bar\varepsilon_2\} \bsom.
	% \]
	%\end{description}
%We can combine both conclusions to have
\[
\essosc_{\mcQ}v \leq \lbr1-\eta_d \rbr \bsom_v\qquad \text{where} \quad \mcQ:= B_{\frac14c_o\varrho}\times (-\mathbf{d}\bsom_v^{2-p}(\tfrac14c_o\varrho)^{sp},0),
\]
%where $\mathbf{d} = \min\{\bar\nu_3(\varepsilon_1)^{2-p}, (\bar\sigma_2\bar\varepsilon_1\varepsilon_4)^{2-p}\}$, $\eta_d= \min\{\tfrac12 \bar{\sigma_2}\bar\varepsilon_1\varepsilon_4, \tfrac12 \varepsilon_1\}$
and 
\[
c_o^{\frac{sp}{p-1}}\tail((u-\bsmu^{\pm})_{\pm};0,c_o\varrho,(-\lbr\tfrac{\bsom_v}{\mathbf{a}} \rbr^{2-p}(c_o\varrho)^{sp},0)) \leq \eta_d \bsom,
\]
where 
%$\mathbf{d}_t:=\min\{\tfrac14,\bar\varepsilon_1,\bar\varepsilon_1\varepsilon_4, \bar\sigma_2\bar\varepsilon_2\}$ and 
$\bsom,\bsom_v$ are related by $\bsom_v = (q-1)\lbr \tfrac{1+\tau}{\tau}\rbr^{q-1}\tfrac{\bsom}{\bsmu^-}$. 
\end{proposition}

%In this case, we assume $p \leq 2$ and obtain reduction of oscillation. The proof follows the calculations of \cite[Section 4]{liaoHolderRegularityParabolic2022}. 

%\begin{remark}
%	We note that all the calculations from \cref{section6} holds with $\tau = 1$ in the case $p<2$ with obvious modifications. We will not write these standard modifications for the sake of brevity. 
%\end{remark}
In the singular case $p<2$, we obtain the reduction of oscillation verbatim as in \cite[Proposition 9.2]{adimurthi2025localholderregularitybounded}:
\begin{proposition}
	Suppose for two  constants $\alpha \in (0,1)$ and $\varepsilon \in (0,1)$, the following assumption holds:
	\[
	|\{\pm (\bsmu^{\pm}_v - v(\cdot,t_o))\geq \varepsilon \bsom_v\} \cap B_{c_o\varrho} | \geq \alpha |B_{c_o\varrho}|.
	\]
	Then there exists constants $\delta = \delta(\datanb{,\alpha})$, $\eta = \eta(\datanb{,\alpha})$ and $c_o = c_o(\eta,\varepsilon)$ such that the following conclusion follows:
	\[
	c_o^{\frac{sp}{p-1}} \tailp((u - \bsmu^{\pm})_{\pm}; \mcQ_o) \leq \eta \varepsilon \bsom,
	\]
	and 
	\begin{equation*}
	\pm(\bsmu^{\pm}_v - v) \geq \eta \varepsilon \bsom_v \quad \text{on} \quad B_{c_o\varrho} \times (t_o + \tfrac12 \delta (\varepsilon \bsom_v)^{2-p}(c_o\varrho)^{sp}, t_o + \delta  (\varepsilon \bsom_v)^{2-p}(c_o\varrho)^{sp}),
		\end{equation*}
		provided $B_{2c_o\varrho} \times (t_o ,  t_o + \delta (\varepsilon \bsom_v)^{2-p}(c_o\varrho)^{sp}) \subset \mcQ_o$. 
		Moreover, we have $\delta \approx \alpha^{p-1}$ and $\eta \approx \alpha^{k}$ for some $k>1$ depending only on data.
\end{proposition}

\section{Completing the proof of H\"{o}lder regularity}

The entire covering argument is by now a very standard argument, see for example  \cite[Section 10]{adimurthi2025localholderregularitybounded} for the details. This completes the proof of the H\"older regularity.

\section*{Data Availability Statement}

This paper contains no associated data. 

\section*{Conflict of Interest Statement}

The authors certify that there is no conflict of interest.

\section*{References} 
%\nocite{*}
%\bibliographystyle{plain}
%\begin{singlespace}
%	\bibliography{main}

\begin{thebibliography}{10}
	
	\bibitem{adimurthi2025localholderregularitybounded}
	Karthik Adimurthi.
	\newblock Local H\"older regularity for bounded, signed solutions to nonlocal
	trudinger equations.
	\newblock {\em arXiv:2503.07184}, 2025.
	
	\bibitem{adimurthiOlderRegularityFractional2022}
	Karthik Adimurthi, Harsh Prasad, and Vivek Tewary.
	\newblock H\"{o}lder regularity for fractional {$p$}-{L}aplace equations.
	\newblock {\em Proc. Indian Acad. Sci. Math. Sci.}, 133(1):Paper No. 14, 24,
	2023.
	
	\bibitem{adimurthi2024localholderregularitynonlocal}
	Karthik Adimurthi, Harsh Prasad, and Vivek Tewary.
	\newblock Local H\"older regularity for nonlocal parabolic $p$-laplace
	equations.
	\newblock {\em Annali della Scuola Normale Superiore di Pisa - Classe di
		Scienze}, 2024.
	
	\bibitem{bassHarnackInequalitiesNonlocal2005}
	Richard Bass and Moritz Kassmann.
	\newblock Harnack inequalities for non-local operators of variable order.
	\newblock {\em Transactions of the American Mathematical Society},
	357(2):837--850, 2005.
	
	\bibitem{bassHolderContinuityHarmonic2005}
	Richard~F. Bass and Moritz Kassmann.
	\newblock H\"older {{Continuity}} of {{Harmonic Functions}} with {{Respect}} to
	{{Operators}} of {{Variable Order}}.
	\newblock {\em Communications in Partial Differential Equations},
	30(8):1249--1259, July 2005.
	
	\bibitem{biswas2025lipschitzregularityfractionalplaplacian}
	Anup Biswas and Erwin Topp.
	\newblock Lipschitz regularity of fractional $p$-laplacian, 2025.
	
	\bibitem{bogeleinHolderRegularitySigned2021}
	Verena B{\"o}gelein, Frank Duzaar, and Naian Liao.
	\newblock On the {{H\"older}} regularity of signed solutions to a doubly
	nonlinear equation.
	\newblock {\em Journal of Functional Analysis}, 281(9):109173, November 2021.
	
	\bibitem{bogeleinHolderRegularitySigned2023}
	Verena B\"{o}gelein, Frank Duzaar, Naian Liao, and Leah Sch\"{a}tzler.
	\newblock On the {H}\"{o}lder regularity of signed solutions to a doubly
	nonlinear equation. {P}art {II}.
	\newblock {\em Rev. Mat. Iberoam.}, 39(3):1005--1037, 2023.
	
	\bibitem{brascoHigherHolderRegularity2018}
	Lorenzo Brasco, Erik Lindgren, and Armin Schikorra.
	\newblock Higher {{H\"older}} regularity for the fractional p-{{Laplacian}} in
	the superquadratic case.
	\newblock {\em Advances in Mathematics}, 338:782--846, November 2018.
	
	\bibitem{brascoContinuitySolutionsNonlinear2021}
	Lorenzo Brasco, Erik Lindgren, and Martin Str{\"o}mqvist.
	\newblock Continuity of solutions to a nonlinear fractional diffusion equation.
	\newblock {\em Journal of Evolution Equations}, June 2021.
	
	\bibitem{BOGELEIN2025111078}
	Verena Bögelein, Frank Duzaar, Naian Liao, Giovanni {Molica Bisci}, and
	Raffaella Servadei.
	\newblock Regularity for the fractional p-laplace equation.
	\newblock {\em Journal of Functional Analysis}, 289(9):111078, 2025.
	
	\bibitem{caffarelliRegularityTheoryParabolic2011}
	Luis Caffarelli, Chi~Hin Chan, and Alexis Vasseur.
	\newblock Regularity theory for parabolic nonlinear integral operators.
	\newblock {\em Journal of the American Mathematical Society}, 24(3):849--869,
	July 2011.
	
	\bibitem{caffarelliDriftDiffusionEquations2010}
	Luis Caffarelli and Alexis Vasseur.
	\newblock Drift diffusion equations with fractional diffusion and the
	quasi-geostrophic equation.
	\newblock {\em Annals of Mathematics. Second Series}, 171(3):1903--1930, 2010.
	
	\bibitem{ya-zheLocalBehaviorSolutions1988}
	Ya~Zhe Chen and Emmanuele DiBenedetto.
	\newblock On the local behaviour of solutions of singular parabolic equations.
	\newblock {\em Arch. Rational Mech. Anal.}, 103(4):319--345, 1988.
	
	\bibitem{cozziRegularityResultsHarnack2017}
	Matteo Cozzi.
	\newblock Regularity results and {{Harnack}} inequalities for minimizers and
	solutions of nonlocal problems: {{A}} unified approach via fractional {{De
			Giorgi}} classes.
	\newblock {\em Journal of Functional Analysis}, 272(11):4762--4837, June 2017.
	
	\bibitem{cozziFractionalGiorgiClasses2019}
	Matteo Cozzi.
	\newblock Fractional {{De Giorgi Classes}} and {{Applications}} to {{Nonlocal
			Regularity Theory}}.
	\newblock In Serena Dipierro, editor, {\em Contemporary {{Research}} in
		{{Elliptic PDEs}} and {{Related Topics}}}, Springer {{INdAM Series}}, pages
	277--299. {Springer International Publishing}, {Cham}, 2019.
	
	\bibitem{dibenedettoLocalBehaviourSolutions1986}
	E.~Di~Benedetto.
	\newblock {On the local behaviour of solutions of degenerate parabolic
		equations with measurable coefficients}.
	\newblock {\em Annali della Scuola Normale Superiore di Pisa - Classe di
		Scienze}, 13(3):487--535, 1986.
	
	\bibitem{dicastroLocalBehaviorFractional2016}
	Agnese Di~Castro, Tuomo Kuusi, and Giampiero Palatucci.
	\newblock Local behaviour of fractional p-minimizers.
	\newblock {\em Annales de l'Institut Henri Poincar\'e C, Analyse non
		lin\'eaire}, 33(5):1279--1299, September 2016.
	
	\bibitem{dibenedettoDegenerateParabolicEquations1993}
	Emmanuele DiBenedetto.
	\newblock {\em Degenerate Parabolic Equations}.
	\newblock Universitext. {Springer-Verlag, New York}, 1993.
	
	\bibitem{dibenedettoHarnackEstimatesQuasilinear2008}
	Emmanuele DiBenedetto, Ugo Gianazza, and Vincenzo Vespri.
	\newblock Harnack estimates for quasi-linear degenerate parabolic differential
	equations.
	\newblock {\em Acta Mathematica}, 200(2):181--209, January 2008.
	
	\bibitem{dibenedettoHarnackInequalityDegenerate2012}
	Emmanuele DiBenedetto, Ugo Gianazza, and Vincenzo Vespri.
	\newblock {\em Harnack's Inequality for Degenerate and Singular Parabolic
		Equations}.
	\newblock Springer {{Monographs}} in {{Mathematics}}. {Springer, New York},
	2012.
	
	\bibitem{dingLocalBoundednessHolder2021}
	Mengyao Ding, Chao Zhang, and Shulin Zhou.
	\newblock Local boundedness and {{H\"older}} continuity for the parabolic
	fractional p-{{Laplace}} equations.
	\newblock {\em Calculus of Variations and Partial Differential Equations},
	60(1):38, January 2021.
	
	\bibitem{gianazzaNewProofHolder2010}
	Ugo Gianazza, Mikhail Surnachev, and Vincenzo Vespri.
	\newblock A new proof of the {{H\"older}} continuity of solutions to
	p-{{Laplace}} type parabolic equations.
	\newblock 3(3):263--278, July 2010.
	
	\bibitem{hwangHolderContinuityBounded2015a}
	Sukjung Hwang and Gary~M. Lieberman.
	\newblock Hölder continuity of bounded weak solutions to generalized parabolic
	$p$-laplacian equations {I}: degenerate case.
	\newblock {\em Electronic Journal of Differential Equations}, pages No. 287,
	32, 2015.
	
	\bibitem{hwangHolderContinuityBounded2015}
	Sukjung Hwang and Gary~M. Lieberman.
	\newblock Hölder continuity of bounded weak solutions to generalized parabolic
	$p$-laplacian equations {II}: singular case.
	\newblock {\em Electronic Journal of Differential Equations}, pages No. 288,
	24, 2015.
	
	\bibitem{ishiiClassIntegralEquations2010}
	Hitoshi Ishii and Gou Nakamura.
	\newblock A class of integral equations and approximation of p-{{Laplace}}
	equations.
	\newblock {\em Calculus of Variations and Partial Differential Equations},
	37(3):485--522, March 2010.
	
	\bibitem{kassmannPrioriEstimatesIntegrodifferential2009}
	Moritz Kassmann.
	\newblock A priori estimates for integro-differential operators with measurable
	kernels.
	\newblock {\em Calculus of Variations and Partial Differential Equations},
	34(1):1--21, January 2009.
	
	\bibitem{MR2957656}
	Tuomo Kuusi, Rojbin Laleoglu, Juhana Siljander, and Jos\'{e}~Miguel Urbano.
	\newblock H\"{o}lder continuity for {T}rudinger's equation in measure spaces.
	\newblock {\em Calc. Var. Partial Differential Equations}, 45(1-2):193--229,
	2012.
	
	\bibitem{MR3029403}
	Tuomo Kuusi, Juhana Siljander, and Jos\'{e}~Miguel Urbano.
	\newblock Local {H}\"{o}lder continuity for doubly nonlinear parabolic
	equations.
	\newblock {\em Indiana Univ. Math. J.}, 61(1):399--430, 2012.
	
	\bibitem{li2025holderregularityweaksolutions}
	Qifan Li.
	\newblock H\"older regularity of weak solutions to nonlocal doubly degenerate
	parabolic equations, 2025.
	
	\bibitem{liaoUnifiedApproachHolder2020}
	Naian Liao.
	\newblock A unified approach to the {{H\"older}} regularity of solutions to
	degenerate and singular parabolic equations.
	\newblock {\em Journal of Differential Equations}, 268(10):5704--5750, May
	2020.
	
	\bibitem{liaoHolderRegularityParabolic2024}
	Naian Liao.
	\newblock H\"{o}lder regularity for parabolic fractional {$p$}-{L}aplacian.
	\newblock {\em Calc. Var. Partial Differential Equations}, 63(1):Paper No. 22,
	34, 2024.
	
	\bibitem{liaoHolderRegularitySigned2022}
	Naian Liao and Leah Sch\"{a}tzler.
	\newblock On the {H}\"{o}lder regularity of signed solutions to a doubly
	nonlinear equation. {P}art {III}.
	\newblock {\em Int. Math. Res. Not. IMRN}, (3):2376--2400, 2022.
	
	\bibitem{silvestreHolderEstimatesSolutions2006}
	Luis Silvestre.
	\newblock H\"older {{Estimates}} for {{Solutions}} of {{Integro-differential
			Equations Like The Fractional Laplace}}.
	\newblock {\em Indiana University Mathematics Journal}, 55(3):1155--1174, 2006.
	
	\bibitem{tilliRemarksHolderContinuity2006}
	Paolo Tilli.
	\newblock Remarks on the {{H\"older}} continuity of solutions to elliptic
	equations in divergence form.
	\newblock {\em Calculus of Variations and Partial Differential Equations},
	25(3):395--401, March 2006.
	
\end{thebibliography}
%\end{singlespace}

\end{document}